\documentclass{amsart}[10pt]
\usepackage[english]{babel}
\usepackage[utf8]{inputenc}
\usepackage{amsthm,amsfonts,amssymb,amsmath}
\usepackage{stmaryrd}
\numberwithin{equation}{section}
\usepackage{dsfont}
\usepackage{multirow}
\usepackage{url}
\usepackage{algorithm2e} 
\usepackage{color} 
\usepackage{graphics}
\usepackage{booktabs}
\usepackage{epsfig}
\usepackage{epstopdf}
\usepackage{psfrag}               
\usepackage{mathrsfs}
\usepackage{mathtools}            
\usepackage{subfigure}
\usepackage{longtable}	

\usepackage{paralist}
\usepackage[inline]{enumitem}

\usepackage{hyperref}

\newcommand{\bbE}{\mathbb{E}}

\newcommand{\bbN}{\mathbb{N}}
\newcommand{\bbP}{\mathbb{P}}

\newcommand{\bbR}{\mathbb{R}}
\newcommand{\bbS}{\mathbb{S}}
\newcommand{\bbT}{\mathbb{T}}

\newcommand{\cA}{\mathcal{A}}
\newcommand{\cB}{\mathcal{B}}
\newcommand{\cC}{\mathcal{C}}

\newcommand{\cF}{\mathcal{F}}

\newcommand{\cJ}{\mathcal{J}}
\newcommand{\cK}{\mathcal{K}}
\newcommand{\cL}{\mathcal{L}}
\newcommand{\cM}{\mathcal{M}}

\newcommand{\cO}{\mathcal{O}}
\newcommand{\cP}{\mathcal{P}}

\newcommand{\cR}{\mathcal{R}}

\newcommand{\cW}{\mathcal{W}}
\newcommand{\cX}{\mathcal{X}}

\newcommand{\cZ}{\mathcal{Z}}

\newcommand{\bA}{{\mathbf A}}
\newcommand{\bB}{{\mathbf B}}

\newcommand{\bC}{{\mathbf C}}
\newcommand{\bP}{{\mathbf P}}

\newcommand{\vvec}{\mathbf{v}}
\newcommand{\wvec}{\mathbf{w}}
\newcommand{\xvec}{\mathbf{x}}
\newcommand{\yvec}{\mathbf{y}}
\newcommand{\zvec}{\mathbf{z}}
\newcommand{\fvec}{\mathbf{f}}
\newcommand{\xivec}{\boldsymbol{\xi}}

\newcommand{\zerovec}{\mathbf{0}} 

\newcommand{\bPhi}{\mathbf{\Phi}}
\newcommand{\bPsi}{\mathbf{\Psi}}
\newcommand{\bD}{\mathbf{D}}

\newcommand{\Amat}{\mathbf{A}}
\newcommand{\Bmat}{\mathbf{B}}
\newcommand{\Cmat}{\mathbf{C}}

\newcommand{\Imat}{\mathbf{I}}

\newcommand{\Mmat}{\mathbf{M}}

\newcommand{\Rmat}{\mathbf{R}}
\newcommand{\Smat}{\mathbf{S}}

\newcommand{\Sigmamat}{\boldsymbol{\Sigma}}

\newcommand{\T}{\top} 

\newcommand{\norm}[2]{     \| #1       \|_{ #2 }}

\newcommand{\normiii}[2]{\vert\kern-0.25ex\vert\kern-0.25ex\vert #1 \vert\kern-0.25ex \vert\kern-0.25ex\vert_{ #2 }}
\newcommand{\Normiii}[2]{\left\vert\kern-0.25ex\left\vert\kern-0.25ex\left\vert #1 \right\vert\kern-0.25ex\right\vert\kern-0.25ex\right\vert_{ #2 }}
\newcommand{\scalar}[2]{     ( #1       )_{ #2 }}

\newcommand{\white}{\cW}

\newcommand{\rd}{\mathrm{d}}
\newcommand{\from}{\colon}
\newcommand{\GP}{\cZ}
\newcommand{\eps}{\varepsilon}

\newcommand{\normal}{\mathsf{N}}

\newtheorem{lemma}{Lemma}[section]
\newtheorem{proposition}[lemma]{Proposition}
\newtheorem{theorem}[lemma]{Theorem}

\newtheorem{assumption}[lemma]{Assumption}

\theoremstyle{remark}
\newtheorem{remark}[lemma]{Remark}

\theoremstyle{definition}
\newtheorem{definition}[lemma]{Definition}
\newtheorem{example}[lemma]{Example}
\DeclareFontFamily{OT1}{pzc}{}
\DeclareFontShape{OT1}{pzc}{m}{it}{<-> s * [1.10] pzcmi7t}{}
\DeclareMathAlphabet{\mathpzc}{OT1}{pzc}{m}{it}

\newcommand{\ra}{{\hat{r}}}

\begin{document}

\title[Multilevel approximation of GRFs: compression, estimation, kriging]
	{Multilevel approximation of Gaussian random fields:
         Covariance compression, estimation \\ and spatial prediction}

\author[H.~Harbrecht, L.~Herrmann, K.~Kirchner, and Ch.~Schwab]{Helmut Harbrecht 
		\and Lukas Herrmann \and \\ 
        Kristin Kirchner \and Christoph Schwab}

\address[H.~Harbrecht]{Department of Mathematics and Computer Science\\
University of Basel\\
Spiegelgasse 1, 4051 Basel\\
Switzerland}
\email{helmut.harbrecht@unibas.ch}

\address[L.~Herrmann]{Johann Radon Institute for Computational and Applied Mathematics\\
Austrian Academy of Sciences\\
Altenbergerstrasse 69, 4040 Linz\\
Austria}
\email{lukas.herrmann@ricam.oeaw.ac.at}

\address[K.~Kirchner]{
	Delft Institute of Applied Mathematics \\
	Delft University of Technology \\
	P.O.~Box 5031, 2600 GA Delft \\
	The Netherlands}
\email{k.kirchner@tudelft.nl}

\address[Ch.~Schwab]{
	Seminar for Applied Mathematics \\
	ETH Z\"urich \\
	R\"amistrasse 101, CH-8092 Z\"urich \\ 
	Switzerland}
\email{christoph.schwab@sam.math.ethz.ch}


\thanks{Acknowledgment. 
LH, KK and CS acknowledge helpful discussions with
Sara van de Geer. 
This paper was conceived and written in large parts 
at SAM, D-MATH, ETH Z\"urich.}


\begin{abstract}
Centered Gaussian random fields (GRFs) 
indexed by 
compacta such as smooth, bounded domains in Euclidean space or
smooth, compact and orientable manifolds
are determined by their covariance operators.
We consider centered GRFs 
given sample-wise as variational solutions to 
\emph{coloring} operator 
equations driven by spatial white noise,
with pseudodifferential coloring operator being
elliptic, self-adjoint and positive from the H\"ormander class.
This includes the Mat\'ern class of GRFs as a special case.
Using microlocal tools and 
biorthogonal multiresolution analyses on the manifold,
we prove that
the precision and covariance operators, respectively,
may be identified with bi-infinite matrices and
finite sections may be diagonally preconditioned
rendering the condition number independent of 
the dimension $p$ of this section.
We prove that 
a tapering strategy by thresholding as e.g.\ in 
[Bickel, P.J.\ and Levina, E.
    Covariance regularization by thresholding,
    Ann. Statist., 36 (2008), 2577--2604]
applied on finite sections 
of the bi-infinite precision and covariance matrices
results in \emph{optimally numerically sparse} approximations.
Numerical sparsity signifies that
only asymptotically linearly many nonzero matrix entries 
are sufficient to approximate the original section 
of the bi-infinite covariance or precision matrix 
using this tapering strategy to arbitrary precision.
This tapering strategy is non-adaptive 
and the locations of these nonzero matrix entries 
are known a priori. 
The tapered covariance or precision matrices may
also be optimally diagonal preconditioned. 
Analysis of the relative size of the entries
of the tapered covariance matrices motivates 
novel, multilevel Monte Carlo (MLMC) oracles for covariance estimation,
in sample complexity that scales log-linearly 
with respect to the number $p$ of parameters.
This extends
[Bickel, P.J.\ and Levina, E. Regularized Estimation of Large Covariance
Matrices, Ann.\ Stat., 36 (2008), pp.\ 199--227]
to estimation of (finite sections of) 
pseudodifferential covariances for GRFs by this fast MLMC method.
Assuming at hand sections of the bi-infinite
covariance matrix in wavelet coordinates,
we propose and analyze a novel 
\emph{compressive algorithm for simulating and kriging of GRFs}.
The complexity (work and memory vs.\ accuracy)
of these three algorithms 
scales near-optimally in terms of the number of parameters $p$ 
of the sample-wise approximation of the GRF in Sobolev scales.
\end{abstract}

\keywords{Mat\'{e}rn covariance, multilevel Monte Carlo methods, kriging, wavelets.}

\subjclass[2010]{Primary: 
	62M20, 
	65C60; 
	secondary: 
	62M09, 
	65C05.} 

\date{}

\maketitle

\section{Introduction}
\label{section:intro}

\subsection{Background and problem formulation}
\label{sec:Backgnd}

Several methodologies in uncertainty quantification 
and data assimilation require 
the storage 
of the covariance matrix $\bC$ 
or the precision matrix $\bP=\bC^{-1}$ 
corresponding to an underlying statistical model 
as well as computations involving these matrices. 
Explicit examples include 
simulations,  
predictions and   
Bayesian or likelihood-based inference in spatial statistics. 
Here, one of the main computational challenges 
is to handle large datasets, 
as the covariance and precision matrices $\bC,\bP$ are, 
in general, densely populated and, for this reason, the 
computational cost for predictions or inference 
is cubic 
in the number of observations. 

A widely used class of statistical models 
is that of Gaussian processes, 
which are uniquely defined by their mean 
and covariance structure. 
These Gaussian processes 
may be indexed by subsets $\cX$ of $\bbR^n$, such as  
bounded Euclidean domains and
surfaces (or, more generally, manifolds), 
and also by graphs. 
In the former case, 
methods to cope with the computational 
challenges named above include  
low-rank approximations 
such as, e.g., fixed-rank kriging, 
predictive processes, 
and process convolutions 
\cite{banerjee2008gaussian, cressie2008fixed, higdon2002space}. 
Furthermore, approaches which 
reduce the computational cost 
by exploiting sparsity have been considered in the literature. 
More precisely, both sparse approximations of the 
covariance matrix 
$\bC_{ij} = \bbE[\GP(x_i) \GP(x_j)]$ 
(aka.\ \emph{covariance tapering} \cite{furrer2006covariance}) 
and of the precision matrix \cite{datta2016hierarchical} 
for a random field $\GP$ 
have been proposed 
and used for statistical applications. 
Alternatively, 
one can approximate  
the random field $\GP$ 
by a finite dimensional basis expansion, 
\begin{equation}\label{eq:intro:expansion}  
\GP(x) 
= 
\sum\limits_{j=1}^p z_j \varphi_j(x), 
\qquad 
x\in\cX. 
\end{equation}
Here, it is the choice of the basis functions $\{\varphi_j\}_{j=1}^p$ 
that will determine the sparsity pattern 
of the covariance and precision matrices of the stochastic weights, 
$\bC_{ij} = \bbE[z_i z_j]$ 
as well as the 
corresponding computational cost. 
For instance, 
in the stochastic partial differential equation (SPDE) approach
as proposed in \cite{Lindgren2011}, 
the Gaussian random field (GRF) $\GP$ on $\cX\subset\bbR^n$ 
is modeled as the solution of 
a white noise driven SPDE and its precision operator 
is, in general, a fractional power of an elliptic second-order 
differential operator. In the case that this power 
is an integer, the precision operator is local, 
which facilitates sparsity of $\bP$ if 
the functions $\{\varphi_j\}_{j=1}^p$ in \eqref{eq:intro:expansion} are chosen, 
e.g., as a finite element basis. 
In the general (fractional-order) case, the covariance and precision operators 
for the SPDE approach are non-local and more 
sophisticated methods have to be 
exploited for computational efficiency~\cite{Bolin2020,HKS1}. 
Note also that in the case that $\cX$ is a manifold 
the fractional-order covariance and precision operators can be seen 
as pseudodifferential operators. 
As an alternative to the finite element method,  
multiresolution approximations of the process 
have been suggested, where the basis functions 
$\{\varphi_j\}_{j=1}^p$ in \eqref{eq:intro:expansion}  
originate from a multiresolution analysis (MRA), 
see \cite{katzfuss2017multi, nychka2015multiresolution}. 
This approach seems to perform well 
(see also the comparison in \cite{HeatonEtal2019}); 
however, to the best of our knowledge 
no error bounds for these approximations 
have been derived and, therefore, 
they need to be adjusted 
for each specific model. 

In the context of graph-based data, 
significant attention has been directed in recent years
at computational and statistical modeling in high dimensional settings, 
see e.g.\ \cite{JJSvdG,UhlerGGrMod}.
Here, Gaussian random fields play an important role, 
where the precision operator is a (regularized)
discrete, fractional graph Laplacian.
It is known that 
for large data, i.e., 
in the (high-dimensional) 
large graph limit, 
the graph Laplacian converges 
to a (pseudo)differential operator $\cP$, 
see \cite{StuartGrphLim2020}. 

In the infinite-dimensional setting,
for a compact Riemannian manifold $\cX=\cM$, we consider 
GRFs $\GP$ obtained by ``coloring'' white noise on 
the Hilbert space $L^2(\cM)$ with 
the compact inverse of a pseudodifferential operator $\cA$ 
that is a positive, self-adjoint unbounded operator on $L^2(\cM)$. 
Then, the corresponding covariance and precision operators $\cC$ and $\cP$  
are pseudodifferential operators, and we prove that
$\cC=\cA^{-2}$ and $\cP=\cA^2$.
The connection of this setting to the above, 
is facilitated through \emph{biorthogonal Riesz bases (wavelet bases)} 
$\bPsi$ and $\widetilde{\bPsi}$ of $L^2(\cM)$, which give rise to  
\emph{equivalent, bi-infinite matrix representations $\bC, \bP \in \bbR^{\bbN\times \bbN}$} 
of $\cC$ and $\cP$. 
Finite sections of these bi-infinite matrices with $p$ parameters, 
i.e., $\bC\approx\bC_p \in\bbR^{p\times p}$ 
and $\bP\approx\bP_p \in\bbR^{p\times p}$, 
correspond to approximate representations of the GRF 
as in \eqref{eq:intro:expansion}, where 
the basis functions $\{\varphi_j \}_{j=1}^p$ 
are those functions of the wavelet basis $\bPsi$ 
corresponding to the finite set of indices used 
to generate $\bC_p, \bP_p$.

\subsection{Contributions}
\label{sec:Contr}

In this work we establish 
\emph{optimal numerical sparsity and optimal preconditioning} of both, 
the precision operator $\cP$ and the covariance operator $\cC$ 
when represented in the wavelet bases $\bPsi$. 
Specifically, our compression analysis reveals 
\emph{universal a-priori tapering patterns}  
for finite sections $\bC_p,\bP_p\in\bbR^{p\times p}$ 
of both, the possibly bi-infinite 
covariance and the precision matrices 
$\bC = \cC(\bPsi)(\bPsi)$, $\bP = \cP(\bPsi)(\bPsi) \in \bbR^{\bbN\times \bbN}$. 
We prove that, in the above general setting,
the number of nonvanishing coefficients in the 
numerically tapered matrices 
$\bC^\eps_p, \bP^\eps_p \in\bbR^{p\times p}$ scales linearly with $p$ 
at a certified accuracy $\eps>0$ compared to $\bC_p, \bP_p$. 
In addition, we prove \emph{diagonal preconditioning} renders
the condition numbers of the family of $p$-sections
$\{\bP_p\}_{p\geq 1}$, $\{\bC_p\}_{p\geq 1}$ of $\bP$ and $\bC$, 
uniformly bounded with respect to $p\in \bbN$.

The sparsity bounds for these wavelet matrix representations 
are closely related to corresponding
compression estimates for wavelet representation of 
elliptic pseudodifferential operators \cite{DHSWaveletBEM2007}.
Our setting accommodates 
elliptic, self-adjoint pseudodifferential 
coloring operators $\cA$ including, in particular, 
the Mat\'{e}rn class of GRFs on compact manifolds,
but extending substantially beyond these. 
In particular, stationarity of $\GP$ is not required. 

These results on sparsity and preconditioning 
of $\bC_p, \bP_p$ give rise to several applications
which are developed in Section~\ref{section:application}.  
Firstly, 
in Section \ref{subsec:simulation},
we consider the efficient numerical simulation 
of the GRF $\GP$ by combining our results on 
sparsity and preconditioning of the approximate covariance 
matrix $\bC^\eps_p$ with an algorithm to compute the matrix square root 
based on a contour integral \cite{Hale2008}. 
We furthermore propose and analyze 
a wavelet-based numerical covariance estimation algorithm 
for the $p\times p$ section $\bC_p$ of the covariance matrix $\bC$.
The proposed method is of multilevel Monte Carlo type: 
Given i.i.d.\ 
realizations of the GRF~$\GP$ in wavelet coordinates
with different, sample-dependent spatial resolution, 
our multilevel, wavelet-based sampling strategy 
resulting in an approximate covariance matrix 
$\widetilde{\bC}_p\in\bbR^{p\times p}$ 
will require essentially $\cO(p)$ data, memory and work.
As a final application, 
we consider 
spatial prediction (aka.\ kriging) for the GRF $\GP$ 
in Section \ref{subsec:krging}.  
Assuming at hand an approximate covariance matrix 
$\widetilde{\bC}_p$ in a wavelet-based 
multiresolution representation, 
we prove 
(cf.\ Remark \ref{rmk:OK+p}) that  
approximate kriging, consistent to the order of spatial resolution
and subject to $K$ noisy observation functionals,
can be achieved in $\cO(K+p)$ work and memory.

\subsection{Outline}
\label{sec:Outline}

This paper is structured as follows.
Section~\ref{section:GRFs-manifolds} introduces the abstract setting
of GRFs on smooth, compact manifolds, pseudodifferential
coloring operators and the corresponding estimates for the Schwartz 
kernels of these operators.
Section~\ref{section:results} recapitulates key technical results
on wavelet compression of pseudodifferential operators on manifolds,
with particular attention to numerical compression and 
multilevel preconditioning of covariance and precision matrices 
resulting as finite sections of the 
\emph{equivalent, bi-infinite matrix representations of the covariance and
	precision operators}. 
Section~\ref{section:application} presents several major applications
of the proposed wavelet compression framework for computational
simulation. 
Specifically, 
Section~\ref{subsec:simulation} discusses
a functional-integral based algorithm 
of essentially linear $\cO(p)$ work and memory 
for approximating the square root of the covariance matrix.
Section~\ref{sec:MLMCCovEst} presents a 
multilevel covariance estimation algorithm 
from i.i.d.\ samples of a GRF, of essentially $\cO(p)$ complexity,
and 
Section~\ref{subsec:krging} a novel, sparse kriging 
algorithm for GRFs resulting from 
pseudodifferential coloring of white~noise.
Section~\ref{sec:numexp}
then presents a suite of numerical experiments for the 
simulation and estimation of GRFs on manifolds 
of dimension $n=1$ and $n=2$. 
We also comment on the use of the Cholesky decomposition
in connection with wavelet coordinates to achieve efficient
numerical simulation.
Section~\ref{sec:Concl} summarizes the main results, 
and indicates further applications and extensions of the 
sparsity and preconditioning results. 

Finally, this work contains four appendices: 
Appendix~\ref{appendix:PDO_review} briefly recapitulates 
the H\"ormander calculus of pseudodifferential operators on manifolds, 
Appendix~\ref{appendix:wavelets} reviews construction and properties
of MRAs on smooth, compact manifolds, 
Appendix~\ref{appendix:ColShft} presents 
(Whittle--)Mat\'ern covariance models \cite{Lindgren2011,matern1960} 
as particular instances of the general theory, 
and Appendix~\ref{appendix:sample_numbers} 
provides the justification for the work--accuracy 
relation for the multilevel Monte Carlo algorithm 
in Section~\ref{sec:MLMCCovEst}. 

\subsection{Notation}
\label{sec:Notat}

For an open domain $G\subset \bbR^n$, 
the support of a real-valued function 
$\phi\from G\to \bbR$ is denoted by 
$\operatorname{supp}(\phi):=\overline{\{x \in G : \phi(x) \neq 0\}}$, 
where the closure is taken in the ambient space~$\bbR^n$. 
If for some subset $G'\subset G\subset\bbR^n$, there exists a compact set $G''$
such that $G'\subset G'' \subset G$, we say that $G'$ is compactly included in $G$
and write $G'\subset\subset G$. 
The space of all smooth real-valued functions in $G$ 
is given by $C^\infty(G)$, and 
$C^\infty_0(G)\subset C^\infty(G)$ is the subspace of all  
smooth functions $\phi$ with $\operatorname{supp}(\phi)\subset\subset G$. 
For a smoothness order $s \in [0,\infty)$, $H^s(G)$ is the 
Sobolev--Slobodeckij space. 

For a smooth, compact Riemannian manifold $\cM$ 
The geodesic distance on $\cM$ 
will be denoted by  $\operatorname{dist}(\,\cdot\,, \,\cdot\,)$. 
For any $q\in [1,\infty)$, $s\in [0,\infty)$,
the function spaces $L^q(\cM)$ and $H^s(\cM)$
denote the $q$-integrable functions with respect to the intrinsic measure on $\cM$
and the Sobolev--Slobodeckij spaces, respectively. 
We write $\langle \,\cdot\,, \,\cdot\, \rangle$
for the duality pairing with respect to the spaces $H^s(\cM)$, 
where we shall not explicitly include the dependence on $s$. 
For (pseudodifferential) operators on function spaces on $\cM$, 
we shall use calligraphic symbols. 
Particular such pseudodifferential operators are 
the coloring operator $\cA$, as well as
the covariance and precision operators $\cC$, $\cP$.
A generic (pseudodifferential) operator shall often 
be denoted by~$\cB$.

For a vector $\vvec\in\bbR^n$ or a square-summable sequence 
$\vvec\in\ell^2(\cJ)$ 
indexed by a countable set $\cJ$, 
we define $\| \vvec \|_2 := \sqrt{\sum_j v_j^2}$ . 
We shall also use the same notation for 
the operator norm induced by 
$\| \,\cdot\, \|_2$ (note that 
for $\bbR^n$ this defines a matrix norm on $\bbR^{n\times n}$). 
The spectrum and condition number of a matrix $\Amat$  
or an operator $\Amat$ on $\ell^2(\cJ)$ 
with respect to the norm $\| \,\cdot\, \|_2$ 
is denoted by $\sigma(\Amat)$ 
and $\operatorname{cond}_2(\Amat)$, respectively .
In addition, 
$\|\Amat\|_{\rm HS}$ denotes the Frobenius norm 
if $\Amat\in\bbR^{n\times n}$ and the Hilbert--Schmidt norm 
in the more general case that $\Amat\from\ell^2(\cJ)\to\ell^2(\cJ)$. 

For any two sequences $(a_k)_{k\in\bbN}$ and $(b_k)_{k\in\bbN}$, we write 
$a_k\lesssim b_k$, if there exists a constant $C>0$  
independent of $k$, such that $a_k\leq C b_k$ for all $k$. 
Analogously, we write $b_k \gtrsim a_k$, and $a_k \simeq b_k$ 
whenever both relations hold, $a_k \lesssim b_k$
and $a_k \gtrsim b_k$.

Throughout this manuscript, we let 
$(\Omega,\cF,\bbP)$ be a complete probability space 
with expectation operator $\bbE[\,\cdot\,]$. 
For two random vectors or random sequences 
$\vvec,\wvec$ on $(\Omega,\cF,\bbP)$, the notation 
$\vvec \overset{d}{=} \wvec$ 
indicates that $\vvec$ and $\wvec$ have identical distribution, 
and $\vvec\sim\normal(\mathbf{m},\bC)$ 
denotes a Gaussian distribution 
with mean $\mathbf{m}$ and covariance~$\bC$.  

\section{Gaussian random fields on manifolds}
\label{section:GRFs-manifolds}

We first give a concise presentation 
of the Gaussian random fields (GRFs)
of interest and of the basic setup. 
A GRF $\GP$ considered in this work 
on the probability space $(\Omega,\cF,\bbP)$, 
is centered and indexed by a 
compact Riemannian manifold $\cM$ of dimension $n\in\bbN$. 
Specifically we assume, 
$(\GP(x))_{x\in\cM}$ is a family 
of $\cF$-measurable $\bbR$-valued random variables 
such that for all finite sets $\{x_1,, \ldots, x_m\}\subset\cM$ 
the random vector $(\GP(x_1), \ldots,\GP(x_m))^\top$ 
is centered  Gaussian,
and such that 
the mapping $\GP\from\cM\times\Omega\to\bbR$ 
is $\mathscr{B}(\cM)\otimes\cF$-measurable. 
Here, $\mathscr{B}(\cM)$ denotes the Borel $\sigma$-algebra 
generated by a topology on $\cM$ 
with respect to a distance~$\operatorname{dist}(\,\cdot\,, \,\cdot\,):\cM\times \cM\to \bbR$ 
which may be chosen, e.g., as the geodesic distance on $\cM$. 
In this case, the covariance kernel  
$k\from \cM\times\cM\to\bbR$, $k(x,x'):=\bbE[\GP(x)\GP(x')]$ 
is a symmetric and positive definite function. 
Furthermore, we suppose that 
$(\cM, \mathscr{B}(\cM))$ is equipped with the surface measure $\mu$ 
induced by the first fundamental form, 
see \cite[Def.~1.73]{Aubin1998Riemannian} 
for a definition, 
see also Subsection~\ref{sec:manifolds:SobSpc} in Appendix~\ref{appendix:PDO_review}.  
The precise assumptions on the manifold are 
spelled out in Assumption~\ref{ass:cM-and-cA}\ref{ass:cM-and-cA_I} below. 
For a recap on notation and definitions pertaining to smooth manifolds, 
the reader is referred to 
Appendix~\ref{sec:DiffGeo}.

Specifically, we consider a GRF $\GP$ generated by a linear 
\emph{coloring} (elliptic pseudodifferential) \emph{operator} 
$\cA \in OPS^{\ra}_{1,0}(\cM)$ of order $\ra > n/2$ 
via the white noise driven 
stochastic (pseudo) differential equation (SPDE)
\begin{equation}\label{eq:WhNois}
\cA \GP = \cW \quad \text{on} \quad \cM.  
\end{equation}
Here and throughout,  
$\cW$ denotes white noise on the Hilbertian Lebesgue space 
$L^2(\cM)$, 
i.e., it is an $L^2(\cM)$-valued weak random variable, 
cf.~\cite[Chap.~6.4]{Balakrishnan1981},
with characteristic function  
$L^2(\cM) \ni \phi 
\mapsto \bbE[ \exp( i\scalar{\phi, \cW}{L^2(\cM)}) ]  
= 
\exp\bigl( -\tfrac{1}{2}\norm{\phi}{L^2(\cM)}^2 \bigr)$. 
Due to $\ra > n/2$, the  
Kolmogorov--Chentsov continuity theorem 
ensures that there exists a modification of $\GP$ in \eqref{eq:WhNois} 
whose realizations are continuous on $\cM$, $\bbP$-a.s. 
For some of our arguments, 
we assume that $\{ \widetilde{\gamma}_i \}_{i=1}^{M}$ 
is a smooth atlas of $\cM$ so that
$\widetilde{\gamma}_i\from G\to\widetilde{\gamma}_i(G)=:\widetilde{\cM}_i \subset \cM$ 
is diffeomorphic for some open set $G\subset \bbR^n$
and $\cM = \bigcup_{i=1}^M \widetilde{\cM}_i$. 
Furthermore, we 
let $\{ \chi_i \}_{i=1}^{M}$ be a smooth partition of unity 
corresponding to the atlas $\{ \widetilde{\gamma}_i \}_{i=1}^{M}$, i.e., 
for all $i=1,\ldots,M$, 
the function 
$\chi_i \from \cM \to [0,1]$ is smooth and 
compactly supported in $\widetilde{\cM}_i$, 
and 
$\sum_{i=1}^M \chi_i = 1$. 
For every $r\in\bbR$, 
the operator class $OPS^r_{1,0}(\cM)$ is then defined through local coordinates 
and we will therefore first introduce the class $OPS^r_{1,0}(G)$ 
for an open set $G\subset \bbR^n$. 
To this end, suppose that the \emph{symbol} $b\in C^\infty(G\times\bbR^n)$
satisfies that,
for every compact set $K\subset \subset G$
and for any $\alpha,\beta \in \bbN_0^n$, 
there exists a constant $C_{K,\alpha,\beta}>0$
such that
\begin{equation}\label{eq:est_symbol}
\forall x\in K, 
\;\; 
\forall\xi\in \bbR^n: 
\quad
\bigl| D^\beta_x D^\alpha_\xi b(x,\xi) \bigr| 
\leq 
C_{K,\alpha,\beta} (1+|\xi|)^{r-|\alpha|}. 
\end{equation}
The class of
pseudodifferential operators $OPS^{r}_{1,0}(G)$ consists 
then of maps 
\begin{equation}\label{eq:OPS-symbol}
B \from 
C^\infty_0(G)\to C^\infty(G) , 
\quad 
(Bf)(x)
:= 
\int_{\xi\in \bbR^n} b(x,\xi) \hat{f}(\xi) \exp(ix\cdot\xi) \, \rd\xi,
\end{equation}
where we note that the Fourier transform $\hat{f}$ 
of $f$ is well defined, since 
$f$ has compact support in $G$.
For the manifold $\cM$, 
the operator class $OPS^r_{1,0}(\cM)$ results by localization 
using coordinate charts of $\cM$, 
i.e., an operator $\cB\from C^\infty(\cM) \to C^\infty(\cM)$ 
belongs to $OPS^r_{1,0}(\cM)$ 
if all of 
the \emph{transported operators} do, 
i.e., $B_{i,i'} \in OPS^{r}_{1,0}(G)$ 
for all $i,i' = 1,\ldots,M$, where
\begin{equation}\label{eq:OPS_transported}
B_{i,i'}f :=
\bigl[ \bigl( \cB [ \chi_i ( f\circ\widetilde{\gamma}_i^{-1} ) ] \bigr)\chi_{i'} \bigr]
\circ\widetilde{\gamma}_{i'}, 
\quad 
f\in C_0^\infty(G), 
\quad 
i,i'=1,\ldots,M
.
\end{equation}
We refer to Section~\ref{sec:PDO} in Appendix~\ref{appendix:PDO_review}
for further details. 
There, also elements of  
the \emph{H\"ormander pseudodifferential operator calculus} 
of $OPS^r_{1,0}(\cM)$ are reviewed in Section~\ref{sec:PDOCalc}.
The Laplace--Beltrami operator on $\cM$ is denoted by $\Delta_{\cM}$. 
It is a second-order, elliptic differential operator on $\cM$
(e.g., \cite[Chap.~4]{Aubin1998Riemannian}) 
and, therefore, 
an element of $OPS^2_{1,0}(\cM)$.
For $s>0$, the Hilbertian Sobolev space 
$H^s(\cM)$ may thus be defined by 
(see, e.g., \cite[Chap.~2]{Aubin1998Riemannian})
\begin{equation}\label{eq:def:Sobolev_sp}
H^s(\cM)
:= 
(1-\Delta_{\cM})^{-s/2} \left( L^2(\cM) \right), 
\quad 
\| v \|_{H^s(\cM)} 
:= 
\bigl\| (1-\Delta_{\cM})^{s/2} v  \bigr\|_{L^2(\cM)}
, 
\end{equation}
see also Subsection~\ref{sec:manifolds:SobSpc}. 
For $s>0$, $H^{-s}(\cM)$ 
denotes the dual space of $H^s(\cM)$ 
(with respect to the identification $L^2(\cM) \cong L^2(\cM)^*$).

Existence and uniqueness of 
a solution $\GP$ to the SPDE \eqref{eq:WhNois} 
are ensured if the manifold~$\cM$ and the coloring operator 
$\cA$ in \eqref{eq:WhNois} satisfy 
certain regularity and positivity assumptions. 
These conditions are summarized below. 
\begin{assumption}\label{ass:cM-and-cA}
	\
	
	\begin{enumerate}[label={\normalfont{(\Roman*)}}, leftmargin=0.75cm]
		\item \label{ass:cM-and-cA_I}
		The manifold $\cM$ is a smooth, closed, bounded and connected  
		orientable Riemannian manifold of dimension $n$ 
		immersed into Euclidean space
		${\mathbb R}^D$ for some $D > n$.  
		In particular, $\cM$ has no boundary $\partial \cM = \emptyset$.
		\item \label{ass:cM-and-cA_II}
		The operator $\cA\in OPS^{\ra}_{1,0}(\cM)$ for some $\ra > n/2$ 
		is self-adjoint and positive 
		in the sense that there exists a constant $a_- > 0$ such that
		\[  
		\forall w \in H^{\ra/2}(\cM): 
		\quad 
		\langle \cA w, w \rangle \geq a_-\|w\|_{H^{\ra/2}(\cM)}^2.
		\]
	\end{enumerate} 
\end{assumption}

Under Assumptions~\ref{ass:cM-and-cA}\ref{ass:cM-and-cA_I}--\ref{ass:cM-and-cA_II}, 
the operator $\cA\in OPS^{\ra}_{1,0}(\cM)$ 
is a bijective, continuous mapping from 
$H^{s}(\cM)$ to $H^{s-\ra}(\cM)$ for any $s\in\bbR$ 
(see Proposition~\ref{prop:OPS_algebra_prop}\ref{item:range})
and, therefore, 
$\GP$ in \eqref{eq:WhNois} is well-defined. 
Moreover, the mapping properties of $\cA$ 
imply regularity of the GRF $\GP$: 
Since $\cW \in H^{-n/2-\varepsilon}(\cM)$ ($\bbP$-a.s.)
for any $\varepsilon > 0$, 
\begin{equation}\label{eq:ZPathReg}
\GP \in H^{s}(\cM),
\quad \text{for every} \quad 
s <  \ra -n/2 
\quad 
\text{($\bbP$-a.s.)},
\end{equation}
and, for any integrability $q\in (0,\infty)$, $0\leq s < \ra -n/2$,
\begin{equation}\label{eq:ZPathRegp}
\bbE\bigl[ \| \GP \|_{H^s(\cM)}^q \bigr] < \infty .
\end{equation} 
This follows, e.g., as in \cite[Lem.~3]{CK2020} and \cite[Lem.~2.2]{HKS1} 
using the asymptotic behavior  
of the eigenvalues of $\cA\in OPS_{1,0}^{\ra}(\cM)$ (Weyl's law).
\begin{example}\label{ex:whittle-matern} 
	In models of \emph{Whittle--Mat\'ern type} 
	(see also Appendix~\ref{appendix:ColShft}), 
	the pseudodifferential 
	operator~$\cA$ in \eqref{eq:WhNois}
	takes the form 
	$\cA = (\cL +\kappa^2)^\beta$ 
	with base (pseudo)differential
	operator $\cL\in OPS^{\bar{r}}_{1,0}(\cM)$ 
	for some $\beta,\bar{r}>0$.   
	In particular, $\kappa\in C^{\infty}(\cM)$ 
	determines the local correlation scale of the GRF~$\GP$.
	For any $\phi\in C^\infty(\cM)$, the multiplier with $\phi$, 
	i.e., the operator $f\mapsto\phi f$, 
	is an element of  
	$OPS^{0}_{1,0}(\cM)$.  
	For this reason, $\cA\in  OPS^{\ra}_{1,0}(\cM)$ with $\ra = \beta\bar{r} > 0$, 
	see Propositions~\ref{prop:OPS_algebra_prop} and~\ref{prop:OPS_real_powers}
	in Appendix~\ref{appendix:PDO_review}. 
	Explicit examples include  
	the SPDE-based extensions of GRFs 
	with  
	Mat\'ern covariance structure \cite{matern1960} 
	to the torus $\cM=\bbT^n$ or the sphere $\cM=\bbS^n$, 
	where $\cL = -\Delta_\cM$, 
	$\bar{r}=2$,  and $\kappa >0$ is constant,  
	see e.g.~\cite{Lindgren2011}. 
\end{example} 

The covariance operator 
$\cC\from L^2(\cM) \to L^2(\cM)$ 
of the GRF $\GP$ in~\eqref{eq:WhNois} 
is defined through the relation  
\[ 
(\cC v, w)_{L^2(\cM)} 
= 
\bbE\bigl[(\GP, v)_{L^2(\cM)} (\GP, w)_{L^2(\cM)} \bigr] 
\quad 
\forall v,w\in L^2(\cM). 
\]
If it exists, we define
the precision operator $\cP:=\cC^{-1}$ corresponding to $\GP$. 
The operators $\cC, \cP$ inherit several properties 
from the coloring operator $\cA$ in \eqref{eq:WhNois}.
\begin{proposition}\label{prop:C-and-P-ops} 
	Let $\ra > n/2$ and suppose that $\cM$ and $\cA \in OPS^{\ra}_{1,0}(\cM)$
	satisfy Assumptions~\ref{ass:cM-and-cA}\ref{ass:cM-and-cA_I}--\ref{ass:cM-and-cA_II}. 
	The covariance operator $\cC$ of the GRF $\GP$
	in \eqref{eq:WhNois} is then 
	\begin{equation}\label{eq:def:covariance}
	\cC= \cA^{-2} 
	\in 
	OPS^{-2\ra}_{1,0}(\cM)
	\end{equation}
	and, for every $s\in\bbR$, 
	$\cC \from H^{s}(\cM) \to H^{s+2\ra}(\cM)$ is an isomorphism. 
	Furthermore, under these assumptions, 
	the covariance operator $\cC$ in \eqref{eq:def:covariance} is   
	self-adjoint, (strictly) positive definite and compact  
	on $L^2(\cM)$, with a finite trace.
	
	Vice versa, 
	the precision operator $\cP$ of the GRF $\GP$ 
	in \eqref{eq:WhNois} is 
	\begin{equation}\label{eq:def:precision}
	\cP = \cA^2 
	\in 
	OPS^{2\ra}_{1,0}(\cM) 
	\end{equation}
	and, for every $s\in \bbR$
	it is an isomorphism as a mapping 
	$\cP \from  H^{s}(\cM) \to H^{s-2\ra}(\cM)$. 
	The precision operator $\cP  = \cA^2$ is a  
	self-adjoint, positive definite, 
	unbounded operator on $L^2(\cM)$, 
	whose spectrum  is discrete and accumulates only at $\infty$. 
\end{proposition}

\begin{proof}
	Assumption \ref{ass:cM-and-cA} implies that $\cA \in OPS^{\ra}_{1,0}(\cM)$ 
	is boundedly invertible. 
	For this reason, by 
	Proposition~\ref{prop:OPS_real_powers}, 
	the covariance operator 
	$\cC = \cA^{-2} \in OPS^{-2\ra}_{1,0}(\cM)$ is well-defined,
	self-adjoint, and positive definite. 
	By Proposition~\ref{prop:OPS_algebra_prop}\ref{item:range}, 
	continuity of 
	$\cC\from H^{s}(\cM)\to H^{s+2\ra}(\cM)$ 
	for all $s\in \bbR$ follows. 
	In particular, the choice $s=0$ 
	shows that $\cC \from L^2(\cM) \to L^2(\cM)$ is compact
	due to the compactness of the embedding 
	$H^{2r}(\cM) \subset L^2(\cM)$ for any $r>0$ which, in turn, 
	is a consequence of the assumed compactness of $\cM$ and of Rellich's theorem.
	
	To verify that $\cC$ has a finite trace on $L^2(\cM)$, 
	we let $\{\lambda_j(\cC)\}_{j\in\bbN}$ and 
	$\{\lambda_j(\cA)\}_{j\in\bbN}$  
	denote the eigenvalues 
	of $\cC$ and $\cA$, respectively, and we note that 
	self-adjointness of $\cA$ (stipulated in Assumption \ref{ass:cM-and-cA}) 
	and the spectral mapping theorem
	imply, for all $j\in \mathbb{N}$, the asymptotic behavior
	$\lambda_j(\cC) = \lambda_j(\cA)^{-2} \simeq j^{-2\ra/n}$ 
	for $j\to \infty$.
	Since $\{ j^{-2\ra/n} \}_{j\in\bbN} \in \ell^1(\bbN)$ 
	if and only if $2\ra/n > 1$, the claim follows.
	
	The assertions for $\cP$ can be shown along the same lines by 
	using that the eigenvalues of $\cP$ 
	are given by $\lambda_j(\cP) = 1/\lambda_j(\cC)$.
\end{proof}

An important relation between GRFs obtained by 
``pseudodifferential coloring'' 
of white noise as in \eqref{eq:WhNois}, 
their covariance operators, and their covariance kernels  
is established in the classical \emph{Schwartz kernel theorem}, 
see e.g.~\cite[Thm.~5.2.1]{HorI}.
Every continuous function 
$k \in  C(X_1 \times X_2)$
on the Cartesian product of
two open, nonempty sets 
$X_1, X_2 \subset \bbR^n$
defines an integral operator 
$\cK \from C(X_2)\to C(X_1)$ 
via 
\[
(\cK\phi)(x_1) = \int_{{X_2}} k(x_1, x_2) \phi(x_2) \, \rd x_2 
\quad 
\forall x_1 \in X_1. 
\]
This definition may be extended to the case that 
$k$ is a \emph{generalized function}
and~$\phi$ is smooth and compactly supported,
cf.\ \cite[Eq.~(5.2.1)]{HorI}.
Suppose now that $G\subset\bbR^n$ is open and consider a generic 
$B \in OPS^r_{1,0}(G)$.
By the Schwartz kernel theorem, cf.~\cite[Thm.~5.2.1]{HorI}, 
the pseudodifferential operator $B$ 
admits a distributional Schwartz kernel $k_B$. 
We (formally)\footnote{The derivation is rigorous, when 
	understood ``in the sense of distributions''.} 
calculate for $u,v\in C^\infty_0(G)$ with integrals understood as 
oscillatory integrals 
\begin{align*} 
\langle k_B , u\otimes v\rangle 
&\textstyle 
= \int_{\operatorname{supp}(v)} v(x) b(x,D) u(x) \, \rd x 
\\
&\textstyle 
=
\int_{{\operatorname{supp}(v)}} 
\int_{{\bbR^n}} b(x,\xi) \exp(i x\cdot \xi) v(x) \hat{u}(\xi) \, \rd\xi \, \rd x 
\\
&\textstyle 
= 
(2\pi)^{-n} 
\int_{\operatorname{supp}(v)} 
\int_{\bbR^n} 
\int_{\operatorname{supp}(u)}
b(x,\xi) \exp(i(x-y)\cdot \xi) v(x) u(y) \, \rd y \, \rd\xi \, \rd x .
\end{align*}
In the sense of distributions we thus obtain 
\[
\textstyle 
k_B (x,x-y) = (2\pi)^{-n} \int_{\bbR^n} b(x,\xi) \exp(i(x-y)\cdot \xi) \, \rd \xi ,
\]
so that for $w\in \bbR^n$ and for $\alpha \in \bbN_0^n$
\begin{equation}\label{eq:KDeriv}
\textstyle 
w^\alpha k_B(x,w) 
= 
(2\pi)^{-n} 
\int_{\bbR^n} \exp(iw\cdot \xi) D^\alpha_\xi b(x,\xi) \, \rd\xi
\end{equation}
with
$b(x,\xi) = \int_{G} \exp(-i w\cdot\xi) k_B (x,w) \, \rd w$
and $w^\alpha := \prod_{i=1}^n w_i^{\alpha_i}$. 
Since $b(x,\xi)$ satisfies~\eqref{eq:est_symbol},
the integral in~\eqref{eq:KDeriv} is absolutely convergent for 
$r-|\alpha| < -n$, i.e., $|\alpha| > n+r$. 
On the compact manifold $\cM$, a corresponding result holds 
by repeating the preceding calculation 
in coordinate charts 
$\{\widetilde{\gamma}_i\}_{i=1}^M$ of (a finite atlas of)~$\cM$.

\begin{proposition}\label{prop:CZEst}
	Let $\cB\in OPS^r_{1,0}(\cM)$ with 
	corresponding Schwartz kernel $k_\cB$.  
	In addition, for $i,i'=1,\ldots,M$, 
	let $B_{i,i'}\in OPS^{r}_{1,0}(G)$ 
	be defined according to~\eqref{eq:OPS_transported}, 
	and denote the corresponding Schwartz kernel by $k_{B_{i,i'}}$.
	
	Then, for every $\alpha, \beta \in \bbN^n_0$ with 
	$n+r+|\alpha|+|\beta| > 0$, 
	there exist constants $c_{\alpha,\beta}>0$ such that, 
	for all $i,i'=1,\ldots, M$,  
	\begin{equation}\label{eq:KAest}
	\forall x^*,y^*\in 
	\widetilde{\cM}_{i,i'}^{\cap}, x^* \neq y^*:
	\;\; 
	\bigl| \partial^\alpha_x\partial^\beta_y k_{B_{i,i'} } (x,y) \bigr| 
	\leq 
	c_{\alpha,\beta} \operatorname{dist}(x^*,y^*)^{-(n+r+|\alpha|+|\beta|)} ,
	\end{equation}
	where we 
	used the notation 
	$\widetilde{\cM}_{i,i'}^{\cap} := \widetilde{\cM}_{i} \cap \widetilde{\cM}_{i'}$,  
	$x := \widetilde{\gamma}_{i}^{-1}(x^*)$ and $y := \widetilde{\gamma}_{i'}^{-1}(y^*)$.
	
	In particular, 
	$k_{\cB}(\,\cdot\,, \,\cdot\,) \in C^\infty(\cM\times \cM \backslash \bigtriangleup)$
	where $\bigtriangleup = \{ (x^*,x^*) : x^* \in \cM \}$. 
\end{proposition}
The kernel estimates \eqref{eq:KAest} are in principle known.
For a detailed derivation of \eqref{eq:KAest},
we refer, e.g., to \cite[Lem.~3.0.2, 3.0.3]{RSchneider98}.
\begin{remark} 
	The kernel bound \eqref{eq:KAest} is stated with respect to the 
	distance $\operatorname{dist}$ which could be either the geodetic distance
	intrinsic to $\cM$ or also the Euclidean distance of the points
	$x^*, y^*$ immersed via $\cM$ into $\bbR^D$. This follows 
	directly from our assumptions on $\cM$, in particular its compactness.
	The numerical values of the constants $c_{\alpha,\beta}>0$ in \eqref{eq:KAest} will, 
	of course, depend on the precise notion of distance employed in \eqref{eq:KAest}.
\end{remark}

\begin{remark} \label{rmk:analyticK}
	In the case that $\cM$ and the coefficients of $\cA$ in \eqref{eq:WhNois} are 
	\emph{analytic}, the kernel estimates \eqref{eq:KAest} hold
	\emph{with explicit dependence of the constants}
	$c_{\alpha,\beta}$ on the differentiation orders $|\alpha|$, $|\beta|$.
	This follows from an analytic version of the pseudodifferential calculus
	which was developed in \cite{PDOGevrey1967}.
	It implies 
	that the covariance kernel  
	is \emph{asymptotically smooth} in the sense of \cite{WHackbHmatBook}.
	This, in turn, mathematically justifies 
	low-rank compressed, numerical approximations of covariance matrices
	in ${\mathcal H}$-matrix format, as described in 
	\cite{WHackbHmatBook} and, in connection with GRFs on manifolds,
	in \cite{Doelz2017}. The presently proposed, wavelet-based 
	compression results and \eqref{eq:KAest} hold also for finite
	differentiability of the covariance function
	in greater generality.
\end{remark} 

\section{Covariance/precision preconditioning and compression}
\label{section:results}

We consider a GRF $\GP$ 
indexed by a compact Riemannian manifold $\cM$ 
as described in Assumption~\ref{ass:cM-and-cA}\ref{ass:cM-and-cA_I}. 
We assume that $\GP$ is colored via the
white noise driven SPDE \eqref{eq:WhNois}   
with coloring operator $\cA \in OPS^{\ra}_{1,0}(\cM)$ 
satisfying 
Assumption~\ref{ass:cM-and-cA}\ref{ass:cM-and-cA_II}.  
We recall from Example~\ref{ex:whittle-matern} that 
the coloring operator $\cA$ can 
possibly be obtained as a fractional power 
of a shifted base elliptic (pseudo)differential operator
$\cL\in OPS^{\bar{r}}_{1,0}(\cM)$.  
This \emph{Whittle--Mat\'{e}rn} scenario is detailed 
in Appendix~\ref{appendix:ColShft}.
The covariance and precision operators 
in \eqref{eq:def:covariance} and \eqref{eq:def:precision}
of the GRF $\GP$  
allow for \emph{equivalent, bi-infinite matrix representations}
\begin{equation}\label{eq:BiInfMat}
\bC = \cC(\bPsi)(\bPsi) \in \bbR^{\bbN\times \bbN} 
\quad 
\text{and} 
\quad 
\bP = \cP(\bPsi)(\bPsi) \in \bbR^{\bbN\times \bbN} 
\end{equation}
when represented with respect to a MRA $\bPsi$ 
as introduced in Subsection~\ref{subsec:results:MRAs} below.

For a suitable choice 
of the 
MRA $\bPsi$ we will show  the following.
\begin{enumerate}[label=\arabic*., leftmargin=0.75cm]
	\item 
	Diagonal preconditioning renders the condition numbers 
	of arbitrary sections 
	of the bi-infinite matrices $\bC$ and $\bP$ 
	in \eqref{eq:BiInfMat} 
	uniformly bounded 
	with respect to the number of active indices. 
	\item 
	The covariance and precision operators 
	admit numerically sparse representations 
	with respect to the MRA $\bPsi$.
\end{enumerate}
These are our main findings on 
the compression of the covariance matrix $\bC$ 
and the precision matrix $\bP$ and they are detailed in
Subsections~\ref{subsec:results:precon}--\ref{subsec:results:sparse}. 

\subsection{Multiresolution analysis on manifolds}
\label{subsec:results:MRAs} 

We let $\{V_j\}_{j > j_0}$ 
be a sequence of nested, linear subspaces 
$V_j \subset V_{j+1} \subset \ldots \subset L^2(\cM)$.  
We then say that the family $\{V_j\}_{j > j_0}$ 
has  
\emph{regularity} $\gamma > 0$ and 
\emph{(approximation) order} $d\in \bbN$ if
\begin{equation}\label{eq:VjOrdReg}
\begin{split} 
\gamma 
&= 
\sup\left\{ s\in \bbR: 
V_j \subset H^{s}(\cM) \; \forall j > j_0 \right\} , 
\\ 
d 
&= 
\sup\Bigl\{ s\in \bbR: 
\inf_{v_j \in V_j } \| v - v_j \|_{L^2(\cM)} 
\lesssim 
2^{-js} \| v \|_{H^s(\cM)} 
\; 
\forall v\in H^{s}(\cM) 
\;
\forall j > j_0 \Bigr\} .  
\hspace{-0.3cm}  
\end{split}
\end{equation}
We shall suppose that the subspaces $\{V_j\}_{j > j_0}$ 
are $H^{r/2}(\cM)$-conforming, i.e.,  
that in \eqref{eq:VjOrdReg} we have $\gamma > \max\{ 0,r/2 \}$ 
for some fixed order $r\in\bbR$.

We furthermore assume that $\dim(V_j) = \cO(2^{nj})$
and, for each $j>j_0$, the space~$V_j$ is spanned by
a \emph{single-scale} basis $\bPhi_j$, i.e., 
\begin{equation}\label{eq:bPhij}
\forall j > j_0:
\quad 
V_j = \operatorname{span}\bPhi_j, 
\quad \text{where}\quad  
\bPhi_j:= \{ \phi_{j,k}: k\in \Delta_j \}. 
\end{equation}
Here, the index set 
$\Delta_j$ describes the spatial localization
of elements in $\bPhi_j$. 
We associate with these bases \emph{dual single-scale bases} 
defined by 
\begin{equation}\label{eq:bPhij-tilde}
\forall j > j_0: 
\quad   
\widetilde{\bPhi}_j := \{ \widetilde{\phi}_{j,k}: k \in \Delta_j \},
\;\;\; 
\text{with} 
\quad 
\langle \phi_{j,k}, \widetilde{\phi}_{j,k'}\rangle = \delta_{k,k'} 
\quad 
\forall k,k'\in \Delta_j. 
\end{equation} 
The vector spaces   
$\widetilde{V}_j := \operatorname{span}\widetilde{\bPhi}_j$,
$j>j_0$, 
are also nested, $\widetilde{V}_j \subset\widetilde{V}_{j+1} \subset 
\ldots \subset L^2(\cM)$, and 
the family $\{\widetilde{V}_j \}_{j>j_0}$ 
provides regularity $\widetilde{\gamma} > 0$ 
and approximation order $\widetilde{d}$. In particular, having the dual basis 
at hand, we can define the projector onto $V_j$ by 
\begin{equation}\label{eq:def_projector_Q_j}
\forall v\in L^2(\cM):
\quad
Q_j v := 
\sum_{k\in \Delta_j}
\langle 
v, \widetilde{\phi}_{j,k}
\rangle
\phi_{j,k} . 
\end{equation}
We refer to Appendix~\ref{appendix:wavelets} 
for a summary of basic properties of the bases 
$\bPhi_j$ and $\widetilde{\bPhi}_j$ 
and 
for a brief description how they can be constructed on manifolds. 

Given single-scale bases $\bPhi_j$ and $\widetilde{\bPhi}_j$,  
set  
$\nabla_j := \Delta_{j+1}\backslash \Delta_j$.
One then can construct 
\emph{biorthogonal complement bases} 
\begin{equation}\label{eq:bPsij}
\bPsi_j = \{\psi_{j,k} : k\in \nabla_j \} 
\quad 
\text{and}
\quad  
\widetilde{\bPsi}_j = \{\widetilde{\psi}_{j,k} : k\in \nabla_j \} , 
\qquad 
j>j_0, 
\end{equation} 
satisfying the \emph{biorthogonality relation}
\begin{equation}\label{eq:Biorth}
\langle \psi_{j,k} , \widetilde{\psi}_{j',k'} \rangle 
= \delta_{(j,k),(j',k')} 
= 
\begin{cases} 
1, 
& \text{if } j=j'\text{ and } k = k', 
\\
0,
& 
\text{otherwise},  
\end{cases} 
\end{equation}
such that 
\begin{equation}\label{eq:SupPsi}
\operatorname{diam} (\operatorname{supp} \psi_{j,k})
\simeq 2^{-j}, 
\quad 
j>j_0, 
\end{equation}
see Appendix~\ref{appendix:wavelets}.
For $j>j_0$, define 
$W_j := \operatorname{span} \bPsi_j$ and
$\widetilde{W}_j := \operatorname{span} \widetilde{\bPsi}_j$.
The biorthogonality \eqref{eq:Biorth} implies that, 
for all $j > j_0$, 
\[ 
V_{j+1} = W_j \oplus V_j, 
\qquad 
\widetilde{V}_{j+1} =  \widetilde{W}_j \oplus \widetilde{V}_j, 
\qquad 
\widetilde{V}_j \perp W_j,
\qquad 
V_j \perp \widetilde{W}_j . 
\]
In what follows, we use the convention 
\[
W_{j_0} := V_{j_0 + 1}, 
\quad 
\widetilde{W}_{j_0} := \widetilde{V}_{j_0+1}, 
\quad 
\text{and} 
\quad\; 
\bPsi_{j_0} := \bPhi_{j_0+1},  
\quad 
\widetilde{\bPsi}_{j_0} := \widetilde{\bPhi}_{j_0+1}. 
\]

As explained in Appendix~\ref{appendix:wavelets}, 
a biorthogonal dual pair $\bPsi, \widetilde{\bPsi}$
of wavelet bases is now obtained
from the union of the coarsest single-scale basis 
and the complement bases, i.e., 
\[
\bPsi = \bigcup_{j\geq j_0} \bPsi_{j}, 
\qquad 
\widetilde{\bPsi} = \bigcup_{j\geq j_0} \widetilde{\bPsi}_j. 
\]
We refer to $\bPsi$, resp.\ 
to $\widetilde{\bPsi}$, as primal, resp.\ dual, 
\emph{multiresolution analysis} (MRAs).
Here and throughout, 
all basis functions 
in $\bPsi$ and $\widetilde{\bPsi}$
are assumed to be normalized in $L^2(\cM)$.
Furthermore, they satisfy the 
\emph{vanishing moment property}:
\begin{equation}\label{eq:VanMom}
|\langle v, \psi_{j,k}\rangle| 
\lesssim 
2^{-j(\widetilde{d}+n/2)} 
\sup\nolimits_{|\alpha|=\widetilde{d}, \, x\in \operatorname{supp}( \psi_{j,k}) } 
|\partial^\alpha v(x)|
\quad 
\forall 
(j,k)\in \cJ. 
\end{equation}
Here, 
the countable index set is defined by 
\begin{equation}\label{eq:cJindex} 
\cJ 
:= 
\{(j,k) : j \geq j_0, \; k\in\nabla_j\}, 
\end{equation} 
where we set $\nabla_{j_0} := \Delta_{j_0+1}$, 
and the constant implied in $\lesssim$ 
in \eqref{eq:VanMom} independent of $(j,k)\in \cJ$.
A corresponding property holds for the 
duals $\widetilde{\psi}_{j,k}$.
We note that the biorthogonality allows constructions 
of $\bPsi, \widetilde{\bPsi}$ with 
$\widetilde{d} \gg d$, which will be crucial in effective 
compression of covariance operators.

The second key property of the 
multiresolution bases $\bPsi, \widetilde{\bPsi}$ 
is that they comprise \emph{Riesz bases} 
for a range of Sobolev spaces on $\cM$ 
and corresponding \emph{norm equivalences} hold:  
For all $v\in H^{t}(\cM)$, we have
\begin{equation}\label{eq:Riesz} 
\begin{split} 
\| v \|_{H^t(\cM)}^2 
& \simeq 
\sum_{j\geq j_0} \sum_{k\in \nabla_j} 2^{2jt} 
|\langle v, \widetilde{\psi}_{j,k} \rangle|^2, 
\qquad 
t\in (-\widetilde{\gamma}, \gamma),
\\ 
\| v \|_{H^t(\cM)}^2 
&\simeq 
\sum_{j\geq j_0} 
\sum_{k\in \nabla_j} 
2^{2jt} 
|\langle v, \psi_{j,k}\rangle|^2,
\qquad  
t\in (-\gamma, \widetilde{\gamma}).
\end{split} 
\end{equation}

\subsection{Covariance and precision operator preconditioning}
\label{subsec:results:precon} 

Recall the index set $\cJ$ from \eqref{eq:cJindex}. 
For $\lambda = (j,k)\in\cJ$, we set $|\lambda|:=j$. 
Furthermore, $\bD^s$ denotes
the bi-infinite diagonal matrix
\begin{equation}\label{eq:def:bD} 
\bD^s 
:= \operatorname{diag}
\bigl( 2^{s|\lambda|} 
: \lambda \in \cJ \bigr), 
\quad 
s\in\bbR. 
\end{equation} 
The next result is based on Proposition~\ref{prop:DiagPC} 
in Appendix~\ref{appendix:wavelets}. 
\begin{proposition}\label{prop:precon:C} 
	Let $\GP$ be a GRF 
	indexed by a manifold $\cM$ 
	which is defined through 
	the white noise driven SPDE 
	\eqref{eq:WhNois}.  
	Assume that the manifold $\cM$ 
	and the coloring operator 
	$\cA\in OPS^{\ra}_{1,0}(\cM)$ 
	in \eqref{eq:WhNois}
	satisfy 
	Assumptions~\ref{ass:cM-and-cA}\ref{ass:cM-and-cA_I}--\ref{ass:cM-and-cA_II}. 
	Let $\bPsi$ be a Riesz basis for $L^2(\cM)$ 
	which, properly rescaled, is a MRA in $H^{s}(\cM)$ 
	for $-\ra \leq s \leq 0$ 
	such that the norm equivalences \eqref{eq:Riesz}
	hold with 
	$\widetilde{\gamma} > \ra$ and $\gamma > 0$.
	
	Then, the bi-infinite matrix representation $\bC$ 
	for the covariance operator 
	\eqref{eq:def:covariance} 
	in the MRA $\bPsi$, see \eqref{eq:BiInfMat}, 
	satisfies the following: 
	\begin{enumerate}[leftmargin=0.75cm]
		\item\label{prop:precon:Cfull}  
		The bi-infinite matrix representation 
		$\bC$ is symmetric positive definite,  
		and it induces a self-adjoint, 
		positive definite, compact operator on $\ell^2(\cJ)$.
		Furthermore, there exist constants $0<c_- \leq c_+ < \infty$ such that
		$\sigma(\bD^{\ra} \bC \bD^{\ra}) \subset [c_-,c_+]$ 
		and 
		$\operatorname{cond}_2(\bD^{\ra} \bC \bD^{\ra}) \simeq 1$, 
		with $\bD^{\ra}$ defined according to \eqref{eq:def:bD}.  
		\item\label{prop:precon:CLambda}  
		For every index set $\Lambda \subset \cJ$ with $p=\#(\Lambda)<\infty$, 
		the $\Lambda$-section of $\bC$, 
		$\bC_\Lambda = \{ \bC_{\lambda, \lambda'} 
		: \lambda, \lambda'\in \Lambda \} \in \bbR^{p\times p}$,  
		is symmetric, positive definite and 
		it satisfies 
		$\sigma(\bD^{\ra}_\Lambda \bC_\Lambda \bD^{\ra}_\Lambda) \subset [c_-,c_+]$. 
		Here, 
		$\bD^{\ra}_\Lambda := 
		\{ \bD^{\ra}_{\lambda, \lambda'} 
		: \lambda, \lambda'\in \Lambda \} \in \bbR^{p\times p}$. 
	\end{enumerate}
\end{proposition}
\begin{proof} 
	Under 
	Assumptions~\ref{ass:cM-and-cA}\ref{ass:cM-and-cA_I}--\ref{ass:cM-and-cA_II}
	by Proposition~\ref{prop:C-and-P-ops} 
	$\cC=\cA^{-2}\in OPS^{-2\ra}_{1,0}(\cM)$ is 
	a self-adjoint, compact operator on $L^2(\cM)$. This implies 
	that the bi-infinite matrix $\bC$ is symmetric and compact 
	as an operator on $\ell^2(\cJ)$. 
	In addition, Assumption~\ref{ass:cM-and-cA}\ref{ass:cM-and-cA_II} 
	implies positivity of $\bC$: 
	by Proposition~\ref{prop:OPS_algebra_prop}\ref{item:range}
	and Proposition~\ref{prop:OPS_real_powers}, 
	the linear operator $\cC^{-1/2}\from L^2(\cM) \to H^{-\ra}(\cM)$
	is bounded. 
	Thus, there is  a constant $C_0 >0$ such that 
	\[ 
	\| v \|_{H^{-\ra}(\cM)}^2 
	=
	\|  \cC^{-1/2} \cC^{1/2} v\|_{H^{-\ra}(\cM)}^2 
	\leq 
	C_0
	\| \cC^{1/2} v\|_{L^2(\cM)}^2 
	=
	C_0
	\langle \cC v, v \rangle 
	\;\;  
	\forall 
	v \in H^{-\ra}(\cM). 
	\] 
	By 	
	writing 
	$v=\vvec^\top\bPsi \in H^{-\ra}(\cM)$ 
	for $v \in H^{-\ra}(\cM)$, 
	the 
	norm equivalences in \eqref{eq:Riesz} imply
	that there exists a constant $c_{-\ra}>0$ such that 
	\[
	c_{-\ra}^{-1} \| \bD^{-\ra} \vvec \|_2^2  
	\leq 
	\| v \|_{H^{-\ra}(\cM)}^2  
	\leq 
	c_{-\ra} \| \bD^{-\ra} \vvec \|_2^2 
	\]
	and we conclude that, for every $\vvec \in \ell_2(\cJ)$, 
	\[ 
	\| \bD^{-\ra} \vvec \|_2^2 
	\leq 
	c_{-\ra} 
	\| v \|_{H^{-\ra}(\cM)}^2 
	\leq 
	c_{-\ra}  C_0
	\langle \cC v, v \rangle
	=
	c_{-\ra}  C_0
	\vvec^\top \bC \vvec.  
	\] 
	As $\bPsi$ is a Riesz basis, 
	$\vvec \ne \zerovec$ holds if and only if $v=\vvec^\top\bPsi \ne 0$, 
	whence
	\[
	\vvec^\top \bC \vvec > 0 
	\quad 
	\Longleftrightarrow 
	\quad 
	\vvec \ne \zerovec  
	\qquad 
	\text{and} 
	\qquad 
	\vvec^\top \bC \vvec 
	\geq 
	\widetilde{c} 
	\| \bD^{-\ra} \vvec \|_2^2  
	\] 
	with $\widetilde{c} := c_{-\ra}^{-1} C_0^{-1}>0$ 
	follow. 
	Restricting this statement to sequences 
	$\vvec$ 
	which satisfy $v_\lambda = 0$ 
	for $\lambda \in \cJ\setminus\Lambda$, 
	we obtain that also 
	$\bC_\Lambda \in \bbR^{p\times p}$ is 
	symmetric positive definite, 
	where we recall that $p=\#(\Lambda)<\infty$.
	Furthermore,  
	\begin{equation}\label{eq:CpSPD} 
	\vvec^\top \bC_\Lambda \vvec 
	\geq 
	\widetilde{c} 
	\| \bD_\Lambda^{-\ra} \vvec \|_2^2  
	\quad 
	\forall 
	\vvec\in\bbR^p, 
	\end{equation}
	where the constant $\widetilde{c}>0$ 
	is independent of $\Lambda\subset\cJ$. 
	
	The assumed norm equivalences \eqref{eq:Riesz} of $\bPsi$ 
	show in particular stability 
	for $t=-\ra$ and for $t=0$,  
	and \eqref{eq:rgamtgam} holds with $-2\ra$ in place of $r$.
	Thus, \eqref{eq:DiagPC} with 
	$\bC_J$ in place of $\bB_J$ 
	and taking limit $J\to \infty$ 
	implies $\sigma(\bD^{\ra} \bC \bD^{\ra}) \subset [c_-,c_+] $ 
	and $\operatorname{cond}_2(\bD^{\ra} \bC \bD^{\ra}) \simeq 1$ 
	in~\ref{prop:precon:Cfull}. 
	From this, also 
	$\sigma(\bD^{\ra}_\Lambda \bC_\Lambda \bD^{\ra}_\Lambda) \subset [c_-,c_+]$ 
	in~\ref{prop:precon:CLambda} follows, 
	as the convex hull of the spectrum $\sigma(\bC_\Lambda)$ 
	is contained in that of $\bC$.
\end{proof}
In the next proposition we state the corresponding result 
for the precision operator $\cP$ of the GRF $\GP$.
\begin{proposition}\label{prop:precon:P} 
	Let $\GP$ be a GRF 
	indexed by a manifold $\cM$ 
	which is defined through 
	the white noise driven SPDE  
	\eqref{eq:WhNois}.  
	Assume that the manifold $\cM$ 
	and the coloring operator 
	$\cA\in OPS^{\ra}_{1,0}(\cM)$ 
	in \eqref{eq:WhNois}
	satisfy 
	Assumptions~\ref{ass:cM-and-cA}\ref{ass:cM-and-cA_I}--\ref{ass:cM-and-cA_II}. 
	Let $\bPsi$ be a Riesz basis for $L^2(\cM)$ 
	which, properly rescaled, is a MRA in $H^{s}(\cM)$ 
	for $0\leq s \leq \ra$ such that the norm equivalences \eqref{eq:Riesz}
	hold with 
	$\gamma > \ra$ and $\widetilde{\gamma} > 0$. 
	
	Then, the bi-infinite matrix representation $\bP$ 
	for the precision operator 
	\eqref{eq:def:precision} 
	in the MRA $\bPsi$, see \eqref{eq:BiInfMat}, 
	satisfies the following: 
	\begin{enumerate}[leftmargin=0.75cm]
		\item\label{prop:precon:Pfull}  
		The bi-infinite matrix representation $\bP$ 
		is symmetric  positive definite and it induces a self-adjoint, 
		positive, 
		unbounded operator  on $\ell^2(\cJ)$. 
		Furthermore, 
		there exist constants 
		$0<c_- \leq c_+ < \infty$ such that
		$\sigma(\bD^{-\ra} \bP \bD^{-\ra}) \subset [c_-,c_+]$ 
		and 
		$\operatorname{cond}_2(\bD^{-\ra} \bP \bD^{-\ra}) \simeq 1$, 
		where 
		$\bD^{-\ra}$ is defined according to \eqref{eq:def:bD}.  
		\item\label{prop:precon:PLambda}
		For every index set $\Lambda \subseteq \cJ$ with 
		$p=\#(\Lambda)<\infty$, 
		the $\Lambda$-section $\bP_\Lambda$
		of $\bP$ is symmetric, positive definite and 
		$\sigma(\bD^{-\ra}_\Lambda \bP_\Lambda \bD^{-\ra}_\Lambda) \subset [c_-,c_+]$. 
	\end{enumerate}
\end{proposition}

\begin{remark}
	At first glance, the implementation of the preconditioning in 
	Proposition~\ref{prop:precon:C}\ref{prop:precon:CLambda} or 
	Proposition~\ref{prop:precon:P}\ref{prop:precon:PLambda} 
	requires knowledge of the order $\ra$ of the coloring operator 
	$\cA$ in \eqref{eq:WhNois}. 
	However, note that the diagonal entries
	of $\bC_\Lambda$ satisfy
	\[
	\langle\cC\psi_{j,k},\psi_{j,k}\rangle \simeq 2^{-2\ra j}.
	\]	
	Therefore, in wavelet coordinates, a diagonal scaling would be
	sufficient for preconditioning and even improves it. Nonetheless,
	in covariance estimation from data, the order $\ra$ could be estimated from
	the coefficient decay rate from i.i.d.\ realizations of $\GP$ in wavelet 
	coordinates. 
	We refer to Subsection \ref{subsct:decay} for a numerical illustration.
\end{remark} 

\subsection{Covariance and precision operator sparsity}
\label{subsec:results:sparse}

The GRF $\GP$ may be expanded in the MRA $\bPsi$, 
\begin{equation*} 
\GP = 
\zvec^\top\bPsi
=
\sum_{j\geq j_0}
\sum_{k\in\nabla_j}
z_{j,k} 
\psi_{j,k}
:=
\sum_{j\geq j_0}
\sum_{k\in\nabla_j}
\langle 
\GP,\widetilde{\psi}_{j,k}
\rangle 
\psi_{j,k}, 
\end{equation*}
or in the dual MRA $\widetilde{\bPsi}$, 
\begin{equation}\label{eq:GRF_rep_by_dual_MRA}
\GP = 
\widetilde{\zvec}^\top\widetilde{\bPsi}
=
\sum_{j\geq j_0}
\sum_{k\in\nabla_j}
\widetilde{z}_{j,k} 
\widetilde{\psi}_{j,k}
:=
\sum_{j\geq j_0}
\sum_{k\in\nabla_j}
\langle 
\GP, \psi_{j,k}
\rangle 
\widetilde{\psi}_{j,k}
.
\end{equation}
The latter MRA representation of the GRF $\GP$ 
is related to the bi-infinite covariance matrix $\Cmat$ via 
\begin{equation}\label{eq:Cov_z_coord_relation}
\Cmat_{\lambda,\lambda'}
=
\langle \cC \psi_{\lambda}, \psi_{\lambda'}
\rangle 
=
\bbE[ \langle \GP, \psi_{\lambda} \rangle \langle \GP, \psi_{\lambda'} \rangle ]
=
\bbE[ \widetilde{z}_{\lambda} \widetilde{z}_{\lambda'}   ]  
\quad \forall \lambda,\lambda'\in\cJ
,
\end{equation}
where we again 
used the notation $\lambda = (j,k)\in\cJ$ with $|\lambda|=j \geq j_0$ 
and $k\in \nabla_j$.
Note that for any $f=\fvec^\top \bPsi = \sum_{\lambda\in \cJ} f_\lambda \psi_\lambda$ 
the MRA coordinates of $\cC f$ are (formally) given by
\[
\cC f = \sum_{\lambda \in \cJ}(\Cmat \fvec)_\lambda \widetilde{\psi}_{\lambda}
.
\]

Also the white noise driven SPDE
\eqref{eq:WhNois} may be cast in the dual MRA coordinates 
$\GP = \widetilde{\zvec}^\top \widetilde{\bPsi} $, 
which implies that 
$\widetilde{\zvec} \sim 
\normal(0, \widetilde{\Amat}^{-1} \widetilde{\Mmat} \widetilde{\Amat}^{-1})$
and thus by~\eqref{eq:Cov_z_coord_relation}, 
\[
\Cmat = \widetilde{\Amat}^{-1} \widetilde{\Mmat} \widetilde{\Amat}^{-1}
.
\]
Here, we used the notation 
$\widetilde{\Amat} = \cA(\widetilde{\bPsi})(\widetilde{\bPsi})$
and 
$\widetilde{\Mmat} = \mathrm{Id}(\widetilde{\bPsi})(\widetilde{\bPsi})$
in $\bbR^{\bbN\times \bbN}$. 

\subsubsection{Matrix estimates}
\label{subsubsec:MatrixEst}

The significance of using MRAs $\bPsi, \widetilde{\bPsi}$ 
for the representation \eqref{eq:GRF_rep_by_dual_MRA}
is in the \emph{numerical sparsity} of the corresponding matrices 
that result after truncating the index set $\cJ$ to finite index sets $\Lambda$.
By numerical sparsity, we mean that for any $\varepsilon>0$ 
there exists a sparse matrix, which is $\varepsilon$-close to the in general fully populated matrix.

In the following, 
we use index sets of the form 
$\Lambda_J = \{(j,k):j_0\leq j\leq J, k\in\nabla_j\}$, $J\geq j_0$,
and define, throughout what follows,   
\begin{equation}\label{eq:p=p(J)}
p=p(J) = \#(\Lambda_J).
\end{equation}
The matrices will be denoted by 
$\bA_{p}:=\bA_{\Lambda_J}$, 
$\bC_{p}:=\bC_{\Lambda_J}$ and 
$\bP_{p}:=\bP_{\Lambda_J}$.
Specifically, when represented in the MRA $\bPsi$
the matrices $\bA_{p}$, $\bC_{p}$ and $\bP_{p}$ of size $p\times p$
corresponding to coloring, covariance and precision (pseudodifferential) 
operators $\cA$, $\cC$ and $\cP$ of the GRF $\GP$ 
can be replaced by compressed approximations 
${\bA}^{\varepsilon}_p$, ${\bC}^{\varepsilon}_p$ and ${\bP}^{\varepsilon}_p$
of the same size $p\times p$ with $\cO(p)$ nonvanishing entries 
while preserving
the consistency orders $\cO(p^{-a})$ of these matrices 
with respect to the
exact counterparts $\bA$, $\bC$ and $\bP$. 
Thus, the components $\widetilde{z}_\lambda$ 
of the random coefficient vectors 
in the dual representation \eqref{eq:GRF_rep_by_dual_MRA} of $\GP$
are generically nonzero, but numerically decorrelate in the sense that 
$\bbE[\widetilde{z}_\lambda \widetilde{z}_{\lambda'}]$
is negligible for most pairs $(\lambda,\lambda')$.
This facilitates fast approximate simulation of~$\GP$
and efficient matrix estimation of $\bC$, $\bP$, 
see Section~\ref{section:application}. 

Specifically, 
for a generic pseudodifferential operator $\cB \in OPS_{1,0}^r(\cM)$ 
the kernel estimates \eqref{eq:KAest}
combined with the cancellation property \eqref{eq:VanMom} of the 
MRA $\bPsi$ (and a related property of the dual basis $\widetilde{\bPsi}$)
imply that the majority of the $p^2$ entries
\[
[\bB_p]_{\lambda,\lambda'} 
= 
\cB(\psi_{j',k'})(\psi_{j,k}) 
=
\langle \cB \psi_{j',k'}, \psi_{j,k} \rangle ,
\quad 
\lambda = (j,k) \in \Lambda_J, \; 
\lambda' = (j',k') \in \Lambda_J, 
\]
are nonzero, in general, but 
negligibly small \cite{DHSWaveletBEM2007,DPSII,RSchneider98}.
The following result quantifies this smallness.
Recall that the singular support of a function $f$ on $\cM$, 
denoted by $\operatorname{sing}\operatorname{supp}(f)$,
is given by 
$\operatorname{sing}\operatorname{supp}(f) 
= \{x\in\cM : f \text{ is not smooth at } x\}$  
and define 
\begin{equation}\label{eq:Sjk} 
S_{j,k} 
:= 
\operatorname{conv} 
\operatorname{hull}(\operatorname{supp}(\psi_{j,k}))\subset \cM, 
\qquad 
S'_{j,k} := \operatorname{sing} \operatorname{supp}(\psi_{j,k})\subset \cM, 
\end{equation} 
where $(j,k)\in\cJ$, see \eqref{eq:cJindex}.
The next proposition 
presents asymptotic size bounds on the entries $ [\bB_p]_{\lambda,\lambda'} $ 
taken from \cite[Thms.~6.1,~6.3]{DHSWaveletBEM2007}. 
\begin{proposition}\label{prop:MatEst}
	Assume that $\cB \in OPS^r_{1,0}(\cM)$ and, 
	furthermore, that a pair of mutually biorthogonal 
	MRAs $\bPsi$,~$\widetilde{\bPsi}$ 
	with $n+r+2\widetilde{d}>0$ as defined above in local coordinates 
	are available on $\cM$, 
	where $\cM$ fulfills  
	Assumption~\ref{ass:cM-and-cA}\ref{ass:cM-and-cA_I}. 
	
	Then, the bi-infinite matrix representation 
	$\bB = \cB(\bPsi)(\bPsi)$ 
	of $\cB$ has entries which 
	admit the following estimates, 
	uniformly in $j\in \bbN$:  
	\begin{enumerate}[leftmargin=0.75cm,label={\normalfont(\roman*)}]
		\item\label{prop:MatEst1}
		For every $(j,k), (j',k')\in \cJ$ 
		such that 
		$S_{j,k} \cap S_{j',k'} = \emptyset$, 
		we have
		\[
		| \langle \cB \psi_{j',k'}, \psi_{j,k} \rangle |
		\lesssim 
		2^{-(j+j')(\widetilde{d}+n/2)} 
		\operatorname{dist} (S_{j,k}, S_{j',k'})^{-(n+r+2\widetilde{d})} .
		\]
		\item\label{prop:MatEst2} 
		For every $(j,k), (j',k')\in \cJ$ 
		such that 
		$\operatorname{dist} (S'_{j,k}, S_{j',k'}) \gtrsim 2^{-j'}$, 
		we have
		\[
		| \langle \cB\psi_{j',k'}, \psi_{j,k} \rangle |
		+
		| \langle \cB\psi_{j,k}, \psi_{j',k'} \rangle |
		\lesssim 
		2^{jn/2}  2^{-j'(\widetilde{d}+n/2)} 
		\operatorname{dist} (S'_{j,k}, S_{j',k'})^{-(r+\widetilde{d})}.
		\]
	\end{enumerate}
\end{proposition}

\subsubsection{Matrix compression}
\label{subsubsec:MatCompr}

Proposition~\ref{prop:MatEst} allows to compress the (densely populated)
matrices $\bC_p, \bP_p$ corresponding to 
the action of the covariance and precision operators $\cC$ 
and $\cP$  
on finite-dimensional subspaces
to $O(p)$ nonvanishing
entries while retaining optimal asymptotic error bounds afforded by
the regularity of $\GP$. 

We describe the compression schemes for 
a generic, elliptic pseudodifferential operator
$\cB \in OPS^{r}_{1,0}(\cM)$ of order $r \in \bbR$. 
Note that, for our purposes, we have 
$\cB \in \{\cA,\cC,\cP\}$, 
where the psudodifferential operators $\cA$, $\cC$, and $\cP$ 
are as introduced in 
Section~\ref{section:GRFs-manifolds}.
Furthermore, we write
$\lambda = (j,k) \in \Lambda_J$, $\lambda' = (j',k') \in \Lambda_J$. 
With these multi-indices we associate 
supports $S_\lambda,S_{\lambda'} \subset \cM$ as well as  
singular supports $S'_\lambda, S'_{\lambda'} \subset \cM$
as defined in 
\eqref{eq:Sjk}.

\begin{definition}\label{def:tapering} 
	The \emph{a-priori matrix compression} is 
	defined in terms of positive 
	\emph{block truncation (or ``tapering'') parameters} 
	$\{ \tau'_{jj'}, \tau_{jj'} :  j_0 \leq j,j' \leq J \}$
	as follows: 
	\begin{equation}\label{eq:AprCompr}
	[\bB^\eps_p]_{\lambda,\lambda'} 
	:= 
	\begin{cases} 
	0 & {\rm dist}(S_\lambda,S_{\lambda'}) > \tau_{jj'} \;\mbox{and}\; j,j'>j_0, 
	\\
	0 & {\rm dist}(S_\lambda,S_{\lambda'}) \leq 2^{-\min\{j,j'\}} \text{ and}
	\\
	&  {\rm dist}(S'_\lambda,S_{\lambda'}) > \tau'_{jj'} \;\mbox{if} \; j' > j \geq j_0,
	\\
	&  {\rm dist}(S_\lambda,S'_{\lambda'}) > \tau'_{jj'} \;\mbox{if} \; j > j' \geq j_0,
	\\
	\langle \cB \psi_{\lambda'}, \psi_\lambda \rangle 
	& \text{otherwise}. 
	\end{cases} 
	\end{equation}
	Here, with fixed, real-valued constants
	\begin{equation}\label{eq:Fixaa'}
	a,a'> 1 \;\;\text{sufficiently large} \;
	\text{and} 
	\quad 
	d < d' < \widetilde{d}+r ,
	\end{equation}
	the parameters $\tau_{jj'}$ and $\tau'_{jj'}$ in \eqref{eq:AprCompr} 
	are 
	\begin{equation}\label{eq:AprCompBjj}
	\begin{split} 
	\tau_{jj'} 
	&:= 
	a\max\left\{ 2^{-\min\{j,j'\}}, 2^{[2J(d'-r/2)-(j+j')(d'+\widetilde{d})] 
		/(2 \widetilde{d}+r ) } \right\} ,  
	\\
	\tau'_{jj'} 
	&:= 
	a'\max\left\{ 2^{-\max\{j,j'\}}, 2^{[2J(d'-r/2)-(j+j')d'-\max\{j,j'\}\widetilde{d}]/
		(\widetilde{d}+r)} 
	\right\}.
	\end{split}
	\end{equation}
	The operator corresponding to the tapered matrix $\bB^\varepsilon_p$ 
	will be denoted by $\cB_{p}^{\varepsilon}$.
\end{definition}

The compression of (a $p \times p$ section of) 
the matrix $\bB = \cB(\bPsi)(\bPsi)$
is based 
\begin{enumerate*}[label=\textbf{\alph*)}]
	\item on \emph{a-priori accessible information} on the locations of supports 
	$S_\lambda, S_{\lambda'}\subset \cM$ and 
	of singular supports $S'_\lambda,S'_{\lambda'}\subset \cM$,
	respectively, 
	and 
	\item on sufficiently large
	(with respect to the order~$r$ of $\cB$ 
	and $n=\dim(\cM)$) 
	polynomial exactness orders $d$, $\widetilde{d}$ of the MRAs and 
	norm equivalences $\gamma,\widetilde{\gamma}$ in \eqref{eq:Riesz}.  
\end{enumerate*}
In particular, the second relation in 
\eqref{eq:Fixaa'} imposes an implicit constraint on the MRAs 
$\bPsi, \widetilde{\bPsi}$ in that 
the order $\widetilde{d}$ of 
exactness of $\widetilde{\bPsi}$
is greater than the order $d$ of exactness of $\bPsi$ 
reduced by the order~$r$ of $\cB$, i.e.,  
$\widetilde{d} > d - r$.

\begin{remark}\label{rmk:CPMoments}
	For a coloring operator $\cA\in OPS^{\ra}_{1,0}(\cM)$ 
	with $\ra>0$, the covariance operator satisfies $\cC \in OPS^{-2\ra}_{1,0}(\cM)$ 
	(see Proposition~\ref{prop:C-and-P-ops})
	so that 
	\emph{optimal numerical covariance matrix compression} 
	requires MRAs with $\widetilde{d} > d + 2\ra$ 
	(or $\widetilde{d} >d+2\beta\bar{r}$
	if $\cA = \cL^\beta$ with $\cL\in OPS^{\bar{r}}_{1,0}$ 
	and $\beta > 0$).
	Correspondingly, due to $\cP \in OPS^{2\ra}_{1,0}(\cM)$, 
	\emph{optimal precision matrix compression} requires MRAs 
	with $\widetilde{d} > d - 2\ra$,
	a much less restrictive requirement 
	on the MRAs $\bPsi, \widetilde{\bPsi}$.
	Proposition~\ref{prop:MatEst} thus implies that in one common MRA
	the precision matrix $\bP_p$ of the precision operator $\cP$ 
	affords stronger compression than 
	the corresponding covariance matrix $\bC_p$, 
	and that the
	dual system $\widetilde{\bPsi}$ should have a correspondingly larger
	number $\widetilde{d}$ of vanishing moments.
\end{remark}

For a GRF $\GP$ defined via the SPDE \eqref{eq:WhNois}
with a coloring operator $\cA \in OPS^{\ra}_{1,0}(\cM)$, 
most of the $p$ coefficients of~$\GP$ 
have numerically negligible correlation 
when represented in the MRA $\widetilde{\bPsi}$. 
That is to say,  
MRA representations provide \emph{spatial numerical decorrelation} 
of the GRF~$\GP$.
By Propositions~\ref{prop:precon:C} and \ref{prop:MatEst}, 
when represented in suitable MRAs, 
the Galerkin-projected 
covariance matrices $\{\bC_p\}_{p\geq 1}$ 
of~$\GP$ furthermore are 
numerically sparse and well-conditioned,   
uniformly with respect to the level of spatial resolution $\cO(2^{-J})$ of $\GP$ 
accessed by mesh level $J$, where we recall that $p=\#(\Lambda_J)$
and $\Lambda_J=\{(j,k) : j_0\leq j\leq J, k\in \nabla_j  \}$.

\subsubsection{Consistency and convergence}
\label{subsubsec:ConsConv}

The matrix compression in \eqref{eq:AprCompr}, 
\eqref{eq:Fixaa'}, and \eqref{eq:AprCompBjj} 
results in a family $\{ {\bB}^\eps_{p(J)} \}_{J \geq j_0}$ 
of compressed matrices ${\bB}^\eps_{p(J)} \in \bbR^{p(J)\times p(J)}$ and,
via the basis $\bPsi$, 
in  
associated perturbed operators $\cB^\eps_{p(J)}$ 
where $p(J) = \#(\Lambda_J)$. 
It turns out that the consistency error in 
$\cB_{p(J)} - \cB^\eps_{p(J)}$
can be quantified.
The assertions of the next proposition 
are proven in \cite[Thms.~9.1 \& 10.1]{DHSWaveletBEM2007}.
\begin{proposition}\label{prop:ConsA}
	Suppose that $\cM$ fulfills 
	Assumption~\ref{ass:cM-and-cA}\ref{ass:cM-and-cA_I}
	and let $\bPsi, \widetilde{\bPsi}$ be MRAs on $\cM$ 
	which satisfy $d<\widetilde{d} + r$.
	In addition, let 
	$\cB \in OPS^r_{1,0}(\cM)$ 
	for some $r\in \bbR$, 
	and assume that $\cB$ is self-adjoint and elliptic. 
	
	Then, for $r/2 \leq t,t' \leq d$ and 
	for every $w \in H^{t}(\cM)$, $v \in H^{t'}(\cM)$, 
	the consistency estimate
	\begin{equation}\label{eq:ConsA}
	\bigl| \bigl\langle \bigl(\cB - \cB^\eps_{p(J)}\bigr) Q_J w, Q_J v \bigr\rangle \bigr| 
	\lesssim
	\eps 2^{J(r-t-t')} \| w \|_{{H^t(\cM)}} \| v \|_{H^{t'}(\cM)}
	\end{equation}
	holds, where $\lesssim$ is uniform with respect to $J$, and where
	\begin{equation}\label{eq:epsaa'}
	\eps := a^{-2(d+r/2)} + (a')^{-(\widetilde{d}+r)}  .
	\end{equation}
	If, moreover, $\eps>0$ is sufficiently small (independently of $J$)
	(or, equivalently, 
	the parameters $a, a' > 1$ in \eqref{eq:epsaa'} are sufficiently large),
	the family of compressed operators $\{ \cB^\eps_{p(J)} \}_{J \geq j_0}$ 
	is uniformly stable: 
	There exists a constant $c>0$, independent of $J$, such that 
	\[
	\forall w_J\in V_J: \quad 
	\bigl| \bigl\langle {\cB}^\eps_{p(J)} w_J, w_J \bigr\rangle \bigr| 
	\geq c \| w_J \|_{H^{r/2}(\cM)}^2 .
	\]
\end{proposition}

We apply these results to the
representations of $\cC$ and $\cP$ in the MRA $\bPsi$. 
They afford \emph{optimal compressibility} of their equivalent, 
bi-infinite matrix representations \eqref{eq:BiInfMat}
\emph{provided} the biorthogonal pair of MRAs 
$\bPsi, \widetilde{\bPsi}$
has sufficient regularity and vanishing moments:
Whereas for the diagonal preconditioning results in Section \ref{subsec:results:precon}
only stability in $H^{t}(\cM)$ was required 
($t$ as specified in~\eqref{eq:Riesz} and in Proposition~\ref{prop:precon:C}
or Proposition~\ref{prop:precon:P}, respectively), 
the \emph{numerical compressibility of 
	the bi-infinite matrices $\bC$ and $\bP$}\footnote{We 
	emphasize that the bi-infinite matrices $\bC$ and $\bP$ in \eqref{eq:BiInfMat}
	are in general densely populated. Sparsity can therefore only be 
	asserted up to a numerical compression error which is bounded in 
	Proposition~\ref{prop:ConsA}.}
is based on additional properties of the 
MRAs $\bPsi, \widetilde{\bPsi}$ 
quantified by parameters 
$d,\widetilde{d},\gamma,\widetilde{\gamma}$ 
from Section~\ref{subsec:results:MRAs}.

\begin{proposition}\label{prop:cov-compr}
	Let $\cM$ satisfy 
	Assumption~\ref{ass:cM-and-cA}\ref{ass:cM-and-cA_I}  
	and let  
	the coloring operator $\cA \in OPS^{\ra}_{1,0}(\cM)$
	fulfill Assumption~\ref{ass:cM-and-cA}\ref{ass:cM-and-cA_II}
	for some $\ra>n/2$. 
	In addition, 
	let~$\bPsi$ be a MRA 
	such that \eqref{eq:VjOrdReg}--\eqref{eq:Riesz} hold 
	with $\widetilde{\gamma}>\ra$ and $\gamma>0$. 
	Let $\cC = \cA^{-2}$  be the covariance operator 
	of the GRF $\GP$ in the SPDE~\eqref{eq:WhNois}. 
	Denote the tapered
	covariance matrix by ${\bC}^{\varepsilon}_{p(J)}$, 
	with tapering \eqref{eq:AprCompr} 
	and covariance tapering parameters 
	$\{ \tau_{jj'} (\cC) , \tau'_{jj'}(\cC) : j_0 \leq j,j' \leq J\}$, 
	defined as in \eqref{eq:Fixaa'}--\eqref{eq:AprCompBjj} 
	with $-2\ra$ in place of $r$. 
	
	Then, there exists $\varepsilon_0>0$ such that, 
	for every $\varepsilon\in(0,\varepsilon_0)$,  
	there are parameter choices $a,a' > 0$ in \eqref{eq:Fixaa'}, 
	which are independent of $p(J)$, such that: 
	\begin{enumerate}[leftmargin=0.75cm,label={\normalfont (\roman*)}]
		\item\label{item:item_i_Cov_compr}
		For every $J\geq j_0$, the tapered matrix $\bC^{\varepsilon}_{p(J)}$
		is symmetric, positive definite. 
		\item\label{item:item_ii_Cov_compr}
		Diagonal preconditioning  
		renders $\bC^{\varepsilon}_{p(J)}$ uniformly well-conditioned:
		There are constants 
		$0<\widetilde{c}_-\leq \widetilde{c}_+ < \infty$ 
		such that 
		\[
		\forall J\geq j_0: 
		\quad 
		\sigma\bigl( \bD^{\ra}_{p(J)} {\bC}^{\varepsilon}_{p(J)} \bD^{\ra}_{p(J)} \bigr) 
		\subset 
		[\widetilde{c}_-,\widetilde{c}_+] .
		\]
		\item\label{item:item_iii_Cov_compr}
		The tapered covariance matrices $\{ \bC^{\varepsilon}_{p(J)} \}_{J\geq j_0}$ 
		are optimally sparse 
		in the sense that, as $J\to \infty$, 
		the number of non-zero entries of $\bC^{\varepsilon}_{p(J)}$ 
		is $\cO(p(J))$.  
		\item\label{item:item_iv_Cov_compr}
		Let 
		$\cC^\varepsilon_{p(J)}$ be the operator corresponding to the 
		tapered covariance matrix~$\bC^\varepsilon_{p(J)}$ 
		and assume that   
		\begin{equation}\label{eq:CovCmpPar}
		-\ra \leq t,t' \leq d < \widetilde{d} - 2\ra .   
		\end{equation}
		Then, for every $J\geq j_0$ and all $v\in H^{t'}(\cM),w\in H^{t}(\cM)$, 
		\[
		\bigl| \bigl\langle \bigl(\cC - \cC^{\varepsilon}_{p(J)} \bigr)Q_J w,Q_J v \bigr\rangle \bigr|
		\lesssim 
		\varepsilon 2^{J(-2\ra-t-t')} \| w \|_{H^t(\cM)} \| v \|_{H^{t'}(\cM)} 
		\] 
		holds, where $Q_J$
		is the projector in \eqref{eq:def_projector_Q_j}. 
	\end{enumerate}
\end{proposition}

\begin{proof}
	Throughout this proof, we write $p=p(J)$, see also~\eqref{eq:p=p(J)}.
	
	Proof of \ref{item:item_iv_Cov_compr}: 
	The consistency estimate will follow from \eqref{eq:ConsA} 
	in Proposition~\ref{prop:ConsA} once the assumptions of that 
	proposition are verified. 
	As Assumptions~\ref{ass:cM-and-cA}\ref{ass:cM-and-cA_I}--\ref{ass:cM-and-cA_II} 
	hold, 
	$\cA \in OPS^{\ra}_{1,0}(\cM)$ 
	is self-adjoint, positive and 
	$\cC = \cA^{-2} \in OPS^{-2\ra}_{1,0}(\cM)$
	satisfies the assumptions of Proposition \ref{prop:ConsA} 
	with $r$ replaced by $-2\ra$. 
	Since by assumption also the MRAs $\bPsi, \widetilde{\bPsi}$ satisfy 
	\eqref{eq:VjOrdReg}--\eqref{eq:Riesz} with $-2\ra$ in place of $r$, 
	the tapering scheme \eqref{eq:AprCompr}--\eqref{eq:AprCompBjj}
	with covariance tapering parameters 
	$\tau_{jj'}(\cC), \tau'_{jj'}(\cC)$ 
	corresponding to these orders will allow using 
	Proposition \ref{prop:ConsA}. This implies assertion~\ref{item:item_iv_Cov_compr}.
	The moment conditions on the MRA $\bPsi$ in Remark \ref{rmk:CPMoments}
	also imply the sparsity assertion~\ref{item:item_iii_Cov_compr}
	(see \cite[Thm. 11.1]{DHSWaveletBEM2007}, \cite[Thm. 8.2.10]{RSchneider98}).
	
	To prove positive definiteness for the tapered covariance matrix
	${\bC}^{\varepsilon}_{p(J)}$, we use 
	positive definiteness of the finite section $\bC_{p(J)}$, 
	see \ref{prop:precon:CLambda} of 
	Proposition~\ref{prop:precon:C}, combined with item 
	\ref{item:item_iv_Cov_compr}.
	Namely, choosing in the tapering coefficients 
	$\tau_{j j'}(\cC), \tau_{j j'}'(\cC)$ the parameter 
	$\eps > 0$ sufficiently small, it follows from 
	\ref{item:item_iv_Cov_compr} with $t=t'=-\ra$ and  
	the $H^{-\ra}(\cM)$ Riesz basis property of $\bPsi$ that
	there exists a constant $C>0$, independent of $J$ and $p=p(J)$, such that, 
	for every  
	$\eps\in(0,\eps_0)$, 
	\begin{equation}\label{eq:Ceps-err}
	\forall \vvec\in \bbR^{p(J)}:
	\quad 
	\bigl| \vvec^\top \bigl( \bC_{p(J)} - {\bC}^{\varepsilon}_{p(J)} \bigr) \vvec \bigr|
	\leq 
	C \eps \, \bigl\| \bD_{p(J)}^{-\ra} \vvec \bigr\|_2^2 .
	\end{equation} 
	We therefore find, for 
	$\vvec\in \bbR^{p(J)} \setminus\{0\}$ with $v = \vvec^\top\bPsi\in H^{-\ra}(\cM)$, 
	\begin{align*} 
	\vvec^\top \bC^{\varepsilon}_{p(J)} \vvec
	&=
	\vvec^\top \bC_{p(J)} \vvec 
	+ 
	\vvec^\top \bigl( \bC^{\varepsilon}_{p(J)} - \bC_{p(J)} \bigr) \vvec 
	\geq 
	(\widetilde{c} - C \eps) 
	\bigl\| \bD^{-\ra}_{p(J)} \vvec \bigr\|_2^2 
	>0 , 
	\end{align*} 
	provided that $\eps>0$ is so small that $\widetilde{c} - C \eps> 0$. 
	Here $\widetilde{c}>0$ is the constant in \eqref{eq:CpSPD}, 
	which is independent of $p$. 
	This proves \ref{item:item_i_Cov_compr}.
	
	To show \ref{item:item_ii_Cov_compr}, 
	we again combine \ref{prop:precon:CLambda} 
	of Proposition~\ref{prop:precon:C}
	with \ref{item:item_iv_Cov_compr}. 
	By \ref{prop:precon:CLambda} 
	there exists $c_{-}, c_{+} > 0$ such that 
	$\sigma\bigl( \bD_{p(J)}^{\ra} \bC_{p(J)} \bD_{p(J)}^{\ra} \bigr) \subset [c_-,c_+]$. 
	Furthermore, by \eqref{eq:Ceps-err}  
	$\bigl\| \bD^{\ra}_{p(J)} \bigl( \bC_{p(J)} - \bC^{\eps}_{p(J)} \bigr) \bD^{\ra}_{p(J)} \bigr\|_2 
	\leq c_-/2$ 
	for sufficiently small $\eps>0$.  
	Thus, we obtain assertion \ref{item:item_ii_Cov_compr} 
	with 
	$\widetilde{c}_{-} \geq c_-/2$ and 
	$\widetilde{c}_{+} \leq c_+ + c_-/2$.
\end{proof}

Along the same lines, one proves the following result for
the precision operator. 

\begin{proposition}\label{prop:prec-compr}
	Let $\cM$ satisfy Assumption~\ref{ass:cM-and-cA}\ref{ass:cM-and-cA_I} 
	and let 
	the coloring operator $\cA \in OPS^{\ra}_{1,0}(\cM)$
	fulfill Assumption~\ref{ass:cM-and-cA}\ref{ass:cM-and-cA_II}
	for some $\ra>n/2$. 
	In addition, let~$\bPsi$ be a MRA 
	such that \eqref{eq:VjOrdReg}--\eqref{eq:Riesz} hold 
	with $\gamma>\ra$ and $\widetilde{\gamma}>0$. 
	Let $\cP = \cA^2$  be the precision operator 
	of the GRF $\GP$ in the SPDE~\eqref{eq:WhNois}. 
	Denote the tapered
	precision matrix by ${\bP}^{\varepsilon}_{p(J)}$, 
	with tapering \eqref{eq:AprCompr} 
	and precision tapering parameters 
	$\{\tau_{j j'}(\cP), \tau'_{j j'}(\cP) : j_0 \leq j,j' \leq J\}$, 
	defined as in \eqref{eq:Fixaa'}--\eqref{eq:AprCompBjj} 
	with $2\ra$ in place of $r$. 
	
	Then, there exists $\varepsilon_0>0$ such that, 
	for every $\eps\in(0,\eps_0)$, 
	there are parameter choices $a,a' > 0$ in \eqref{eq:Fixaa'}, 
	which are independent of $p(J)$, such that 
	\begin{enumerate}[leftmargin=0.75cm,label={\normalfont (\roman*)}] 
		\item\label{item:item_i_prec_compr}
		For every $J\geq j_0$, 
		the tapered matrix $\bP^{\varepsilon}_{p(J)}$
		is symmetric, positive definite. 
		\item\label{item:item_ii_prec_compr}
		Diagonal preconditioning  
		renders $\bP^{\varepsilon}_{p(J)}$ uniformly well-conditioned:
		There are constants 
		$0<\widetilde{c}_-\leq \widetilde{c}_+ < \infty$ 
		such that 
		\[
		\forall J\geq j_0: 
		\quad 
		\sigma\bigl(\bD^{-\ra}_{p(J)} {\bP}^{\varepsilon}_{p(J)} \bD^{-\ra}_{p(J)} \bigr) 
		\subset 
		[\widetilde{c}_-,\widetilde{c}_+] .
		\] 
		\item\label{item:item_iii_prec_compr}
		The tapered precision matrices 
		$\{ \bP^{\varepsilon}_{p(J)} \}_{J\geq j_0}$ 
		are optimally sparse in the sense that, as $J\to \infty$
		the number of non-zero entries of $\bP^\eps_{p(J)}$ is $\cO(p(J))$. 
		\item\label{item:item_iv_prec_compr} 
		Let 
		$\cP^\varepsilon_{p(J)}$ be the operator corresponding to the 
		tapered precision matrix~$\bP^\varepsilon_{p(J)}$ 
		and assume that 
		\begin{equation}\label{eq:PrecCmpPar}
		\ra \leq t,t' \leq  d < \widetilde{d} + 2\ra . 
		\end{equation}
		Then, for every $J\geq j_0$ 
		and all $v\in H^{t'}(\cM),w\in H^{t}(\cM)$, 
		\[ 
		\bigl| \bigl\langle \bigl(\cP - \cP^{\varepsilon}_{p(J)} \bigr) Q_J w, Q_J v \bigr\rangle \bigr|
		\lesssim 
		\varepsilon 2^{J(2\ra-t-t')} \| w \|_{H^t(\cM)} \| v \|_{H^{t'}(\cM)}.
		\] 
		holds, where 
		$Q_J$ is the projector in \eqref{eq:def_projector_Q_j}. 
	\end{enumerate}
\end{proposition}

\begin{remark}\label{rmk:CPcompr}
	\begin{enumerate}[leftmargin=0.75cm,label={\normalfont (\roman*)}]  
		\item
		Propositions~\ref{prop:cov-compr} 
		and~\ref{prop:prec-compr} state that the matrix representations
		of both covariance and precision operator of the GRF $\GP$ 
		in suitable wavelet bases can be optimally compressed. 
		We emphasize that in Proposition~\ref{prop:cov-compr}, the 
		moment conditions \eqref{eq:CovCmpPar} 
		on the MRA $\bPsi$ for optimal covariance matrix compression
		are considerably stronger than \eqref{eq:PrecCmpPar}
		imposed for optimal precision matrix compression 
		in Proposition~\ref{prop:prec-compr}. 
		Note also that 
		in Propositions~\ref{prop:cov-compr} and~\ref{prop:prec-compr} 
		possibly different MRAs
		for covariance and precision matrix compression are admitted.
		With respect to \emph{one common MRA $\bPsi$}
		the compressibility of the precision operator matrix is  
		higher than the compressibility of the covariance operator. 
		This is consistent with the fact 
		that a Gaussian Whittle--Mat\'ern field 
		with precision operator 
		$\cP = ( -\Delta_\cM + \kappa^2 )^{2\beta}$ 
		(see Appendix~\ref{appendix:ColShft})
		satisfies a Markov property whenever $2\beta\in\bbN$,  
		compare~e.g.~\cite[Chap.~3]{Rozanov1982}. 
		\item 
		The results are robust with respect to the parameters $a,a'>1$ in \eqref{eq:Fixaa'}:
		Once $a,a'>1$ are sufficiently large, increasing these values in the 
		parameter choices \eqref{eq:AprCompBjj} will not affect the asymptotic statements
		in Propositions~\ref{prop:cov-compr} and~\ref{prop:prec-compr}. 
		Increasing $a,a'$ will, however, change the constants in the
		asymptotic error bounds, e.g., 
		the constant implied in $\cO(p(J))$ will increase with $a,a'$.
		\item
		In the case of the Whittle--Mat\'ern coloring, where 
		$\cA = (\cL + \kappa^2)^\beta$, 
		see Example~\ref{ex:whittle-matern} and Appendix~\ref{appendix:ColShft},
		a shift function $\kappa^2\in C^\infty(\cM)$, 
		which takes large values 
		$\kappa^2(x) \geq \kappa_-^2 \gg 0$ 
		(corresponding to small spatial correlation lengths), 
		might allow quantitative improvements in the matrix compression, 
		see Subsection~\ref{subsct:influence-length} for a numerical illustration. 
		\item 
		For a fixed order $\ra>n/2$ of the coloring operator $\cA$, 
		the tapering pattern 
		\eqref{eq:AprCompr}--\eqref{eq:AprCompBjj} 
		is universal, i.e., independent of the particular 
		(pseudodifferential) operators $\cP$ and $\cC$
		and contains explicit a-priori 
		information about the locations of the $\cO(p(J))$ many ``relevant'' 
		entries of ${\bC}^\varepsilon_p$, ${\bP}^\varepsilon_p$.
		It may be employed in constructing oracle estimators in 
		graphical LASSO algorithms (e.g., \cite{JJSvdG,UhlerGGrMod} 
		and the references there) to infer $\bP_p$ from 
		(multilevel) estimates for~$\bC_p$.
	\end{enumerate}
\end{remark}

\section{Applications: simulation, estimation, and prediction}
\label{section:application}

\subsection{Efficient numerical simulation of colored GRFs}
\label{subsec:simulation}

As a first application of the results from Section~\ref{section:results} 
we consider the problem of sampling from the GRF $\GP$ 
which solves the white noise equation \eqref{eq:WhNois}. 
We recall from 
\eqref{eq:GRF_rep_by_dual_MRA}--\eqref{eq:Cov_z_coord_relation}
that the GRF $\GP$ and the SPDE \eqref{eq:WhNois}  
may equivalently be cast in coordinates 
corresponding to the dual MRA $\widetilde{\bPsi}$: 
\begin{equation}\label{eq:GRFWavMat}
\GP = \sum_{\lambda \in \cJ} \langle \GP, \psi_{\lambda}\rangle \widetilde{\psi}_\lambda  
\quad 
\Longleftrightarrow
\quad 
\widetilde{\bA} 
\widetilde{\zvec} 
= 
\wvec 
. 
\end{equation}
Here, $\widetilde{\bA}$ denotes the bi-infinite matrix
$\cA(\widetilde{\bPsi})(\widetilde{\bPsi})$ 
and 
the coefficient sequences $\widetilde{\zvec}, \wvec$ have entries 
$\widetilde{z}_\lambda = \langle \GP, \psi_{\lambda}\rangle$ and  
$w_\lambda = \langle \white, \widetilde{\psi}_{\lambda}\rangle$, 
respectively. 
By the properties of Gaussian white noise, 
the random vector $\wvec$ is   
$\normal(\zerovec,\widetilde{\Mmat})$-distributed, 
where $\widetilde{\Mmat} = \mathrm{Id}(\widetilde{\bPsi})(\widetilde{\bPsi})$ 
denotes the Gramian with respect to the dual MRA $\widetilde{\bPsi}$. 
For a sequence $\xivec$ of i.i.d.\ $\normal(0,1)$-distributed 
random variables we therefore conclude that 
\begin{equation}\label{eq:zvec-distr} 
\textstyle 
\wvec \overset{d}{=} \sqrt{\widetilde{\Mmat}} \, \xivec 
\quad 
\text{and} 
\quad 
\widetilde{\zvec} \overset{d}{=} 
\widetilde{\bA}^{-1} \sqrt{\widetilde{\Mmat}} \, \xivec,  
\quad 
\widetilde{\zvec} 
\sim 
\normal(\zerovec, \bC), 
\quad 
\bC = \widetilde{\bA}^{-1} \widetilde{\Mmat} \widetilde{\bA}^{-1}. 
\end{equation}
We now consider the vector 
$\widetilde{\zvec}_{p} \in \bbR^{p}$, 
where the subscript $p=p(J)$ corresponds to 
the finite index set 
$\Lambda(J)$ as in \eqref{eq:p=p(J)}. 
As a result of the distributional equalities in \eqref{eq:zvec-distr}, 
sampling from $\widetilde{\zvec}_{p}$ can be realized efficiently 
in essentially (up to $\log$ factors) 
linear computational cost  
by approximating the matrix square root 
of the well-conditioned mass matrix $\widetilde{\Mmat}_{p}$ 
as suggested in \cite{Hale2008} 
and by preconditioning the compressed 
matrix~$\widetilde{\bA}^\eps_{p}$. 
(Note that an analogous preconditioning result as in \eqref{eq:DiagPC}  
of Proposition~\ref{prop:DiagPC} can also be obtained for 
the dual MRA $\widetilde{\bPsi}$.)
A similar approach employing MRAs 
has already been discussed in \cite[Sec.~5]{HKS1}. 

In what follows, we discuss a different viewpoint. 
A common scenario in applications is that 
the coloring operator $\cA$ is not explicitly available, but 
the kernel related to the covariance operator $\cC$ 
via the Schwartz kernel theorem (see Section~\ref{section:GRFs-manifolds}) 
is known. 
In this case, it is in principle possible 
to determine all entries for every finite section $\bC_p$ 
of the bi-infinite covariance matrix $\bC = \cC(\bPsi)(\bPsi)$ 
but not of $\widetilde{\bA}$. 
For this reason, in order to sample from  
$\widetilde{\zvec}_p \sim\normal(\zerovec,\bC_p)$, 
we will focus on approximating the matrix square root $\sqrt{\bC_p}$ 
of the covariance matrix. 

To this end, we first note the following: 
By letting $\Imat_p\in\bbR^{p\times p}$ denote the 
identity matrix and $\xivec_p \in \bbR^p$ 
be a random vector with distribution 
$\xivec\sim\normal(\zerovec,\Imat_p)$, 
we obtain 
\begin{equation}\label{eq:def:zvecp} 
\textstyle 
\widetilde{\zvec}_p 
\overset{d}{=} 
\bD_p^{-\ra}
\sqrt{ \bD_p^{\ra} \bC_p \bD_p^{\ra} } 
\, 
\xivec_p, 
\qquad 
\widetilde{\zvec}_p 
\sim\normal(\zerovec, \bC_p), 
\end{equation}
where $\ra>n/2$ is the order of $\cA\in OPS_{1,0}^{\ra}(\cM)$ and 
$\bD^{\ra}_p$ denotes the 
finite $\Lambda_J \times \Lambda_J$ section 
of the diagonal matrix $\bD^{\ra}$ 
defined in \eqref{eq:def:bD}. 
We let $\bC^\eps_p$ be the tapered covariance matrix 
with tapering \eqref{eq:AprCompr}--\eqref{eq:AprCompBjj} 
(with $-2\ra$ in place of $r$) and 
define the matrices 
\begin{equation}\label{eq:def:Rmats}
\Rmat_p :=  \bD_p^{\ra} \bC_p \bD_p^{\ra} \in \bbR^{p\times p}, 
\qquad 
\Rmat_p^\eps :=  \bD_p^{\ra} \bC_p^\eps \bD_p^{\ra} \in \bbR^{p\times p}, 
\end{equation}
as well as the approximation  
\begin{equation}\label{eq:def:zvecpeps} 
\textstyle 
\widetilde{\zvec}_p^\eps 
:= 
\bD_p^{-\ra}
\sqrt{ \bD_p^{\ra} \bC_p^\eps \bD_p^{\ra} }
\, 
\xivec_p
= 
\bD_p^{-\ra} 
\sqrt{ \Rmat_p^\eps } 
\, 
\xivec_p , 
\qquad 
\widetilde{\zvec}_p^\eps 
\sim\normal(\zerovec, \bC_p^\eps). 
\end{equation}

Note that $\Rmat_p$ is well-conditioned, uniformly in $J$,  
and, for $\eps\in(0,\eps_0)$ sufficiently small,  
also the compressed (sparse) matrix 
$\Rmat_p^\eps$ is uniformly well-conditioned,  
see  
Proposition~\ref{prop:precon:C}\ref{prop:precon:CLambda} and
Proposition~\ref{prop:cov-compr}\ref{item:item_ii_Cov_compr}--\ref{item:item_iii_Cov_compr}, 
respectively. 
In particular, 
\begin{equation}\label{eq:spectral-c-tildes}
\exists\, \widetilde{c}_-, \widetilde{c}_+>0: 
\qquad 
\forall J \geq j_0: 
\quad 
\sigma\bigl( \Rmat^\eps_p \bigr) \subset [\widetilde{c}_-,\widetilde{c}_+]. 
\end{equation}
Therefore, the contour integral method 
suggested in \cite{Hale2008} 
to approximate the matrix square root 
will converge exponentially in the number of 
quadrature nodes of the contour integral. 
Specifically, for fixed $K\in\bbN$, 
we consider (see~\cite[Eq.~(4.4) and comments below]{Hale2008}) 
the approximation 
\begin{equation}\label{eq:def:Smat} 
{\textstyle\sqrt{\Rmat^\eps_p }} 
\approx 
\Smat_{K} 
:=
\frac{2 E \sqrt{\widetilde{c}_- }}{
	\pi K} 
\, \Rmat^\eps_p \, 
\sum_{k=1}^{ K } 
\frac{\operatorname{dn}\left( 
	t_k | 1 - \widehat{\varkappa}_{\mathrm{R}}^{-1} \right)}
{\operatorname{cn}^2\left(t_k | 
	1 - \widehat{\varkappa}_{\mathrm{R}}^{-1}  
	\right)}  
\left(\Rmat^\eps_p   + w_k^2 \Imat_p \right)^{-1} .  
\end{equation} 
Here,  
$\operatorname{sn}, \operatorname{cn}$ 
and $\operatorname{dn}$ are the Jacobian elliptic 
functions~\cite[Ch.~16]{Abramowitz1964}, 
$E$ is the complete elliptic integral 
of the second kind
associated with the parameter 
$\widehat{\varkappa}_{\mathrm{R}}^{-1}$ 
\cite[Ch.~17]{Abramowitz1964}, 
$\widehat{\varkappa}_{\mathrm{R}} := \widetilde{c}_+/\widetilde{c}_-$, 
and, for $k\in\{1,\ldots, K \}$, 
\[
w_k 
:= 
\sqrt{\widetilde{c}_- } \,  
\frac{
	\operatorname{sn}\left(t_k | 
	1 - \widehat{\varkappa}_{\mathrm{R}}^{-1} \right)}{
	\operatorname{cn}\left(t_k | 
	1 - \widehat{\varkappa}_{\mathrm{R}}^{-1} \right)} 
\quad 
\text{and} 
\quad 
t_k := 
\frac{\bigl(k-\tfrac{1}{2} \bigr) E}{K}. 
\]
Employing the approximation $\Smat_K$ from \eqref{eq:def:Smat} 
in \eqref{eq:def:zvecpeps} 
finally yields a computable approximation 
for $\widetilde{\zvec}_p$ in \eqref{eq:def:zvecp}, 
\begin{equation}\label{eq:def:zvecpepsK} 
\widetilde{\zvec}_{p, K}^\eps  
:= 
\bD^{-\ra}_p 
\Smat_K 
\xivec_p, 
\qquad 
\widetilde{\zvec}_{p, K}^\eps   
\sim 
\normal(\zerovec, \bD^{-\ra}_p \Smat_K^2 \bD^{-\ra}_p ). 
\end{equation} 

\begin{theorem}\label{thm:sampling} 
	Suppose that the manifold $\cM$ and 
	the operator $\cA \in OPS^{\ra}_{1,0}(\cM)$ 
	satisfy 
	Assumptions~\ref{ass:cM-and-cA}\ref{ass:cM-and-cA_I}--\ref{ass:cM-and-cA_II}   
	for some $\ra > n/2$. 
	Let $\cC = \cA^{-2}$  be the covariance operator 
	of the GRF $\GP$ that solves the SPDE~\eqref{eq:WhNois}, and
	let $\widetilde{\zvec} = \langle \GP, \Psi \rangle$
	be the coordinates of $\GP$ 
	when cast in the dual MRA $\widetilde{\Psi}$, see \eqref{eq:GRF_rep_by_dual_MRA}. 
	For $p=p(J)$, see \eqref{eq:p=p(J)}, denote the tapered
	covariance matrix by ${\bC}^{\varepsilon}_p$, 
	with tapering \eqref{eq:AprCompr}--\eqref{eq:AprCompBjj}, 
	where $\eps\in(0,\eps_0)$ is sufficiently small 
	such that \ref{item:item_i_Cov_compr}--\ref{item:item_iv_Cov_compr} of 
	Proposition~\ref{prop:cov-compr}
	hold. 
	
	\begin{enumerate}[leftmargin = 1cm]
		\item\label{thm:sampling-i}  
		Let $\Rmat_p^\eps \in \bbR^{p\times p}$ be defined as in  
		\eqref{eq:def:Rmats} and 
		let $\widetilde{c}_{-}, \widetilde{c}_{+}>0$ be 
		the constants in \eqref{eq:spectral-c-tildes}. 
		Then, the family of matrices  
		$\{ \Smat_K\}_{K\in\bbN}$ defined by \eqref{eq:def:Smat} 
		satisfies 
		\[
		\textstyle 
		\exists c,C>0  	
		\quad 
		\forall K\in\bbN: 
		\quad 
		\bigl\| \sqrt{\Rmat^\eps_p} - \Smat_K \bigr\|_2 
		\leq 
		C e^{-cK} , 
		\]
		where the constants $c,C>0$ depend on 
		$\varkappa_{\mathrm{R}} = \widetilde{c}_{+}/\widetilde{c}_{-}$, 
		but not on $p$ and $K$. 
		\item\label{thm:sampling-ii}  
		Let the $\bbR^p$-valued random vectors 
		$\widetilde{\zvec}_p, \widetilde{\zvec}_p^\eps, 
		\widetilde{\zvec}_{p,K}^\eps$ be defined as in 
		\eqref{eq:def:zvecp}, \eqref{eq:def:zvecpeps} 
		and \eqref{eq:def:zvecpepsK}, respectively.  
		Then, there exist constants $C,c>0$ such that 
		for every $p, K\in \bbN$, 
		$\eps \in (0,\eps_0)$, 
		and $0\leq s < \ra - n/2$ 
		we have  
		\begin{align} 
		\bigl( \bbE\bigl[ 
		\| \widetilde{\zvec} - \widetilde{\zvec}_{p, K}^\eps \|_2^2 
		\bigr] \bigr)^{1/2}
		&\leq  C \bigl( 2^{-sJ} + \eps + e^{-cK} \bigr). 
		\label{eq:thm:sampling-err-z} 
		\end{align} 
		In \eqref{eq:thm:sampling-err-z}, 
		the integers $J$ and $p$ are related as in \eqref{eq:p=p(J)}. 
	\end{enumerate}
\end{theorem} 

\begin{proof}
	Part~\ref{thm:sampling-i} is proven in \cite[Thm.~4.1]{Hale2008}. 
	
	To show \eqref{eq:thm:sampling-err-z} of~\ref{thm:sampling-ii}, 
	we first split the error as follows, 
	\begin{align*} 
	\bigl( \bbE\bigl[ 
	\| \widetilde{\zvec} - \widetilde{\zvec}_{p, K}^\eps \|_2^2 
	\bigr] \bigr)^{\frac{1}{2}}
	&\leq 
	\bigl( \bbE\bigl[ 
	\| \widetilde{\zvec} - \widetilde{\zvec}_{p} \|_2^2 
	\bigr] \bigr)^{\frac{1}{2}}
	+ 
	\bigl( \bbE\bigl[ 
	\| \widetilde{\zvec}_p - \widetilde{\zvec}_p^\eps \|_2^2 
	\bigr] \bigr)^{\frac{1}{2}}
	+ 
	\bigl( \bbE\bigl[ 
	\| \widetilde{\zvec}_p^\eps - \widetilde{\zvec}_{p, K}^\eps \|_2^2 
	\bigr] \bigr)^{\frac{1}{2}}
	\\
	&=: 
	\text{(A)} 
	+ 
	\text{(B)} 
	+ 
	\text{(C)}. 
	\end{align*}
	To bound term (C), we note the identity 
	\[ 
	\textstyle 
	\bbE\bigl[ 
	\| \widetilde{\zvec}^\eps_p - \widetilde{\zvec}_{p, K}^\eps \|_2^2 
	\bigr] 
	= 
	\bbE\Bigl[ 
	\bigl\| \bD_p^{-\ra}\bigl(\sqrt{\Rmat_p^\eps} - \Smat_K \bigr) \xivec_p \bigr\|_2^2 
	\Bigr] 
	= 
	\bigl\| \bD_p^{-\ra} \bigl( \sqrt{\Rmat_p^\eps} - \Smat_K \bigr) \bigr\|_{\mathrm{HS}}^2 , 
	\]
	which follows from the fact that $\xi_1,\ldots,\xi_p$ 
	are i.i.d.\ $\normal(0,1)$-distributed. 
	Since 
	\[
	\textstyle 
	\bigl\| \bD_p^{-\ra} \bigl( \sqrt{\Rmat_p^\eps} - \Smat_K \bigr) \bigr\|_{\mathrm{HS}}
	= 
	\bigl\| \bigl( \sqrt{\Rmat_p^\eps} - \Smat_K \bigr) \bD_p^{-\ra} \bigr\|_{\mathrm{HS}}
	\leq 
	\bigl\| \sqrt{\Rmat_p^\eps} - \Smat_K \bigr\|_{2} 
	\bigl\| \bD_p^{-\ra} \bigr\|_{\mathrm{HS}} , 
	\]
	the estimate 
	$\text{(C)} \leq C e^{-cK}$ 
	follows from part~\ref{thm:sampling-i} 
	if $\| \bD_p^{-\ra} \|_{\mathrm{HS}}\lesssim 1$. 
	Indeed, the assumption  
	$\dim(V_j) = \cO(2^{nj})$  
	combined with 
	the identity $V_{j+1} = W_j \oplus V_j$ yields that 
	$\#(\nabla_j) = \dim(W_j) = \cO\bigl((2^n - 1)2^{nj} \bigr)$,  
	for all $j\geq j_0$, and by definition \eqref{eq:def:bD}
	\[ 
	\bigl\| \bD_p^{-\ra} \bigr\|_{\mathrm{HS}}^2 
	= 
	\sum_{\lambda\in\Lambda_J} 
	2^{-2\ra|\lambda|} 
	= 
	\sum_{j=j_0}^J 
	\sum_{k\in\nabla_j} 
	2^{-2\ra j} 
	\lesssim 
	(2^n - 1) 
	\sum_{j=j_0}^J 
	2^{ - (2\ra-n) j } . 
	\]
	Since $\ra>n/2$ is assumed, we conclude that 
	\[
	\bigl\| \bD_p^{-\ra} \bigr\|_{\mathrm{HS}}^2 
	\lesssim
	\sum_{j=0}^\infty 
	2^{ - (2\ra - n) j } 
	= 
	\bigl( 1 - 2^{-(2\ra-n)} \bigr)^{-1} < \infty, 
	\]
	where the constant implied in $\lesssim$ 
	is independent of $J$ and, thus, of $p=p(J)$. 
	
	Similar arguments yield the bound 
	\[
	\textstyle 
	\text{(B)} 
	= 
	\bigl\| \bD_p^{-\ra} \bigl(\sqrt{\Rmat_p} - \sqrt{\Rmat_p^\eps} \bigr) \bigr\|_{\mathrm{HS}} 
	\leq 
	\bigl\|  \sqrt{\Rmat_p} - \sqrt{\Rmat_p^\eps}  \bigr\|_{2} 
	\bigl\| \bD_p^{-\ra} \bigr\|_{\mathrm{HS}}. 
	\]
	We recall from Proposition~\ref{prop:precon:C}\ref{prop:precon:CLambda}
	and Proposition~\ref{prop:cov-compr}\ref{item:item_ii_Cov_compr} that 
	\[
	\sigma(\Rmat_p) 
	= 
	\sigma\bigl( \bD^{\ra}_p \bC_p \bD_p^{\ra} \bigr) 
	\subset [c_-, c_+] 
	\quad 
	\text{and}
	\quad 
	\sigma\bigl(\Rmat_p^\eps \bigr) 
	= 
	\sigma\bigl( \bD^{\ra}_p \bC^\eps_p \bD_p^{\ra} \bigr) 
	\subset [\widetilde{c}_-, \widetilde{c}_+] . 
	\]
	This allows us to apply a Lipschitz-type estimate 
	for the matrix square root 
	(see, e.g., \cite[Lem.~2.2]{Lipschitz_cont_matrices}), 
	which gives 
	\[
	\textstyle 
	\bigl\|  \sqrt{\Rmat_p} - \sqrt{\Rmat_p^\eps}  \bigr\|_{2} 
	\leq 
	\tfrac{1}{ \sqrt{c_-} + \sqrt{\widetilde{c}_{-}} }
	\bigl\|  \Rmat_p - \Rmat_p^\eps \bigr\|_{2} 
	= 
	\tfrac{1}{ \sqrt{c_-} + \sqrt{\widetilde{c}_{-}} } 
	\bigl\|  \bD_p^{\ra} ( \bC_p - \bC_p^\eps) \bD_p^{\ra} \bigr\|_{2} . 
	\]
	For the norm on the right-hand side, we then obtain 
	\[ 
	\bigl\|  \bD_p^{\ra} ( \bC_p - \bC_p^\eps) \bD_p^{\ra} \bigr\|_{2} 
	=
	\sup_{ \substack{ \xvec \in \bbR^p, \\ \|\xvec\|_2=1 }} 
	\bigr|\xvec^\T \bD_p^{\ra} ( \bC_p - \bC_p^\eps) \bD_p^{\ra} \xvec \bigr|  
	=
	\sup_{ \substack{ \vvec \in \bbR^p, \\ \vvec\neq 0}} 
	\frac{ \bigl| \vvec^\T ( \bC_p - \bC_p^\eps) \vvec \bigr| }{\| \bD_p^{-\ra}\vvec \|_2^2 } 
	\lesssim \eps , 
	\]
	where the last estimate has already  
	been observed in \eqref{eq:Ceps-err} in the proof 
	of Proposition~\ref{prop:cov-compr}\ref{item:item_i_Cov_compr}. 
	Thus, $\text{(B)} \lesssim \eps$. 
	
	Finally, for term (A) we find, for any $s\in[0,\ra-n/2)$, that  
	\[
	\text{(A)}^2   
	= 
	\bbE 
	\Biggl[ 
	\sum_{j > J} \sum_{k\in\nabla_j} | \langle \GP, \psi_{j,k} \rangle |^2 
	\Biggr] 
	\leq 
	2^{-2Js}
	\bbE 
	\Biggl[ 
	\sum_{j \geq j_0} \sum_{k\in\nabla_j} 2^{2js} | \langle \GP, \psi_{j,k} \rangle |^2 
	\Biggr] . 
	\]
	For this reason, 
	regularity of the GRF $\GP$ in $L^2(\Omega;H^s(\cM))$, 
	see \eqref{eq:ZPathRegp}, combined with 
	the second of the norm equivalences in \eqref{eq:Riesz}  
	(recalling the approximation property $\tilde{\gamma} > \ra - n/2$ 
	of the dual basis $\widetilde{\Psi}$)
	show that $\text{(A)} \lesssim  2^{-sJ} 
	\bigl( \bbE\bigl[ \| \GP \|_{H^s(\cM)}^2 \bigr] \bigr)^{1/2}$ 
	for every $s\in[0,\ra-n/2)$. 
	This completes the proof of~\ref{thm:sampling-ii}. 
\end{proof}

\subsection{Multilevel Monte Carlo covariance estimation}
\label{sec:MLMCCovEst}

The estimation of covariance matrices $\Sigmamat_p\in \bbR^{p\times p}$
of Gaussian random variables $\zvec$ taking values in $\bbR^p$ 
from $M$ i.i.d.\ realizations 
of 
$\zvec$ has received attention in recent years 
(e.g.\ \cite{BickLevina08a,BickLevina08b,RBLZ08} and the references there).
Focus in these references has been on incorporating a-priori structural assumptions
on $\Sigmamat_p$, such as bandedness etc. 
Here, we utilize the 
compression patterns 
from Subsection~\ref{subsec:results:sparse} 
(which are universal for pseudodifferential coloring $\cA$
by our results in Section~\ref{section:results}). 

To this end, we estimate blocks of finite sections 
$\bC_{\Lambda_J}$, $\bP_{\Lambda_J}$ 
for the bi-infinite matrix representations \eqref{eq:BiInfMat} 
which resolve the GRF $\GP$ at finite spatial (multi) resolution level $J$, 
i.e., at spatial resolution $\cO(2^{-J})$.
We will directly analyze a multilevel estimator.
The number $p$ of parameters 
(in the usual terminology as, e.g., in \cite{RBLZ08,BickLevina08a,BickLevina08b})
in the truncated MRA representation 
\eqref{eq:GRF_rep_by_dual_MRA}
of samples of $\GP$ is then 
$p = \#(\Lambda_J) = \cO(2^{nJ})$. 

We suppose that we are given $M$ approximate, i.i.d.\ 
samples of the GRF $\GP$ at 
various levels of spatial resolution 
with $p = \cO(2^{nJ})$ parameters at the highest resolution level $J$. 
A plain Monte Carlo approach to sample the corresponding 
covariance matrix would result in computational cost $\cO(Mp)$.
The goal of multilevel Monte Carlo (MLMC) estimation 
is to reduce this computational cost while 
keeping the accuracy consistent:
we aim at a sampling strategy reducing
the cost of $\cO(Mp)$ in certain cases to $\cO(\max\{M,p\})$
with asymptotically the same accuracy.

According to Proposition~\ref{prop:cov-compr}, 
the covariance operator 
$\cC$ of the random field $\GP$ in \eqref{eq:WhNois} satisfies
\[ 
\forall v\in H^{t'}(\cM),w\in H^{t}(\cM):
\quad 
\left| \langle (\cC - \cC^\varepsilon_p)Q_Jw,Q_Jv \rangle \right|
\lesssim 
\varepsilon 2^{J(-2\ra-t-t')} \| w \|_t \| v \|_{t'} . 
\] 
The matrix corresponding to the tapered covariance 
operator $\cC^\varepsilon_p$ may be represented as 
$\Cmat^\varepsilon_p 
= \bbE\bigl[ (\widetilde{\zvec}_p \widetilde{\zvec}_p^\top)^\varepsilon \bigr]$,
with the GRF $\GP$ 
being cast in the dual MRA, 
$\GP = \widetilde{\zvec}^\top \widetilde{\bPsi} 
= 
\sum_{j\geq j_0} \sum_{k\in\nabla_j} \widetilde{z}_{j,k} \widetilde{\psi}_{j,k}$
and $\widetilde{\zvec}_p$ denotes the truncated coefficient vector of $\GP$, 
see~\eqref{eq:GRF_rep_by_dual_MRA}.
In the MLMC sampling algorithm we exploit that 
in wavelet coordinates, the blocks
of the covariance matrix 
need to be approximated with block-dependent threshold 
accuracy in order to obtain a consistent approximation 
of the covariance operator~$\cC$.
For $J\geq j_0$, define the MLMC estimator by 
\[ 
\Cmat_p^\varepsilon 
\approx 
E^{*}_J(\Cmat_p^\varepsilon )
:=
\sum_{j,j'=j_0}^J
E_{M_{j,j'}}(\Cmat^\varepsilon _{\rm global}(j,j')) 
.
\] 
Here, $\Cmat^\varepsilon(j,j')$ 
is the section of $\Cmat^\varepsilon $ corresponding to 
$\{(j,k) : k\in\nabla_j\} \times \{(j',k') : k'\in \nabla_{j'}\}$ 
and 
$\Cmat^\varepsilon _{\rm global}(j,j')$ 
is the respective global matrix with zeros at indices that are not in
$\{(j,k) : k\in\nabla_j\} \times \{(j',k') : k'\in \nabla_{j'}\}$.
Furthermore, 
for $j,j' \in \{ j_0,\ldots,J \}$, 
$E_{M_{j,j'}}$ 
denotes a Monte Carlo estimator with $M_{j,j'}$ samples.
More specifically, 
the Monte Carlo estimator 
$E_{M_{j,j'}}(\Cmat^\varepsilon _{\rm global}(j,j'))$
is realized by $M_{j,j'}$ i.i.d.\  
samples of the coefficient vector $\widetilde{\zvec}$ 
at discretization levels $j,j'$ of spatial resolution, 
i.e.,
\[ 
\Cmat^\varepsilon _p(j,j')
\approx
E_{M_{j,j'}}(\Cmat^\varepsilon (j,j')) 
:=
\frac{1}{M_{j,j'}}
\sum_{i=1}^{M_{j,j'}}
\bigl( \widetilde{\zvec}_{i}(j) \widetilde{\zvec}_{i}(j')^\top \bigr)^\varepsilon 
,
\] 
where $\widetilde\zvec(j'')$
is the restriction of the coordinate vector to the coordinates with indices
in $\{(j'',k): k\in \nabla_{j''}\}$.
The operator that corresponds to 
the MLMC estimator $E_J^*(\bC_p^\varepsilon)$ will be denoted 
by $E_J^*(\cC_p^\varepsilon)$, i.e., 
\[
\forall \lambda, \lambda' \in \Lambda_{J}:
\quad
\langle E_J^*(\cC_p^\varepsilon) \psi_\lambda, \psi_{\lambda'} \rangle
= 
(E_J^*(\bC_p^\varepsilon))_{\lambda,\lambda'}. 
\] 
Recall that $\Bmat^\varepsilon $ is the tapered version  
of some 
matrix $\Bmat$,  
as defined in Definition~\ref{def:tapering}.  
We suppose that we are given samples, 
which are independent realizations of $\GP$ 
at multiple scales of resolution, 
expressed in terms of the coordinate vector
\[ 
\bigl\{
\widetilde{\zvec}_{i}^{j_0}: i = 1, \ldots, M_{0}
\bigr\}, 
\ldots
,
\bigl\{
\widetilde{\zvec}_{i}^{J}: i = 1, \ldots, M_{J}
\bigr\}, 
\] 
where $\widetilde{\zvec}^{j}$ denotes the truncation 
of the coordinate vector $\widetilde{\zvec}$
to coordinates with indices in 
$\{(j',k'): j_0\leq j'\leq j, k'\in\nabla_{j'}\}$. 
In this setting, the sample numbers $M_{j,j'}$ are given by
\begin{equation}\label{eq:Mjj'}
M_{j,j'}
:=
\widetilde{M}_{\max\{j,j'\}}, 
\quad 
\text{where} 
\quad 
\widetilde{M}_{j}
:=
\sum_{j'= j}^J
M_{j'}
. 
\end{equation}
\begin{proposition}\label{prop:mlmc_error_est}
	Suppose  
	Assumptions~\ref{ass:cM-and-cA}\ref{ass:cM-and-cA_I}--\ref{ass:cM-and-cA_II} 
	hold for some $\ra>n/2$.
	Let further the assumptions of Proposition~\ref{prop:cov-compr}
	hold with wavelet and dual wavelet parameters 
	$d,\widetilde{d}$ such that $d < \widetilde{d} - 2\ra$.
	
	Then,
	for any $\beta< \ra - n/2$ and $-\ra \leq t,t' \leq d$, 
	there exists a constant $C>0$
	such that the multilevel Monte Carlo estimator 
	$E_J^*(\cC^\varepsilon_p)$ 
	with sample numbers \eqref{eq:Mjj'} 
	satisfies the error bound 
	\begin{align*}
	\Biggl\| 
	\sup_{ u\in H^t(\cM) \setminus\{0\} }
	&\sup_{ v\in H^{t'}(\cM)\setminus\{0\} }
	\frac{|\langle({\cC}^\varepsilon_p - E^{*}_J({\cC}^\varepsilon_p))Q_J u , Q_J v \rangle|}{
		\|u\|_{H^t(\cM)} \|v\|_{H^{t'}(\cM)} } 
	\Biggr\|_{L^2(\Omega)}
	\\
	&\leq 
	\frac{2C}{1-2^{-(\min\{t,t'\}+\beta)}}
	\sum_{j=j_0}^{J}
	\frac{1}{\sqrt{\widetilde{M}_{j}}} 
	2^{-j(\min\{t,t'\}+\beta)}
	\left\|\GP \right\|^2_{L^4(\Omega;H^\beta(\cM))} 
	.
	\end{align*}
\end{proposition}

\begin{proof}
	By the estimate in~\cite[Equation~(9.3)]{DHSWaveletBEM2007} 
	(also exploiting the estimates~\cite[Equations~(4.3) and~(4.2)]{DHSWaveletBEM2007})
	\begin{align*} 
	\text{(I)} 
	&:=
	\Biggl\| 
	\sup_{ u\in H^t(\cM) \setminus\{0\} }
	\sup_{ v\in H^{t'}(\cM)\setminus\{0\} }
	\frac{|\langle({\cC}^\varepsilon _p - E^{*}_J(\cC^\varepsilon _p))Q_J u , Q_J v \rangle|}
	{\|u\|_{H^t(\cM)} \|v\|_{H^{t'}(\cM)}}
	\Biggr\|_{L^2(\Omega)}
	\\
	&\leq 
	\Biggl\| 
	\sum_{j,j'=j_0}^{J}
	2^{- jt } 2^{- j't' }
	\bigl\| \bC^\varepsilon_p(j,j') - E_{M_{j,j'}}(\bC^\varepsilon(j,j')) \bigr\|_2 
	\Biggr\|_{L^2(\Omega)} 
	\\
	&\leq 
	\sum_{j,j'=j_0}^{J}
	2^{- jt } 2^{- j't' }
	\Bigl\| 
	\bigl\| \bC^\varepsilon_p(j,j') - E_{M_{j,j'}}(\bC^\varepsilon(j,j')) \bigr\|_{\rm HS} 
	\Bigr\|_{L^2(\Omega)} 
	\\
	&\leq 
	\sum_{j,j'=j_0}^{J}
	\frac{1}{\sqrt{M_{j,j'}}}
	2^{-jt} 2^{-j't'} 
	\Bigl\| 
	\bigl\|\widetilde{\zvec}(j) \widetilde{\zvec}(j')^\top \bigr\|_{\mathrm{HS}} 
	\Bigr\|_{L^2(\Omega)}, 
	\end{align*} 
	where we used that the operator matrix norm 
	with respect to the Euclidean norm  
	is upper bounded by the 
	Hilbert--Schmidt (or Frobenius) norm. 
	The Frobenius norm satisfies that 
	$\|w(w')^\top\|_{{\mathrm{HS}}} \leq \|w\|_2 \|w'\|_2$ 
	for all $w\in \bbR^m$, $w'\in \bbR^{m'}$, $m,m'\in\bbN$.
	Also note that by~\eqref{eq:Riesz}, 
	$\|\widetilde{\zvec}(j)\|_2 \lesssim 2^{-j\beta} \|\GP\|_{H^\beta(\cM)}$.
	Thus, 
	\[ 
	\text{(I)} 
	\leq 
	C
	\sum_{j,j'=j_0}^{J}
	\frac{1}{\sqrt{M_{j,j'}}}
	2^{-j(t+\beta)} 2^{-j'(t'+\beta)} \left\|\GP \right\|^2_{L^4(\Omega;H^\beta(\cM))}.
	\] 
	Furthermore, 
	\[   
	\sum_{j,j'=j_0}^{J}
	\frac{1}{\sqrt{M_{j,j'}}}
	2^{-j(t+\beta)} 2^{-j'(t'+\beta)}
	=
	\sum_{\bar{j}=j_0}^{J}
	\frac{1}{\sqrt{\widetilde{M}_{\bar{j}}}}
	\sum_{j,j': \max\{j,j'\} = \bar{j}} 2^{-j(t+\beta)} 2^{-j'(t'+\beta)}
	\] 
	and
	\begin{align*}
	\sum_{j,j': \max\{j,j'\}=\bar{j}}
	2^{-j(t+\beta)} 
	&\, 2^{-j'(t'+\beta)}
	\leq
	2^{-\bar{j}(t+\beta)} 
	\sum_{j'=0}^{\bar{j}} 2^{-j'(t'+\beta)}
	+
	2^{-\bar{j}(t'+\beta)} 
	\sum_{j=0}^{\bar{j}} 2^{-j(t+\beta)}
	\\
	&\leq 
	\frac{2^{-\bar{j}(t+\beta)}}{1-2^{-(t'+\beta)}} 
	+
	\frac{2^{-\bar{j}(t'+\beta)}}{1-2^{-(t+\beta)}}
	\leq 2 \frac{2^{-\bar{j}(\min\{t,t'\}+\beta)}}{1-2^{-(\min\{t,t'\}+\beta)}} 
	.
	\end{align*}
	In conclusion the asserted estimate follows, i.e.,
	\[ 
	\text{(I)} 
	\leq 
	\frac{2C}{1-2^{-(\min\{t,t'\}+\beta)}}
	\sum_{\bar{j}=j_0}^{J}
	\frac{1}{\sqrt{\widetilde{M}_{\bar{j}}}} 2^{-\bar{j}(\min\{t,t'\}+\beta)}
	\left\|\GP \right\|^2_{L^4(\Omega;H^\beta(\cM))} 
	.\qquad \qedhere
	\]
\end{proof}

The required computational cost of the estimator $E^*_J$ is 
\begin{equation}\label{eq:comp_cost_mlmc}
\textstyle 
{\rm work} 
=
\cO 
\biggl(\sum\limits_{j=j_0}^J \widetilde{M}_{j} 2^{jn} \biggr) 
\end{equation}
and by Propositions~\ref{prop:cov-compr} and~\ref{prop:mlmc_error_est} 
the accuracy is
\begin{equation}\label{eq:error_mlmc}
\textstyle 
{\rm error} 
= 
\cO
\biggl( 
2^{-J\alpha_0}
+
\sum\limits_{j=j_0}^J
\widetilde{M}_j^{-1/2}
2^{-j \alpha}
\biggr),
\end{equation}
where $\alpha\leq \alpha_0 \leq 2\ra + t + t'$ 
and 
$\alpha= \ra - n/2 - \varepsilon_0 + \min\{t,t'\}$
and where we inserted 
$\beta = \ra - n/2 - \varepsilon_0$ 
for arbitrary small $\varepsilon_0 > 0$.
It remains to choose the sample numbers $\widetilde{M}_j$ 
and equivalently the sample numbers 
$M_j$, $j=j_0,\ldots,J$, in such a way to optimize accuracy versus computational cost.
This has been considered in the context of multilevel integration methods
and GRFs, e.g.,~\cite{HS_2019}.
Following this reference, 
we choose the following sample numbers
\begin{equation*}
\widetilde{M}_j
=
\Bigl\lceil 
\widetilde{M}_{0} 
2^{-j(n+\alpha)2/3}
\Bigr\rceil, 
\quad j=j_0,\ldots,J,
\end{equation*}
and 
\begin{equation*}
\widetilde{M}_{j_0}
=
\begin{cases}
2^{J2\alpha_0}, & \text{if } 2\alpha>n, \\
2^{J2\alpha_0} J^2,  & \text{if } 2\alpha = n, \\
2^{J(2\alpha_0 +2n/3 - 4\alpha/3)}, & \text{if } 2\alpha< n.
\end{cases}
\end{equation*}
The overall computational cost is 
\begin{equation*}
{\rm work} 
=
\begin{cases}
\cO(2^{J2\alpha_0}), & \text{if } 2\alpha>n, \\
\cO(2^{J2\alpha_0} J^3), & \text{if } 2\alpha = n, \\
\cO(2^{J(n-2(\alpha_0 - \alpha))}), & \text{if } 2\alpha< n. 
\end{cases}
\end{equation*}

The proof of the following theorem is postponed 
to Appendix~\ref{appendix:sample_numbers}.
\begin{theorem}\label{thm:mlmc_cov_estimation}
	Let the assumptions of Proposition~\ref{prop:mlmc_error_est}
	be satisfied.
	In addition, let 
	$ \alpha_0 \in [\alpha, 2\ra +t +t']$
	for $\alpha < \ra-n/2 +\min\{t,t'\}$.
	
	An error threshold $\varepsilon>0$ may be achieved, 
	i.e.,
	\[ 
	\Biggl\| 
	\sup_{ u\in H^t(\cM) \setminus\{0\} }
	\sup_{ v\in H^{t'}(\cM)\setminus\{0\} }
	\frac{|\langle({\cC}^\varepsilon_p - E^{*}_J({\cC}^\varepsilon_p))Q_J u , Q_J v \rangle|}{
		\|u\|_{H^t(\cM)} \|v\|_{H^{t'}(\cM)} } 
	\Biggr\|_{L^2(\Omega)}
	= 
	\cO(\varepsilon)
	\] 
	with computational cost 
	\[ 
	{\rm work} =
	\begin{cases}
	\cO(\varepsilon^{-2}) & \text{if } 2\alpha>n, \\
	\cO(\varepsilon^{-2} |\log(\varepsilon^{-1})|) & \text{if } 2\alpha=n ,\\
	\cO(\varepsilon^{-(n/\alpha_0 - 2(1-\alpha/\alpha_0))}) & \text{if } 2\alpha<n . \\
	\end{cases}
	\] 
\end{theorem}

\begin{remark} 
	The results of Proposition~\ref{prop:cov-compr} on the compression 
	of the covariance matrix can, of course, also be used  
	in combination with \emph{single-level} Monte Carlo estimation 
	by computing only those entries of the sample covariance matrix  
	which are needed according to the tapering scheme 
	\eqref{eq:AprCompr}--\eqref{eq:AprCompBjj}. 
\end{remark}

\begin{remark}	
	The MLMC convergence results of this section hold in the 
	root mean squared sense. 
	Bounds that hold in probability could also be derived.
	For the case of \emph{single-level} Monte Carlo estimation with $M$ samples, 
	a computational cost estimate of $\cO(M p  ) $ 
	follows readily by~\cite[Lem.~A.3]{BickLevina08a}.
	Specifically~\cite[Lem.~A.3]{BickLevina08a} 
	(where convergence in probability is derived based on \cite{SauStat91})
	may be applied to the preconditioned compressed covariance matrix 
	$\bD_p^{\ra} \bC_p^\eps \bD_p^{\ra} 
	= 
	(\bD_p^{\ra} \bC_p \bD_p^{\ra})^\eps 
	= 
	\bbE\bigl[ (\bD_p^{\ra}\,\widetilde{\zvec}_p) (\bD_p^{\ra}\,\widetilde{\zvec}_p)^\top \bigr]$. 
	This matrix satisfies the assumptions of~\cite[Lem.~A.3]{BickLevina08a}, 
	since it is uniformly 
	well-conditioned. This is a consequence of Proposition~\ref{prop:cov-compr}(ii).
	Similarly, the use of wavelet coordinates will imply $p$-uniform
	bounds in several classes of regression methods. 
	Bounds of covariance estimators that hold in probability may be of interest 
	when certified bounds on the condition number of the estimator are required. 
	For example when the estimator of the covariance matrix is further 
	used inside an iterative solver for linear systems
	to approximate the precision matrix.
\end{remark} 

\subsection{Spatial prediction in statistics}
\label{subsec:krging}

Optimal linear prediction of random fields
which is also known as ``kriging'',
is a widely used methodology 
in spatial statistics for interpolating spatial data 
subject to uncertainty (see, e.g., \cite{stein99} and the references there).
We note that the kriging predictor can be regarded 
as an orthogonal projection in $L^2(\Omega)$ 
onto the finite-dimensional subspace 
generated by the observations. 
Thus, the theory for kriging without observation noise
may be formulated in an infinite-dimensional setting, 
with a separable Hilbert space as state space of a 
GRF, see e.g.\ \cite{owhadi2015conditioning}.  
For the computational algorithm 
discussed in this section 
we shall consider the GRF $\GP$ 
defined through the SPDE \eqref{eq:WhNois} 
and its (bi-infinite) covariance matrix $\bC\in\bbR^{\bbN\times\bbN}$ 
represented in the MRA $\bPsi$, 
which is truncated to a finite dimension $p$, see \eqref{eq:p=p(J)}, 
$\bC\approx\bC_p\in\bbR^{p\times p}$. 

A typical model in applications is 
to assume that $\GP$ is observed at $K$ distinct
spatial locations $\{x_i\}_{i=1}^K\subset \cM$ under 
i.i.d.\ centered Gaussian measurement noise: 
\[
y_i = \GP(x_i) + \eta_i , 
\quad 
i=1,\ldots,K, 
\qquad 
\eta_i \sim \normal(0,\sigma^2) 
\quad 
\text{i.i.d.} 
\]
One is now interested in predicting the field $\GP$ 
at an unobserved location $x_*\in\cM$ 
(or at several locations)
conditioned on the observations $\{y_i\}_{i=1}^K$. 
In other words, one needs to calculate the 
posterior mean $\bbE[\GP(x_*)|y_1 ,\ldots,y_K]$. 
However, this task turns out to be computationally challenging 
as, assuming a finite spatial resolution of dimension $p$ 
for approximating the GRF $\GP$, 
direct approaches to solve the arising linear systems of equations
entail computational costs 
which are cubic either in $K$ or in $p$ or in both. 

In this section we address how 
the multiresolution representation of the covariance 
and of the precision matrices of $\GP$ in the MRA $\bPsi$
allow an \emph{approximate, compressed kriging process} 
whereby the matrices and vectors are numerically sparse
due to the cancellation properties of the MRAs.
For $p\in \bbN$, 
we truncate the bi-infinite covariance matrix $\bC$ 
of the GRF $\GP$ in the MRA $\bPsi$ to the 
``finite-section'' matrix
$\bC_{p}\in\bbR^{p\times p}$ 
using the index set $\Lambda_J\subset  \mathcal{J}$, see~\eqref{eq:p=p(J)}, 
where $p= \# (\Lambda_J)$. 
Also, we consider an abstract setting with functionals 
$g_1,\ldots,g_K$ 
which gives us the model 
\[
\yvec = \mathbf{G} \widetilde{\zvec} + \boldsymbol{\eta}, 
\] 
where  
$\yvec = (y_1,\ldots,y_K)^\T$ 
is the random vector corresponding to the observations, 
$\mathbf{G} \in\bbR^{K\times p}$ 
is the observation matrix with entries 
$G_{i(j,k)} := \langle g_i, \widetilde \psi_{j,k} \rangle$, 
and $\widetilde{\zvec}$, $\boldsymbol{\eta}$ 
are centered multivariate Gaussian 
distributed random vectors with covariance matrices 
$\bC_{p}\in\bbR^{p\times p}$ and 
$\sigma^2 \mathbf{I}\in\bbR^{K\times K}$, respectively.
We recall that the GRF 
$\GP = \widetilde{\zvec}^\top \widetilde{\bPsi}$
is represented in the dual MRA $\widetilde{\bPsi}$, 
see~\eqref{eq:GRF_rep_by_dual_MRA}. 
Admissible choices for the functionals $g_i$ 
are local averages around points $x_i\in\cM$.
The joint distribution  
of $\widetilde{\zvec}$ and $\yvec$ is thus given by 
\[
\begin{pmatrix} 
\widetilde{\zvec} \\
\yvec
\end{pmatrix} 
\sim \normal \left(
\begin{pmatrix} 
\zerovec \\
\zerovec
\end{pmatrix}, 
\begin{pmatrix} 
\bC_{p} & \bC_{p} \mathbf{G}^\T \\
\mathbf{G}\bC_{p} & \mathbf{G} \bC_{p} \mathbf{G}^\T + \sigma^2 \mathbf{I} 
\end{pmatrix}
\right). 
\]
Then, the law of the posterior 
$\widetilde{\zvec}| \yvec$ is again Gaussian and 
the kriging predictor is given by the posterior mean, namely, 
\begin{equation}
\label{eq:krig-pred_0} 
\boldsymbol{\mu}_{\widetilde{\zvec}|\yvec} 
= 
\bC_{p} \mathbf{G}^\T 
\left( \mathbf{G} \bC_{p} \mathbf{G}^\T +\sigma^2 \mathbf{I} \right)^{-1} 
\yvec . 
\end{equation}

In what follows, we will address how  
the  posterior mean in \eqref{eq:krig-pred_0}
can be approximately realized with low computational cost
when represented in 
the MRA~$\widetilde{\bPsi}$ exploiting wavelet compression techniques.

We will proceed in two steps. 
First, we will analyze the computational cost for approximately computing the posterior mean. 
Secondly, 
we estimate the consistency error incurred by the compression of the covariance matrix.

The main challenge is the efficient numerical 
evaluation of 
$\left( 
\mathbf{G} 
\bC_p \mathbf{G}^\T +\sigma^2 \mathbf{I} \right)^{-1} 
\yvec $.
It will be approximated numerically 
by the conjugate gradient (CG) method
applied to approximately solve the linear system
to find $\vvec$
such that $( 
\mathbf{G} 
\bC_p \mathbf{G}^\T 
+\sigma^2 \mathbf{I} )  \vvec = \yvec$.
It is well-known that after $N$ iterations of CG 
to approximately solve the linear system $\bA \wvec = \fvec$
by $\wvec^N \in {\mathbb R}^K$
for a SPD matrix $\bA$ starting from the initial guess being 
the zero vector, it holds \cite[Thm.\ 10.2.6]{GolubvLoan4th}
with $\| \wvec \|_{\mathbf{A}}^2 := \wvec^\top \mathbf{A} \wvec$ that
\begin{equation}
\label{eq:est_CG}
\bigl\| \wvec - \wvec^N \bigr\|_{{\mathbf{A}}}
\leq 
2
\left( \frac{ \sqrt{ \operatorname{cond}_2(\mathbf{A})} - 1}{ 
	\sqrt{ \operatorname{cond}_2(\mathbf{A}) } + 1} \right)^N
\|\wvec\|_{\mathbf{A}}.
\end{equation}
To estimate the condition number of the matrix 
$\mathbf{A} := \mathbf{G} \bC_p \mathbf{G}^\top + \sigma^2 \mathbf{I} $, 
we observe 
\[ 
\forall \vvec \in \bbR^K: \quad 
\vvec^\top
\mathbf{G} \bC_p \mathbf{G}^\top
\vvec 
\geq 0
\;
.
\] 
On the other hand, 
by Proposition~\ref{prop:precon:C}\ref{prop:precon:CLambda}
\[ 
\forall \vvec \in\bbR^K: \quad 
\vvec^\top
\mathbf{G} \bC_p \mathbf{G}^\top
\vvec 
\leq 
c_+
\vvec^\top
\mathbf{G}  \mathbf{G}^\top
\vvec 
.
\] 
We assume that $g_i\in L^2(\cM)$, $i=1,\ldots,K$, 
and that they have \emph{disjoint supports}, 
i.e., $\mu({\rm supp}(g_i) \cap   {\rm supp}(g_{i'}))=0$
for any $i,i' = 1,\ldots,K$ such that $i\neq i'$. 
We obtain with~\eqref{eq:Riesz} that, 
for every $\vvec\in \bbR^K$, 
\begin{align*} 
\vvec^\top
\mathbf{G}  \mathbf{G}^\top
\vvec
&=
\sum_{i,i'=1}^K
\sum_{j,k}
v_i 
\langle g_i , \widetilde{\psi}_{j,k} \rangle  
\langle g_{i'} , \widetilde{\psi}_{j,k} \rangle 
v_{i'}
\\
& = 
\sum_{j,k}
\biggl|
\sum_{i=1}^K 
v_i 
\langle  g_i , \widetilde{\psi}_{j,k}
\rangle 
\biggr|^2
=
\sum_{j,k}
\biggl| 
\biggl\langle 
\sum_{i=1}^K 
v_ig_i , \widetilde{\psi}_{j,k}
\biggr\rangle 
\biggr|^2
\simeq 
\biggl\|
\sum_{i=1}^K 
v_ig_i 
\biggr\|_{L^2(\cM)}^2. 
\end{align*} 
The disjoint support property 
of $g_1,\ldots,g_K$ implies that
\[ 
\biggl\|
\sum_{i=1}^K 
v_i g_i 
\biggr\|_{L^2(\cM)}^2
=
\sum_{i=1}^K
v_i^2
\|g_i\|_{L^2(\cM)}^2
\leq
\|\vvec \|^2_2
\max_{i=1,\ldots,K}
\bigl\{ \|g_i\|_{L^2(\cM)}^2 \bigr\}. 
\] 
Thus, there exists a constant $C>0$ that depends 
neither on $K$ nor on $p$ 
such that for every $\vvec\in \bbR^K$
\[ 	
\vvec^\top
\mathbf{G}  
\bC_p \mathbf{G}^\top
\vvec
\leq 
C
\|\vvec \|^2_2
\max_{i=1,\ldots,K}
\bigl\{ \|g_i\|_{L^2(\cM)}^2 \bigr\}.
\] 
We conclude that, for every $\vvec\in \bbR^K$, 
\begin{equation}
\label{eq:est_spectrum_kriging}
\sigma^2 \|\vvec \|_2^2
\leq
\vvec^\top 
(
\mathbf{G}  
\bC_p \mathbf{G}^\top
+ 
\sigma^2\mathbf{I}
)
\vvec
\leq 
\Bigl(C\max_{i=1,\ldots,K}
\bigl\{ \|g_i\|_{L^2(\cM)}^2 \bigr\}
+\sigma^2 
\Bigr)
\|\vvec \|^2_2, 
\end{equation}
which implies that
\begin{equation}
\label{eq:cond_number_kriging}
\operatorname{cond}_2
\bigl(\mathbf{G}  
\bC_p \mathbf{G}^\top
+ 
\sigma^2\mathbf{I} \bigr)
\leq 
\frac{ C \max_{i=1,\ldots,K}
	\bigl\{ \|g_i\|_{L^2(\cM)}^2 \bigr\}
}{\sigma^2} + 1
.
\end{equation}
The argument applies verbatim to the compressed matrix $\bC^\varepsilon_p$, 
i.e., 
\begin{equation}
\label{eq:cond_number_kriging_eps}
\operatorname{cond}_2
\bigl(\mathbf{G}  
\bC_p^{\varepsilon} \mathbf{G}^\top
+ 
\sigma^2\mathbf{I} \bigr)
\leq 
\frac{C\max_{i=1,\ldots,K}
	\bigl\{ \|g_i\|_{L^2(\cM)}^2 \bigr\}
}{\sigma^2} + 1
.
\end{equation}
Let us denote by
\begin{equation}
\label{eq:krig-pred_0_compr}
\boldsymbol{\mu}_{\widetilde{\zvec}|\yvec}^{\varepsilon}
:=
\bC_{p}^\varepsilon \mathbf{G}^\T 
\left( 
\mathbf{G} 
\bC_{p}^\varepsilon \mathbf{G}^\T 
+\sigma^2 \mathbf{I} 
\right)^{-1} 
\yvec
\end{equation}
the posterior mean that results
from the compressed covariance matrix $\bC_p^\varepsilon$.

Furthermore, let 
$\vvec^N$ be 
the result of $N$ iterations of CG to approximately solve 
the linear system 
$(\mathbf{G} 
\bC_p^{\varepsilon} \mathbf{G}^\T 
+\sigma^2 \mathbf{I} )  \vvec = \yvec$.
Then, by~\eqref{eq:cond_number_kriging}, \eqref{eq:est_spectrum_kriging}
and by~\eqref{eq:est_CG}
\[ 
\bigl\| \vvec - \vvec^N \bigr\|_2
\leq 
2
\kappa
\left(
\frac{\sqrt{\kappa}-1}{\sqrt{\kappa}+1} 
\right)^N
\|\vvec\|_2
,
\] 
where $\kappa := \sigma^{-2}C\max_{i=1,\ldots,K}
\{ \|g_i\|_{L^2(\cM)}^2\}
+1 $.
Furthermore, we denote by 
$\boldsymbol{\mu}_{\widetilde{\zvec}|\yvec}^{\varepsilon,N} 
: = \bC_p^\varepsilon \mathbf{G}^\top \vvec^N$
and observe that 
\begin{equation}
\label{eq:kriging_CG_error}
\bigl\| 
\boldsymbol{\mu}^\varepsilon_{\widetilde{\zvec}|\yvec} 
- 
\boldsymbol{\mu}_{\widetilde{\zvec}|\yvec}^{\varepsilon,N} 
\bigr\|_2
\leq 
2
\kappa
\left(
\frac{\sqrt{\kappa}-1}{\sqrt{\kappa}+1}\right)^N
\| \bC_p^\varepsilon  \mathbf{G}^\top \|_2
\| \vvec\|_2
.
\end{equation} 
\begin{theorem}\label{thm:SprseKrgCst}
	The computational cost of 
	$\boldsymbol{\mu}_{\widetilde{\zvec}|\yvec}^{\varepsilon,N}$
	to achieve a consistency error for any $\delta\in (0,1)$
	\[
	\bigl\| 
	\boldsymbol{\mu}^\varepsilon_{{\widetilde{\zvec}}|\yvec} 
	- 
	\boldsymbol{\mu}_{\widetilde{\zvec}|\yvec}^{\varepsilon,N} 
	\bigr\|_2 
	= \cO(\delta)
	\] 
	is 
	$\cO ( (K\log(p) +p) \sigma^{-1}\log(\delta^{-1} \sigma^{-2}) )$, 
	where 
	\[
	N\gtrsim  \sigma^{-1}\log(\delta^{-1} \sigma^{-2}).
	\]
\end{theorem}

\begin{proof}
	By elementary manipulations, we observe that the made choice on the number of iterations
	in CG $N$, 
	guarantees the claimed consistency error by~\eqref{eq:kriging_CG_error}.
	
	The matrix $\bC_p^\varepsilon$ has $\cO(p)$ non-zero entries and
	the matrix $\mathbf{G}$ has $\cO(K\log(p))$ many non-zero entries.
	This implies that the application of the matrix $\mathbf{G} 
	\bC_p^{\varepsilon} \mathbf{G}^\top$ to a vector 
	has computational cost $\cO(K\log(p) + p)$,
	which is required $N$ times to compute $\vvec^N$.
	The computational cost of the application 
	of $\bC_p^\varepsilon \mathbf{G}^\top$ is again $\cO(p+ K\log(p))$. 
	Thus, the claimed estimate on the computational cost follows.
\end{proof}
\begin{remark} \label{rmk:OK+p}
	As single-scale basis functions have supports proportional
	to the step size, the observation matrix $\mathbf{G}$ has only
	$\cO(K)$ many, nonzero entries when computed with 
	respect to the single-scale basis $\{ \widetilde{\phi}_{j,k}\}$ 
	in comparison with $\cO(K + \log(p))$ many nonzero entries 
	when computed with 
	respect to the wavelet basis $\{\widetilde{\psi}_{j,k}\}$. 
	Denoting
	by ${\bf T}_{\widetilde{\phi}\to\widetilde{\psi}}$
	the dual fast wavelet transform, 
	both versions of the observation matrix are interconnected by
	${\bf G}_{\widetilde{\phi}}{\bf T}_{\widetilde{\phi}\to\widetilde{\psi}}^{\T} 
	= {\bf G}_{\widetilde{\psi}}$. 
	Consequently, as the fast wavelet transform
	is of linear complexity, computing the action of 
	$\mathbf{G}_{\widetilde{\psi}}  
	\bC_p^{\varepsilon} \mathbf{G}_{\widetilde{\psi}}^\top$
	to a vector via
	\begin{equation}\label{eq:kriging_matrix}
	{\bf G}_{\widetilde{\phi}}{\bf T}_{\widetilde{\phi}\to\widetilde{\psi}}^{\T}
	\bC_p^{\varepsilon} 
	{\bf T}_{\widetilde{\phi}\to\widetilde{\psi}}\mathbf{G}_{\widetilde{\phi}}^\T
	\end{equation}
	reduces the complexity from $\cO (K\log(p)+p)$ to $\cO (K+p)$.
	An illustration of this matrix product is found in Figure~\ref{fig:kriging},
	see Subsection~\ref{subsct:kriging} for the details.
\end{remark}
We estimate the error 
$\|\boldsymbol{\mu}_{{\widetilde{\zvec}} | \yvec} 
- \boldsymbol{\mu}_{{\widetilde{\zvec}} |\yvec}^{\varepsilon}\|_2 $
incurred by using the 
compressed covariance matrix $\bC_p^\varepsilon$. 
\begin{proposition}
	Let the assumptions of Proposition~\ref{prop:cov-compr} hold.
	Recall that $p=p(J)$ for $J\geq j_0$.
	
	Then, 
	there exists a constant $C>0$ independent of $J$
	such that
	\[ 
	\bigl\|
	\boldsymbol{\mu}_{{\widetilde{\zvec}}|\yvec}  
	- 
	\boldsymbol{\mu}_{{\widetilde{\zvec}} |\yvec}^{\varepsilon} 
	\bigr\|_2
	\leq 
	C 
	\sigma^{-4}
	\|\mathbf{y}\|_2 
	\, 
	2^{-2\ra J}.
	\] 
\end{proposition}

\begin{proof}
	The result is an elementary bound obtained from the $2$-norm of SPD matrices
	in terms of their spectrum.
	For any two symmetric, 
	positive semi-definite matrices $\mathbf{A},\mathbf{B}\in \bbR^{K\times K}$ 
	it holds that 
	\[ 
	\|
	(\mathbf{A} + \sigma^2 \mathbf{I})^{-1}
	-
	(\mathbf{B} + \sigma^2 \mathbf{I})^{-1}
	\|_2
	\leq 
	\sigma^{-4}
	\|\mathbf{A}-\mathbf{B}\|_2.
	\] 
	This can be seen by
	\[ 
	\mathbf{A}_\sigma^{-1} - \mathbf{B}_\sigma^{-1}
	=
	\mathbf{A}_\sigma^{-1}\mathbf{B}_\sigma \mathbf{B}_\sigma^{-1} - \mathbf{B}_\sigma^{-1}
	=
	(\mathbf{A}_\sigma^{-1}\mathbf{B}_\sigma - \mathbf{I})\mathbf{B}_\sigma^{-1}
	\] 
	and 
	\[ 
	\mathbf{A}_\sigma^{-1}\mathbf{B}_\sigma - \mathbf{I}
	=
	\mathbf{A}_\sigma^{-1}\mathbf{B}_\sigma - \mathbf{A}_\sigma^{-1}\mathbf{A}_\sigma
	=
	\mathbf{A}_\sigma^{-1} (\mathbf{B}_\sigma - \mathbf{A}_\sigma), 
	\] 
	which implies
	\[  
	\|\mathbf{A}_\sigma^{-1} - \mathbf{B}_\sigma^{-1}\|_2
	\leq 
	\|\mathbf{A}_\sigma^{-1}\|_2 
	\|\mathbf{B}_\sigma - \mathbf{A}_\sigma\|_2 
	\|\mathbf{B}_\sigma^{-1}\|_2 
	= 
	\|\mathbf{A}_\sigma^{-1}\|_2 
	\|\mathbf{B}_\sigma^{-1}\|_2  
	\|\mathbf{A} - \mathbf{B} \|_2
	, 
	\] 
	where $\mathbf{A}_\sigma := \mathbf{A} +\sigma^2\mathbf{I}$
	and $\mathbf{B}_\sigma := \mathbf{B} +\sigma^2\mathbf{I}$.
	The assertion of the proposition now follows by 
	Proposition~\ref{prop:cov-compr} with \eqref{eq:krig-pred_0} 
	and \eqref{eq:krig-pred_0_compr}.
\end{proof}

\section{Numerical experiments}\label{sec:numexp}

\subsection{Preliminary remarks and settings}

For the numerical illustration of our results,
we shall consider the boundary of the domain shown  
in Figure~\ref{fig:domain}. 
It is given by the $2\pi$-periodic, analytic parametrization
\[
\gamma:[0,2\pi]\to\Gamma = \partial\Omega,\quad
\gamma(\phi) = g(\phi)\begin{bmatrix}\cos(\phi)\\\sin(\phi)\end{bmatrix},
\]
where 
\[
g(\phi) 
= 
\alpha_0 + \frac{1}{100}\sum_{k=1}^5 
\bigl( \alpha_{-k} \sin(k\phi) + \alpha_k \cos(k\phi) \bigr)  
\]
is a finite Fourier series with the following coefficients:

\begin{center}{\small 
		\begin{tabular}{llllll} 
			$\alpha_{-5} = 2.2$, 
			&
			$\alpha_{-4} = 0.56$, 
			& 
			$\alpha_{-3} = 0.14$, 
			&
			$\alpha_{-2} = 1.1$, 
			& 
			$\alpha_{-1} = 1.4$, 
			& 
			$\alpha_0 = 50$,
			\\
			$\alpha_5 =0.89$, 
			& 
			$\alpha_4 =-1.5$,  
			& 
			$\alpha_3 =-1.2$, 
			& 
			$\alpha_2 =-1.5$, 
			&
			$\alpha_1 = -0.57$. 
			& 
		\end{tabular}
}\end{center}

The covariance kernels under consideration are from 
the Mat\'ern family \cite{HW,matern1960}, namely\footnote{Here, 
	we represented the Mat\'ern kernel 
	$k_\nu(z) = \frac{2^{1-\nu} \sigma^2}{\Gamma(\nu)} \bigl( \sqrt{2\nu} \frac{z}{\ell} \bigr)^{\nu} 
	K_{\nu}\bigl(\sqrt{2\nu} \frac{z}{\ell} \bigr)$,  
	with $\sigma^2=1$, as a product of an exponential and a polynomial 
	which is possible for $\nu = q - 1/2$ with $q\in\bbN$.}
\begin{gather*}
k_{1/2}(z) = \exp\biggl( -\frac{z}{\ell} \biggr), \quad
k_{3/2}(z) = \biggl( 1+\frac{\sqrt{3}z}{\ell} \biggr) \exp\biggl(-\frac{\sqrt{3}z}{\ell} \biggr),\\
k_{5/2}(z) = \biggl(1+\frac{\sqrt{5}}{\ell}z+\frac{5}{3\ell}z^2 \biggr) \exp\biggl(-\frac{\sqrt{5}z}{\ell}\biggr),
\end{gather*}
where $z=\|x-y\|_2$ for $x,y\in \Gamma$ 
and where the non-dimensional quantity $\ell>0$ denotes the spatial correlation legth. 
These covariance operators are pseudodifferential 
operators of order $r = -2$, $r = -4$, and $r = -6$. 

\begin{figure}[hbt]
	\includegraphics[width=0.45\textwidth]{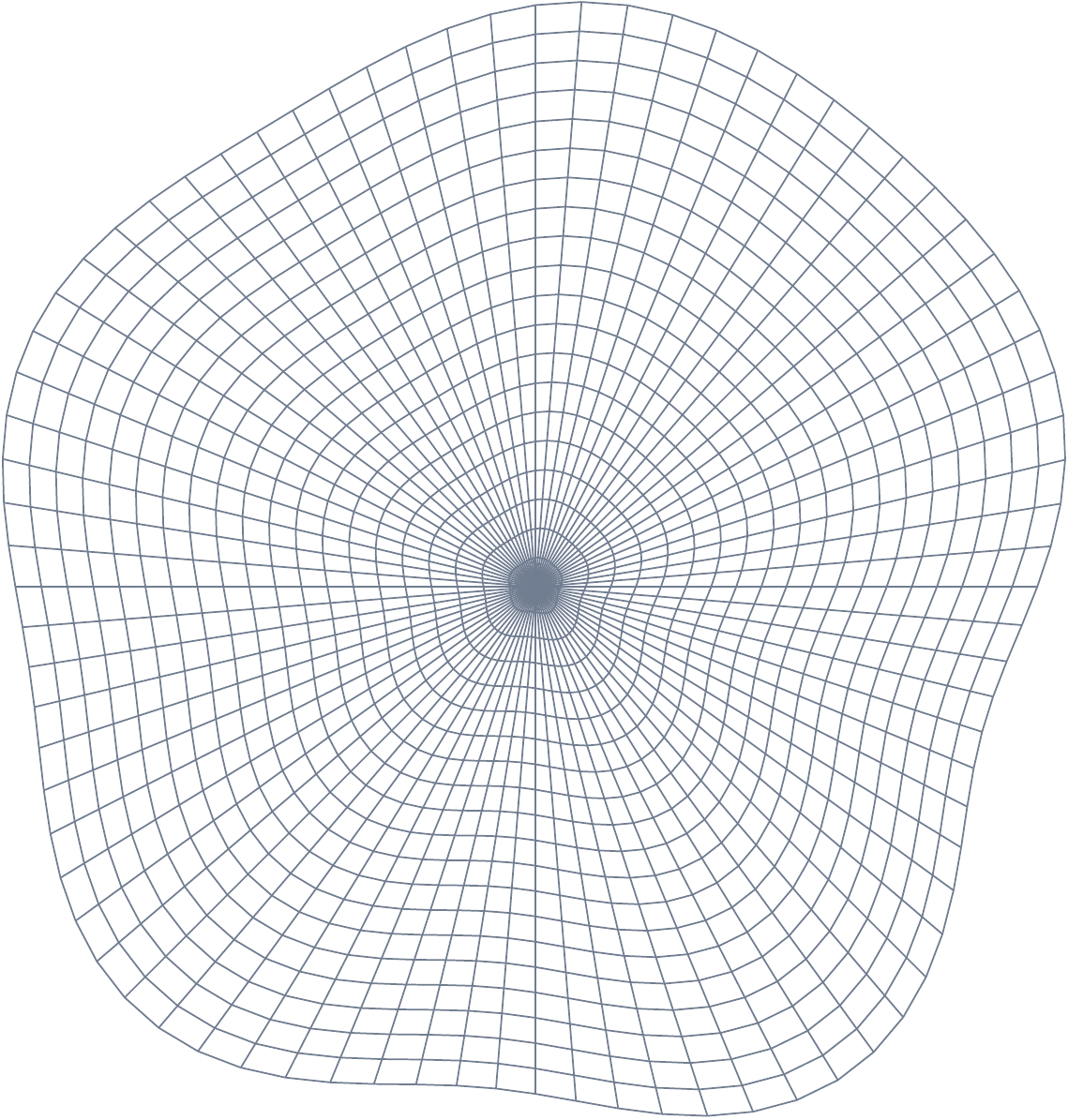}
	\caption{\label{fig:domain}
		The domain under consideration, with a co-ordinate grid.
		Its boundary is used for the numerical tests.}
\end{figure}

We discretize the covariance operators by the (periodic) biorthogonal 
spline wave\-lets $\bPsi^{(d,\widetilde{d})}$ constructed in \cite{CDF}. 
This class of wavelet bases has two parameters, namely the order 
$d$ of the underlying spline space and the number of vanishing 
moments $\widetilde{d}\ge d$, where $d+\widetilde{d}$ is even. 
When $\widetilde{d}$ increases, then the dual wavelet functions
become more regular, enabling preconditioning of 
pseudodifferential operators of negative order. 

For computing the compressed covariance matrix,
the domain is scaled to unit diameter.
Then, we choose $a=a'=2$ and $d' = d + (\widetilde{d}-d+r)/4$ 
in \eqref{eq:Fixaa'}. 
This choice turned out to be robust for different applications. 
The $p\times p$ compressed covariance matrix can be assembled 
in cost which scales linearly with $p$
if exponentially convergent 
$hp$--quadrature methods are employed for the computation of 
matrix entries, cf.~\cite{CvS11_178,HS1,HS2}. 
Further matrix operations such as matrix-vector multiplications
admit additional \emph{a-posteriori compression} which is here
applied. This was found to reduce 
the number of nonzero entries by an additional  
factor between $2$ and $5$, see \cite[Thm.~8.3]{DHSWaveletBEM2007}. 
The pattern of the compressed system matrix shows the 
typical {\em ``finger-band'' structure}, 
see also Figure~\ref{fig:matern}.
\begin{figure}[hbt]
	\includegraphics[trim={90 30 90 25},clip,width=0.45\textwidth]{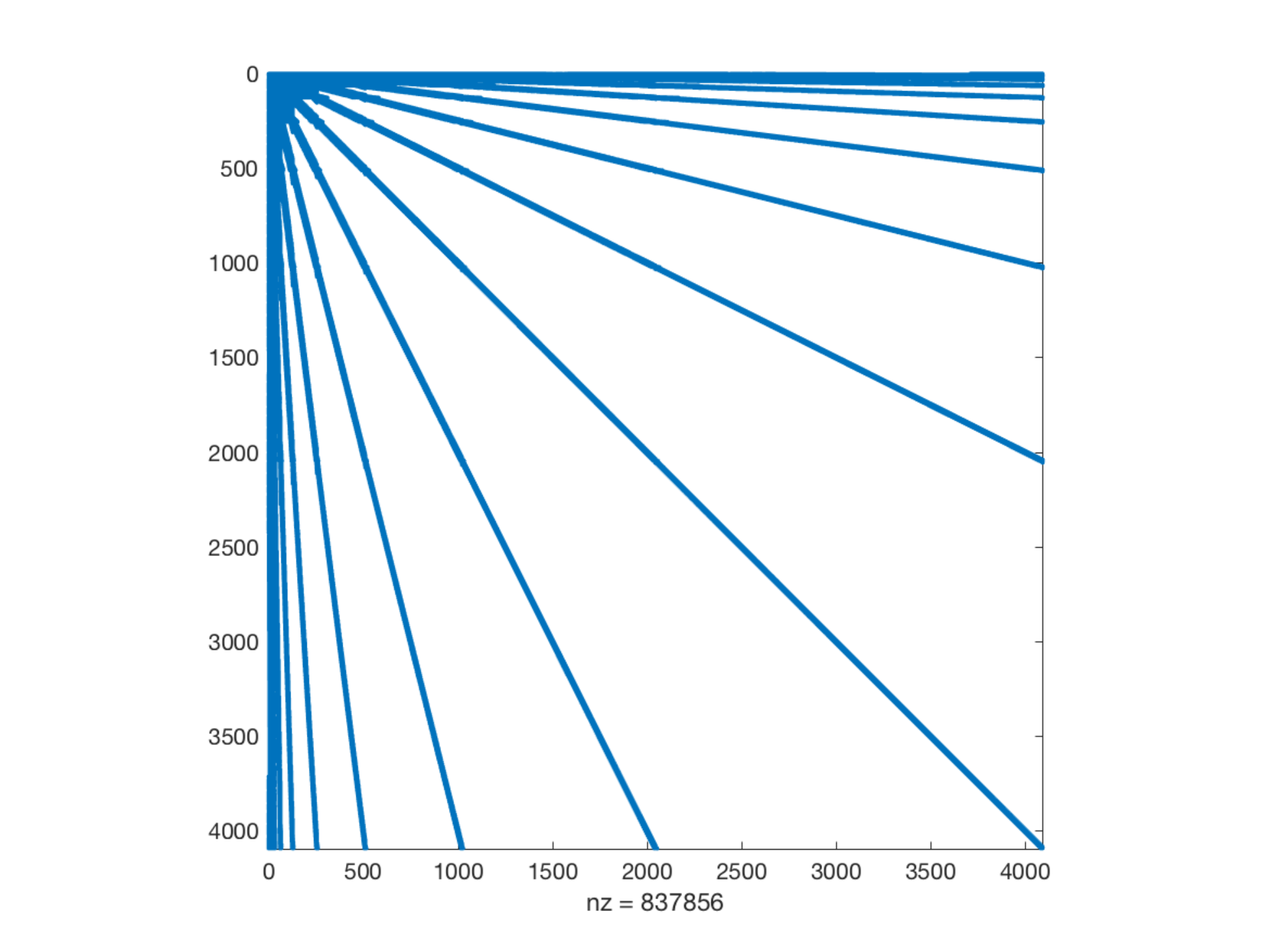}\quad
	\includegraphics[trim={90 30 90 25},clip,width=0.45\textwidth]{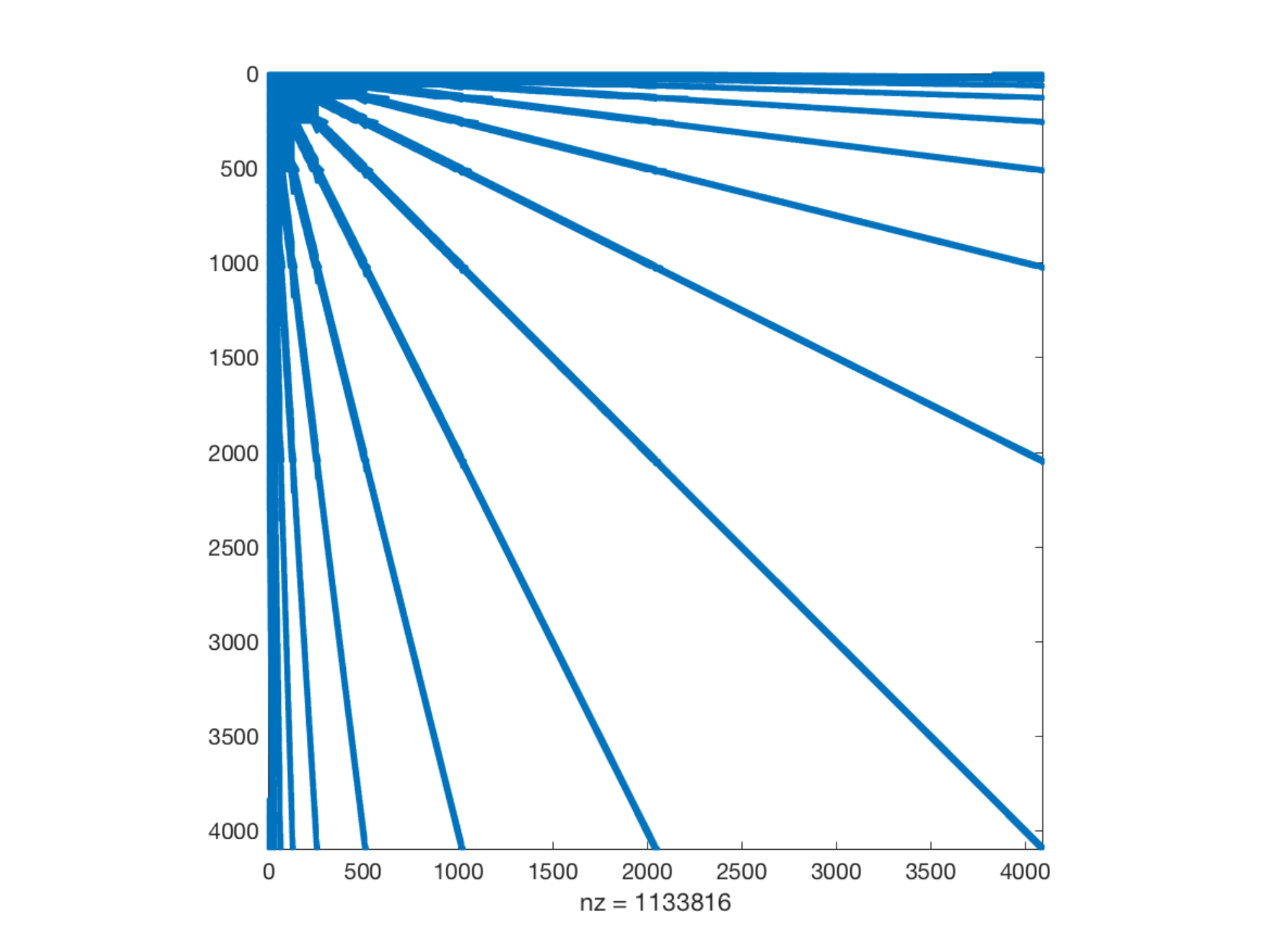}
	\caption{\label{fig:matern}A-priori compression pattern for $p = 4096$ wavelets 
		in case of the Mat\'ern covariance kernel $k_{1/2}$ and $\bPsi^{(2,6)}$ (\emph{left}) 
		and 
		in case of the Mat\'ern covariance kernel $k_{3/2}$ and $\bPsi^{(2,8)}$ (\emph{right}).
		In the left and right matrix, only 5.0\,\% and 6.8\,\% of the matrix coefficients are 
		relevant, respectively.}
\end{figure}

\subsection{Condition numbers and compression rates}

We choose first the correlation length $\ell = 1$ and focus
on the covariance operators for $k_{1/2}$ and $k_{3/2}$.
From the inequality $d < \widetilde{d} + r$ for achieving  
optimal compression rates, we conclude that we 
need at least $\widetilde{d} = 6$ vanishing moments to 
discretize $k_{1/2}$ and $\widetilde{d} = 8$ vanishing 
moments to discretize $k_{3/2}$. In our experiments, we 
also include the borderline case of $\widetilde{d} = d-r$ vanishing 
moments, which leads only to a loglinear compression rate. 

\begin{table}[hbt]
	\begin{tabular}{|cc|c|cc||cc|cc|cc|}\hline
		\multicolumn{9}{|c|}{$k_{1/2}$}\\\hline
		$p$ & $J$ & single-scale & nnz & $\bPsi^{(2,4)}$ &  nnz & $\bPsi^{(2,6)}$ & nnz & $\bPsi^{(2,8)}$ \\\hline
		32 & 5 & $2.6\cdot 10^3$ & 100 & $2.4\cdot 10^2$ & 100 & $1.8\cdot 10^2$ & 100 & $6.6\cdot 10^2$\\
		64 & 6 & $1.1\cdot 10^4$  & 80 & $2.7\cdot 10^2$ & 88 & $1.9\cdot 10^2$ & 98 & $6.7\cdot 10^2$\\
		128 & 7 & $4.5\cdot 10^4$ & 60 & $3.1\cdot 10^2$ & 65 & $1.9\cdot 10^2$ & 71 & $6.8\cdot 10^2$\\
		256 & 8 & $1.9\cdot 10^5$ & 40 & $3.4\cdot 10^2$ & 42 & $1.9\cdot 10^2$ & 48 & $6.8\cdot 10^2$\\
		512 & 9 & $7.6\cdot 10^5$ & 25 & $3.7\cdot 10^2$ & 26 & $1.9\cdot 10^2$ & 30 & $6.8\cdot 10^2$\\
		1024 & 10 & $3.1\cdot 10^6$ & 16 & $3.9\cdot 10^2$ & 16 & $1.9\cdot 10^2$ & 18 & $6.8\cdot 10^2$\\
		2048 & 11 & $1.2\cdot 10^7$ & 9.4 & $4.0\cdot 10^2$ & 9.0 & $1.9\cdot 10^2$ & 10 & $6.8\cdot 10^2$\\
		4096 & 12 & $5.0\cdot 10^7$ & 5.0 & $4.2\cdot 10^2$ & 5.0 & $1.9\cdot 10^2$ & 5.7 & $6.8\cdot 10^2$\\\hline
	\end{tabular}
	\caption{\label{tab:Matern1/2}Condition numbers and compression rates 
		in case of the Mat\'ern covariance kernel $k_{1/2}$. 
		The compression
		rates validate the asymptotically linear behaviour. The condition numbers
		stay bounded for $\bPsi^{(2,6)}$ and $\bPsi^{(2,8)}$, 
		whereas for $\bPsi^{(2,4)}$
		a slight increase is observed.}
\end{table}

\begin{table}[hbt]
	\begin{tabular}{|cc|c|cc||cc|cc|cc|}\hline
		\multicolumn{9}{|c|}{$k_{3/2}$}\\\hline
		$p$ & $J$ & single-scale &  nnz & $\bPsi^{(2,6)}$ & nnz & $\bPsi^{(2,8)}$ & nnz & $\bPsi^{(2,10)}$  \\\hline
		32 & 5 & $3.2\cdot 10^5$ & 100 & $2.3\cdot 10^3$ & 100 & $1.9\cdot 10^4$ & 100 & $1.9\cdot 10^4$  \\
		64 & 6 & $5.8\cdot 10^6$  & 91 & $3.3\cdot 10^3$ & 98 & $2.3\cdot 10^4$ & 100 & $2.0\cdot 10^4$  \\
		128 & 7 & $1.1\cdot 10^8$ & 69 & $4.9\cdot 10^3$ & 75 & $2.5\cdot 10^4$ & 79 &  $2.0\cdot 10^4$  \\
		256 & 8 & $1.9\cdot 10^9$ & 48 & $6.9\cdot 10^3$ & 51 & $2.6\cdot 10^4$ & 55 & $2.0\cdot 10^4$  \\
		512 & 9 & $3.3\cdot 10^{10}$ & 31 & $1.0\cdot 10^4$ & 33 & $2.6\cdot 10^4$ & 36 & $2.0\cdot 10^4$  \\
		1024 & 10 & $5.4\cdot 10^{11}$ & 19 & $1.3\cdot 10^4$ & 20 & $2.7\cdot 10^4$ & 21 & $2.0\cdot 10^4$  \\
		2048 & 11 & $8.8\cdot 10^{12}$ & 11 & $1.8\cdot 10^4$ & 12 & $2.7\cdot 10^4$ & 12 & $2.1\cdot 10^4$  \\ 
		4096 & 12 & $1.4\cdot 10^{14}$ & 6.7 & $2.5\cdot 10^4$ & 6.8 & $2.8\cdot 10^4$ & 7.0 & $2.8\cdot 10^4$  \\ \hline
	\end{tabular}
	\caption{\label{tab:Matern3/2}Condition numbers and compression 
		rates in case of the Mat\'ern covariance kernel $k_{3/2}$.
		The numerical compression rates validate the asymptotically linear behaviour. 
		The numerical condition numbers
		stay bounded for $\bPsi^{(2,8)}$ and ${\bPsi}^{(2,10)}$, 
		whereas for $\bPsi^{(2,6)}$ a slight increase is observed.
	}
\end{table}

The numerical results are listed in Table~\ref{tab:Matern1/2} 
for the Mat\'ern covariance kernel $k_{1/2}$ and in Table\ 
\ref{tab:Matern3/2} for the Mat\'ern covariance kernel $k_{3/2}$. 
We find therein the condition numbers and the a-priori compression 
rates for the discretization by $p = 2^J$ piecewise linear hat 
functions and wavelets, respectively. 
It is seen from the~column 
labeled ``single-scale'' that the condition number grows 
indeed by the factor $2^{|r|}$ in case of the discretization 
by piecewise linear hat functions. 
In contrast, the condition
numbers in case of the discretization by wavelets is 
bounded for all choices of $\widetilde{d}$ except for the
borderline case $\widetilde{d} = d-r$, where the condition numbers
still grow, although quite moderately. 

The compression rates, measured by the percentage of
the number of nonzero coefficients (nnz) relative to $p^2$, 
are also good in the (borderline) case $\widetilde{d} = d-r$, 
although then only a loglinear compression rate can generally 
be expected. 
This is caused by wavelets with fewer vanishing moments
tending to have smaller supports 
as is known to hold for the wavelets $ \psi^{(d,\tilde d)}$ 
from \cite{CDF} which are presently used.
Hence, we obtain less matrix 
coefficients in the system matrix which correspond to wavelets 
with overlapping supports. 
The compression pattern of the system 
matrices and $p = 4096$ wavelets are displayed for the Mat\'ern 
covariance kernel $k_{1/2}$ and $\bPsi^{(2,6)}$ on the left and in 
case of the Mat\'ern covariance kernel $k_{3/2}$ and $\bPsi^{(2,8)}$ 
on the right panel of Figure~\ref{fig:matern}.

\subsection{Influence of the correlation length on the compression rates}
\label{subsct:influence-length}

We should finally comment on the dependence of the compression on 
the correlation length. Since we do not consider the correlation length 
in the a-priori compression, it has no effect on this compression. 
Nonetheless, the correlation length has a considerable effect 
on the a-posteriori compression. This can be
seen in Figure~\ref{fig:corlength}, where we plotted the compression 
rates versus the correlation length in case of the Mat\'ern covariance kernels
$k_{1/2}$ and $k_{3/2}$ for a fixed level of resolution and wavelet
basis (we use $\bPsi^{(2,6)}$ for $k_{1/2}$ and $\bPsi^{(2,8)}$ for $k_{3/2}$). 
While the a-priori compression rates are fixed, the a-posteriori compression
improves as the correlation length decreases. This effect is clear as the 
correlation becomes more and more local, thus, the far-field interaction
gets negligible.

\begin{figure}[hbt]
	\includegraphics[trim={7 7 7 7},width=0.45\textwidth]{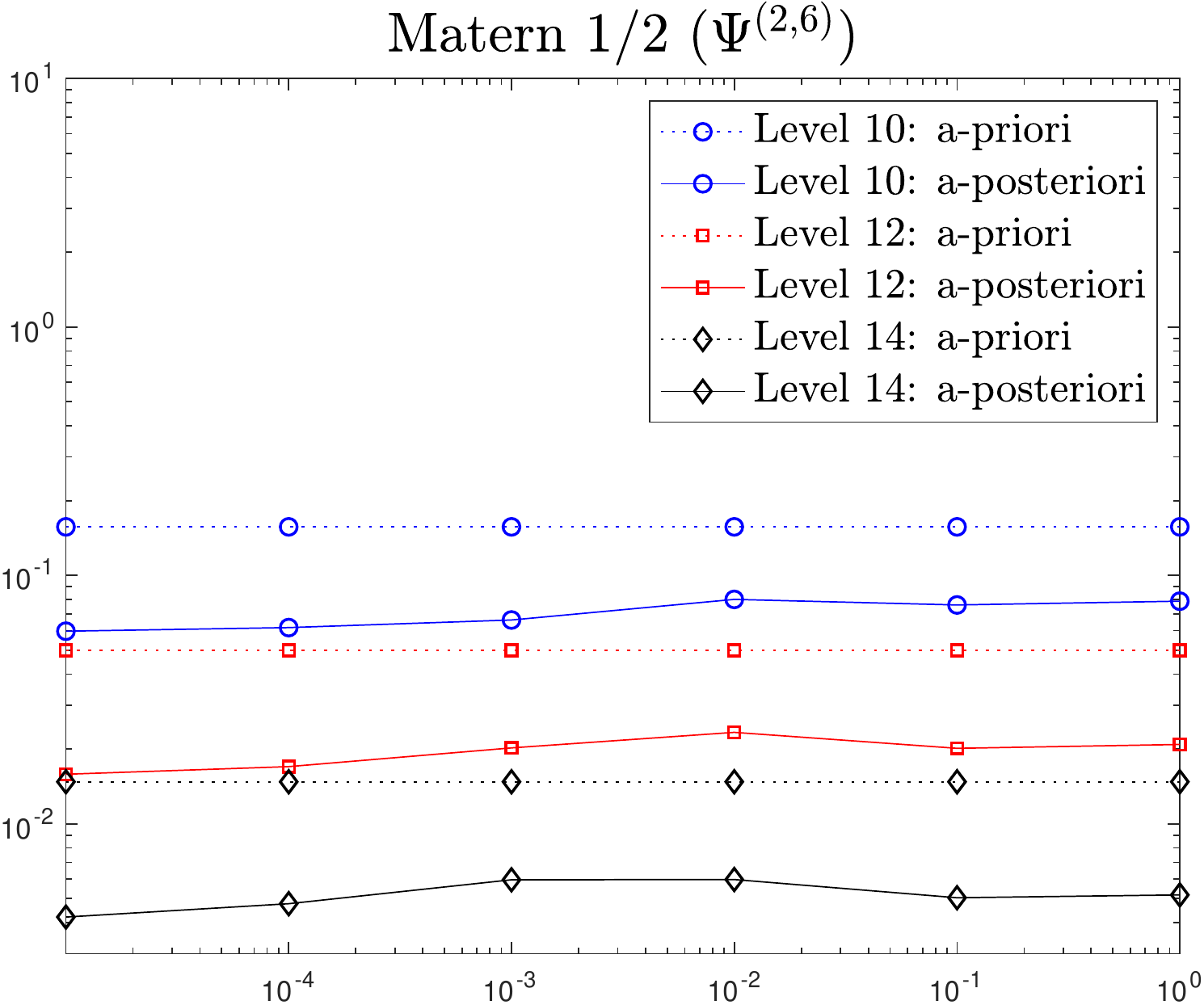}\quad
	\includegraphics[trim={7 7 7 7},width=0.45\textwidth]{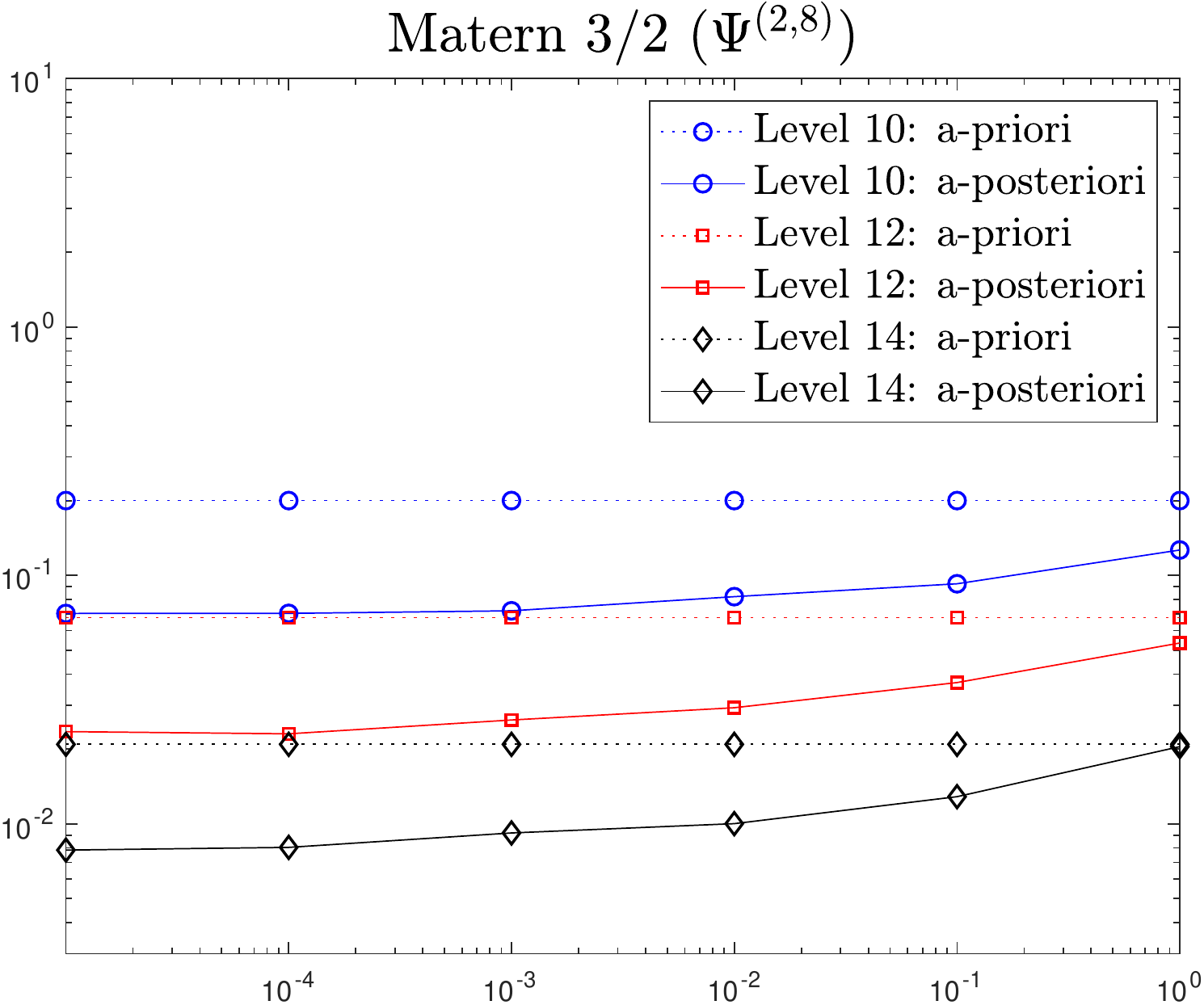}
	\caption{\label{fig:corlength}Influence of the correlation length on the 
		compression rates (measured by the number of nonzero matrix entries
		in percent) in case of the Mat\'ern covariance kernel $k_{1/2}$ and 
		$\bPsi^{(2,6)}$ (\emph{left}) and in case of the Mat\'ern covariance 
		kernel $k_{3/2}$ and $\bPsi^{(2,8)}$ (\emph{right}). While the a-priori 
		compression is unchanged, \emph{a-posteriori compression
			improves as the spatial correlation length decreases}.}
\end{figure}

\subsection{Decay of the diagonal entries}\label{subsct:decay}

We next consider the behavior of the diagonal of the covariance
matrices in wavelet coordinates. In Figure~\ref{fig:decay}, we plotted
the diagonal entries for the Mat\'ern covariance kernels
$k_{1/2}$, $k_{3/2}$, and $k_{5/2}$. We clearly see the 
transition between the levels at the abscissa values $2^j$.
And indeed, if we compute the mean of the diagonal entries
per level, the jump size between subsequent levels is precisely
$4$, $16$, and $64$, which corresponds to $2^{-r}$ with $r$
being the operator order. As a consequence, the knowledge of 
the diagonal entries (or their mean) of just two subsequent levels
is sufficient to estimate the coloring  operator order, 
which, in turn, determines the path regularity and the tapering pattern.
\begin{figure}[hbt]
	\begin{center}
		\includegraphics[width=0.5\textwidth]{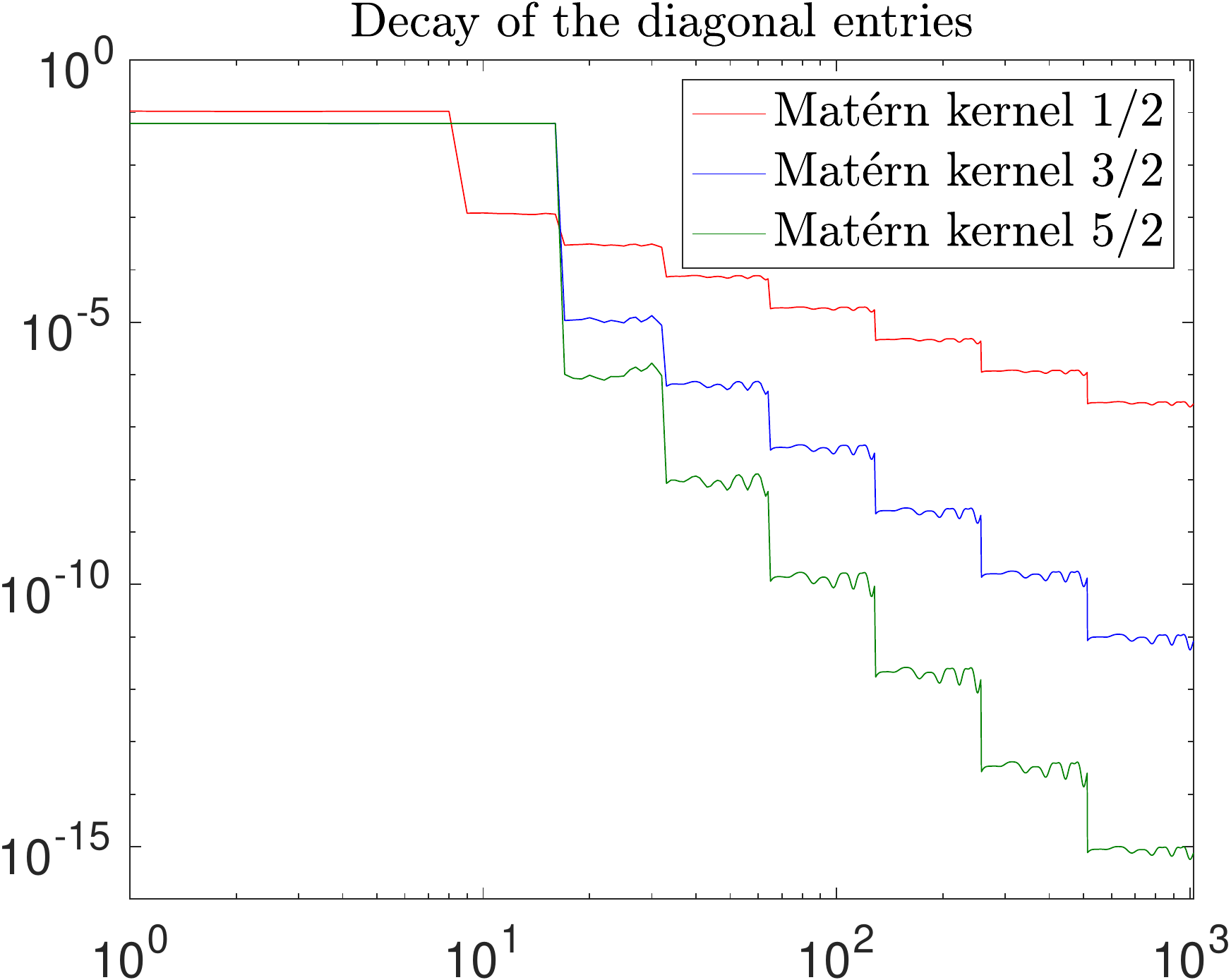}
		\caption{\label{fig:decay}
			Size of the diagonal coefficients in case of the Mat\'ern 
			covariance kernels $k_{1/2}$, $k_{3/2}$, and $k_{5/2}$ for
			$\ell=1$ and $n=1$. We clearly observe the decay relative to the level, 
			which depends on the order of the covariance operator. The jump 
			between the level is of relative height  $4$, $16$, and $64$ 
			and reflects the operator order. 
			Jump height is $2^{|r|}$, where $r = -(2\nu + 1)$.
		}
	\end{center}
\end{figure}
%
\subsection{Fast simulation}
\label{sec:FastSim}

We shall next illustrate the efficient numerical simulation of GRF samples
on the algorithm from Subsection~\ref{subsec:simulation}.
We apply this algorithm to compute the square root of the compressed
covariance matrix. To this end, we employ again the Mat\'ern 
covariance kernel $k_{1/2}$ and $\bPsi^{(2,6)}$ as well as the Mat\'ern 
covariance kernel $k_{3/2}$ and $\bPsi^{(2,8)}$. 
The correlation length is chosen as $\ell=1$. 
We numerically evaluate the $\| \cdot \|_2$ norm error
between the exact matrix square root 
(computed by using the \verb+sqrtm+-function from {\sc Matlab}\footnote{Release 2018b}) 
and the approximation by 
\eqref{eq:def:Smat} in dependence on the parameter $K$. A 
sensitive input parameter is $\widehat{\varkappa}_{\mathrm{R}}^{-1}$,
which is the ratio between the smallest and largest eigenvalue. Therefore,
we use its exact value on one hand and its over- or underestimation by 
a factor of two on the other hand, which accounts for numerical approximation. 
The results are displayed in Figure~\ref{fig:elliptic}.
\begin{figure}[hbt]
	\includegraphics[width=0.5\textwidth]{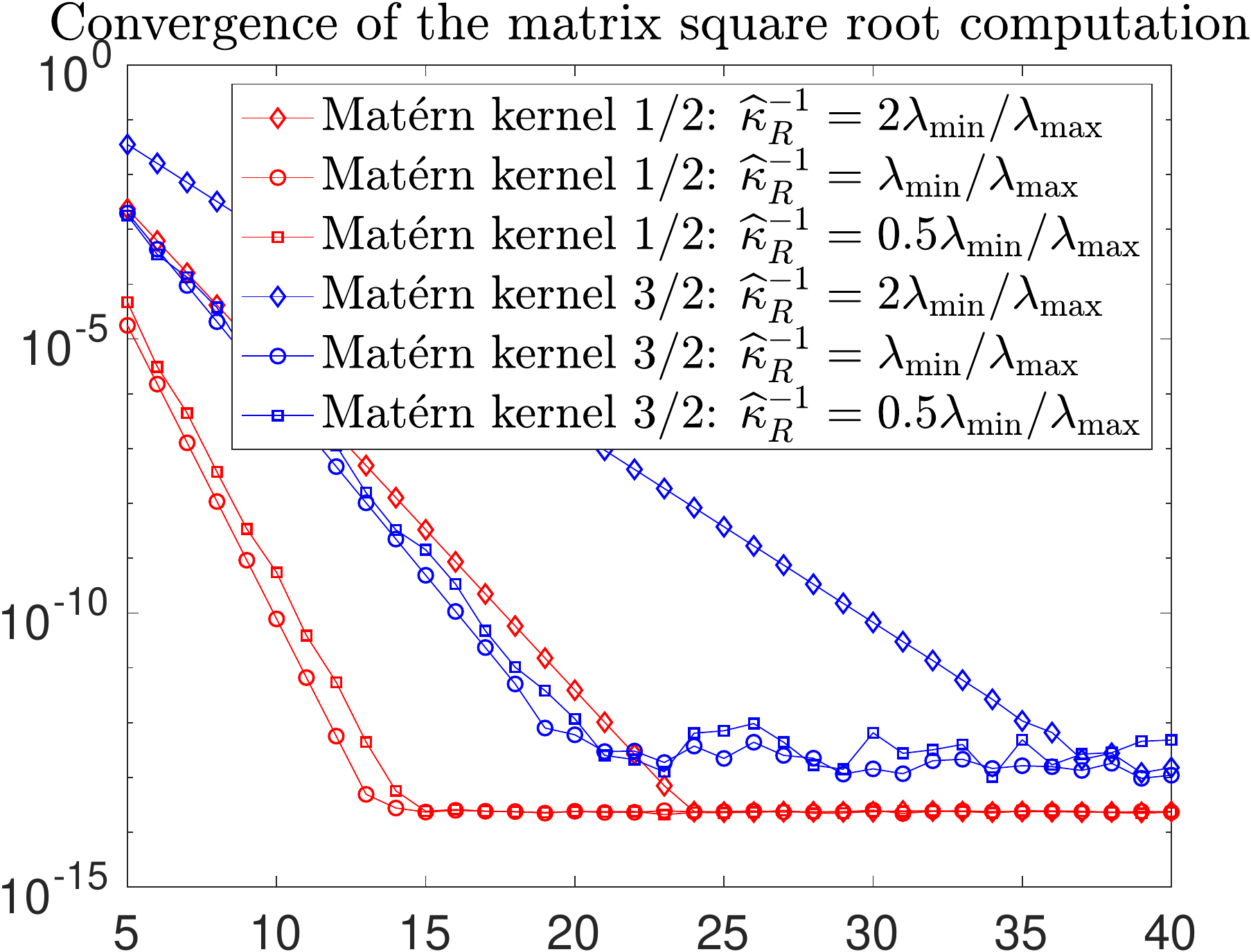}
	\caption{\label{fig:elliptic}Convergence of the approximation \eqref{eq:def:Smat}
		of the matrix square root of the diagonally scaled covariance matrix in case 
		of the Mat\'ern covariance kernel $k_{1/2}$ and $\bPsi^{(2,6)}$ and in case 
		of the Mat\'ern covariance kernel $k_{3/2}$ and $\bPsi^{(2,8)}$. 
		The convergence rate is independent on the discretization level.}
\end{figure}
It tuns out that the convergence heavily depends on $\widehat{\varkappa}_{\mathrm{R}}^{-1}$
and, thus, on the condition number of the matrix under consideration. Especially,
underestimation of $\widehat{\varkappa}_{\mathrm{R}}^{-1}$ seems to be harmless
while overestimation slows down convergence considerably. Nonetheless, in any
case, we achieve after for $K=40$ machine precision. Although the computations
have been only carried out for fixed discretization level (namely, for $p=1024$
wavelets), we obtain exactly the same plots for other values of $p$ as the
condition number of the covariance matrix stays constant in accordance 
with Tables~\ref{tab:Matern1/2} and \ref{tab:Matern3/2}.

\subsection{Covariance estimation}\label{subsct:cov-estimation}

We shall next illustrate the multilevel Monte Carlo estimation of 
the covariance matrix. To that end, we consider the $\ell^2$-difference
between the original (uncompressed) covariance matrix and its 
approximation by the multilevel Monte Carlo method, using 
wavelet matrix compression. As test case, we consider the 
Mat\`ern kernel $k_{1/2}$, which is of order $-2$. Consequently, 
it holds $r=1$, $n=1$, and $t=t'=0$ in Section~\ref{sec:MLMCCovEst}.
The wavelet basis used to discretize 
the covariance matrix is $\bPsi^{(2,6)}$.
For the Monte Carlo sampling, we choose the fixed number of 
$\widetilde{M}_J = 100$ samples on the finest level $J$ of 
spatial resolution
and increase the number $\tilde{M}_j$ of MC samples
by the factor $2^{2(n+\alpha)/3} = 2$ 
when passing from spatial 
resolution level $j$ to $j-1$, $1\leq j \leq J$.

There, 
we chose the borderline case $\alpha = 1/2$. 
This essentially yields the convergence order $2^{J/2}$ of the 
multilevel Monte Carlo method.

\begin{figure}[hbt]
	\includegraphics[trim={60 40 45 30},clip,height=0.4\textwidth,width=0.45\textwidth]{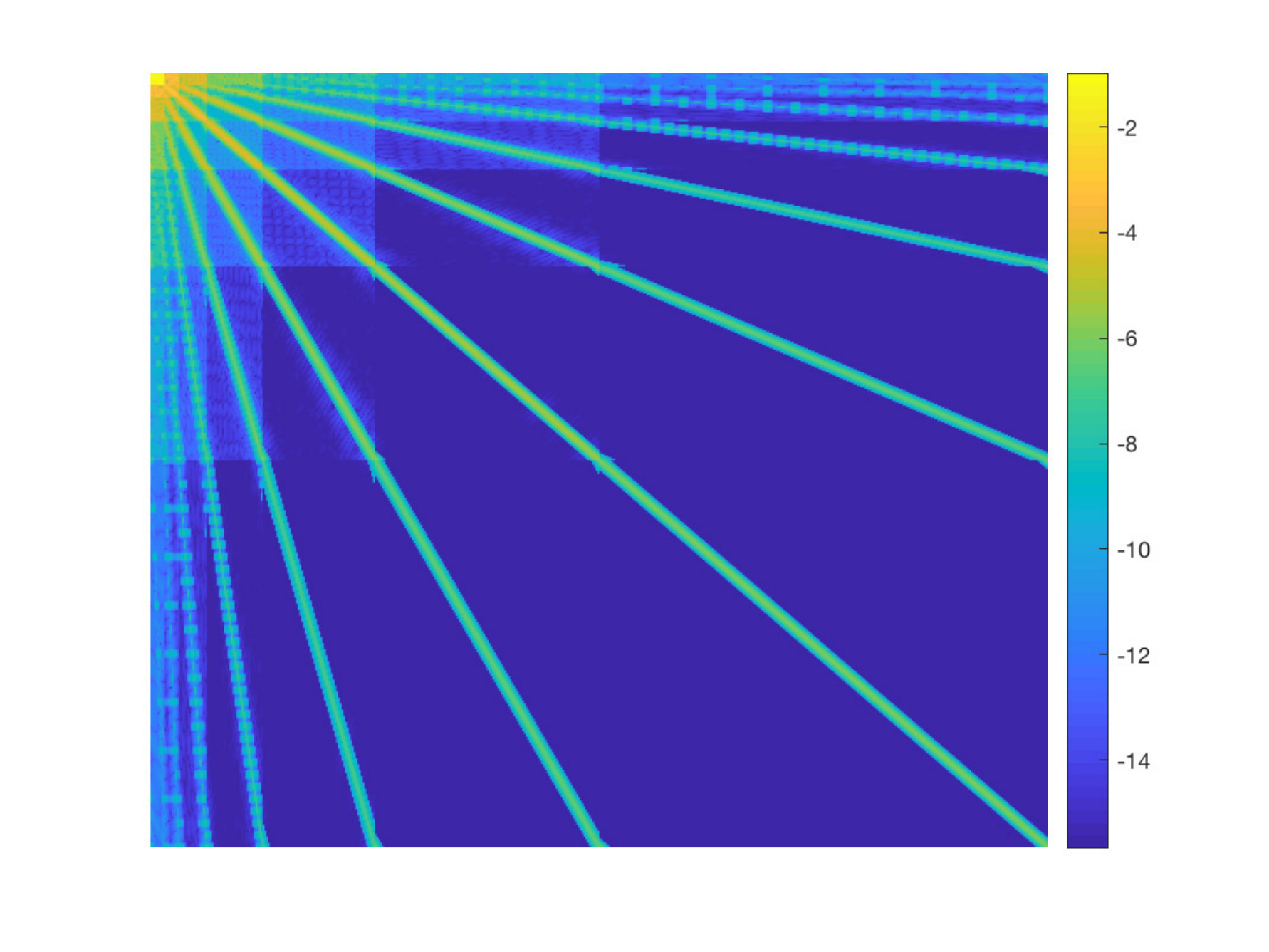}
	\quad 
	\includegraphics[trim={60 40 45 30},clip,height=0.4\textwidth,width=0.45\textwidth]{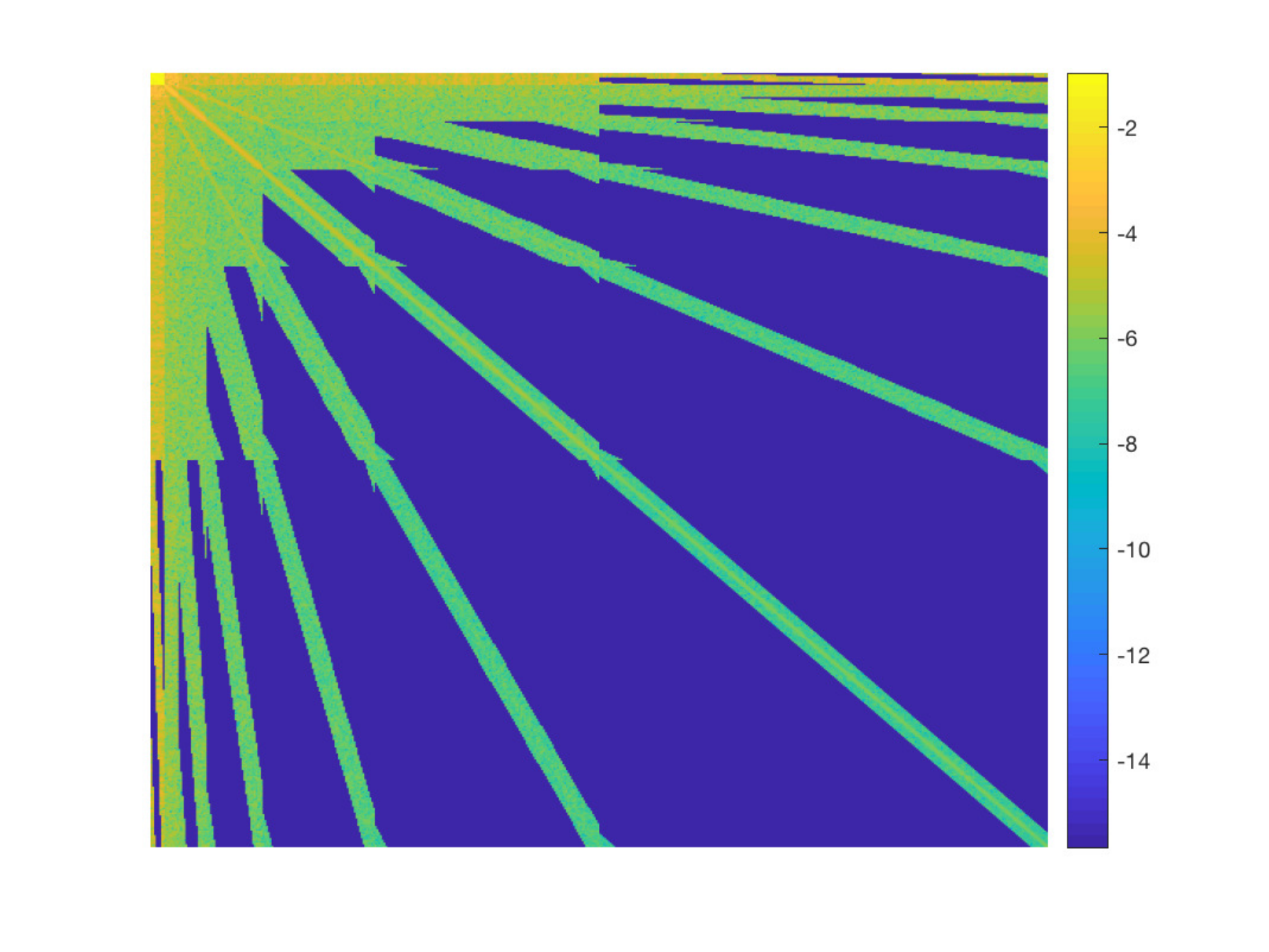}
	\caption{\label{fig:MLMC}
		Truth covariance matrix (\emph{left}) in wavelet representation
		and its multilevel Monte Carlo estimation (\emph{right}) for $p = 512$ parameters.
		Spatial dimension $d=1$, Mat\`{e}rn-covariance kernel $k_{1/2}$, 
		spatial correlation length $\ell = 1$, wavelet $\bPsi^{(2,6)}$.
	}
\end{figure}

\begin{table}
	\begin{tabular}{|cc|c|cc|}\hline
		$p$ & $J$ & $\widetilde{M_j} $ & \multicolumn{2}{c|}{$\ell^2$-error} \\\hline
		8 & 3 & 51200 & $1.1\cdot 10^{-1}$ & --- \\
		16 & 4 & 25600 & $5.4\cdot 10^{-2}$ & (2.1) \\
		32 & 5 & 12800 & $4.9\cdot 10^{-2}$ & (1.1) \\
		64 & 6 & 6400 & $2.9\cdot 10^{-2}$ &(1.7) \\
		128 & 7 & 3200 & $1.9\cdot 10^{-2}$ & (1.5) \\
		256 & 8 & 1600 & $1.3\cdot 10^{-2}$ & (1.4) \\
		512 & 9 & 800 & $1.1\cdot 10^{-2}$ & (1.3) \\
		1024 & 10 & 400 & $8.5\cdot 10^{-3}$ & (1.2) \\
		2048 & 11 & 200 & $5.3\cdot 10^{-3}$ & (1.6) \\
		4096 & 12 & 100 & $2.9\cdot 10^{-3}$ & (1.3) \\\hline
	\end{tabular}
	\caption{\label{tab:MLMC} 
		MLMC Covariance estimation.
		Sample sizes $\widetilde{M}_j$ and accuracy 
		of the multilevel Monte Carlo covariance estimation, with
		$\widetilde{M}_J=100$, and with $\widetilde{M}_j= \widetilde{M}_J 2^{J-j}$;
		presented here for $J=12$.
		Estimation error in operator norm with respect to 
		the (densely populated), exact covariance matrix $\bC_p$ in wavelet coordinates.
	}
\end{table}

The results are presented in Table~\ref{tab:MLMC}. Here, one figures out 
the sample numbers $\widetilde{M_j}$ per level $j$ in case of discretization
level $J=12$. For smaller levels, one just has to remove the largest 
numbers accordingly. We moreover tabulated the $\ell^2$-error between 
the (uncompressed) covariance matrix and its estimate, where the given
numbers correspond to the mean of 10 runs. The convergence
order is like expected, as validated by the contraction factor
$1.41$ between the levels, which is approximately observed.
In Figure~\ref{fig:MLMC}, one finds the original covariance 
matrix of size $512 \times 512$ on the left and its Monte Carlo 
estimate on the right.
In the (single-level) MC estimate, no a-priori (oracle) information
on the sparsity pattern has been provided. Still,
the compression pattern has clearly been identified.

\subsection{Sparse approximate kriging}\label{subsct:kriging}

We next consider the kriging approach presented
in Subsection~\ref{subsec:krging} and especially in
Remark~\ref{rmk:OK+p}. 
To this end, we consider 
that we have given $K=32$ locally supported functionals 
$g_i$ which are equidistantly distributed at the boundary 
$\Gamma$ of the computational domain under consideration.
Then, using $p = 512$ piecewise linear ansatz functions and
wavelets $\bPsi^{(2,6)}$, we obtain the matrix representation 
\eqref{eq:kriging_matrix} which is illustrated in Figure~\ref{fig:kriging}.
\begin{figure}[hbt]
	\includegraphics[width=0.95\textwidth]{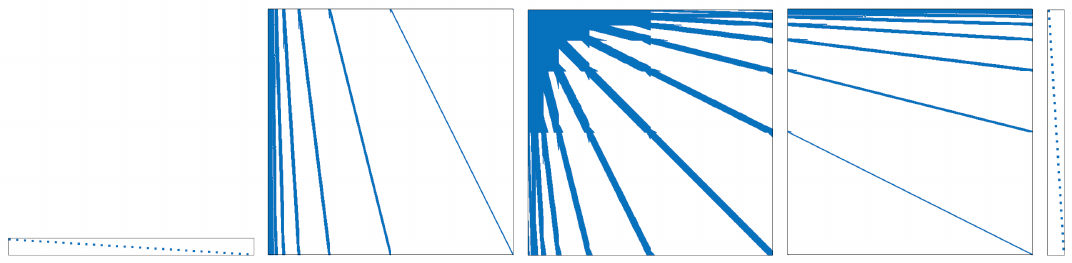}
	\caption{\label{fig:kriging}
		Sparse factorization of the approximate kriging matrix
		${\bf G \bC^\eps_p G}^\top
		=
		{\bf G}_{\tilde\phi}{\bf T}_{\tilde\phi\to\tilde\psi}^{\T}
		\bC_p^{\varepsilon}{\bf T}_{\tilde\phi\to\tilde\psi}\mathbf{G}_{\tilde\phi}^\T
		$ according to 
		Theorem~\ref{thm:SprseKrgCst}, Remark~\ref{rmk:OK+p} and
		\eqref{eq:kriging_matrix}.
	}
\end{figure}
We emphasize that the matrix which arises from the 
fast wavelet transform (second and fourth matrix in
Figure~\ref{fig:kriging}) has obviously $\mathcal{O}(p\log p)$
nonzero matrix coefficients. However, its application to a 
vector can be realized numerically
in $\mathcal{O}(p)$ operations
with a very small constant that depends 
on the filter length of the wavelets.
%
\subsection{Computations of a GRF on $\bbS^2$}\label{subsct:sphere}

After having illustrated the theoretical findings on a manifold
in two spatial dimensions, we shall also demonstrate the wavelet
compression for GRFs in spatial dimension $n=2$. 
We consider the simulation of the centered 
GRF $\GP$ on the unit sphere $\cM = \mathbb{S}^2 \subset \bbR^3$. 
The covariance kernel under consideration is assumed to be 
the Mat\`ern kernel $k_{1/2}$, 
defined in terms of the geodesic distance on $\bbS^2$, with 
unit geodesic correlation length. 
We apply 
piecewise constant wavelets with three vanishing moments, as 
constructed in \cite{HS3}. Since \eqref{eq:Fixaa'} is violated and
the wavelets are also not suitable for preconditioning since $\widetilde{\gamma} 
= 1/2$, we perform the matrix compression as for an operator of order 0. 
This is justified since also the Karhunen-Lo\'eve expansion is
computed with respect to $L^2(\mathbb{S}^2)$. 
As pointed out in \cite{HM}, 
the Cholesky decomposition of the compressed covariance 
matrix can efficiently computed with nested dissection reordering, 
compare Figure~\ref{fig:dissection}. 
Here, we see the original 
matrix pattern of the compressed covariance operator on the 
left, its reordered version in the middle, and the resulting 
Cholesky factor on the right. Indeed, the number of nonzero
matrix coefficients of the Cholesky factor is only about 4--5 times 
higher than that of the compressed covariance operator 
(compare Table~\ref{tab:sphere}).
This appears to be consistent with \cite[Proposition 1]{RothmLevina2010}.

\begin{figure}[hbt]
	\begin{center}
		\includegraphics[trim={115 45 100 30},clip,width=0.3\textwidth]{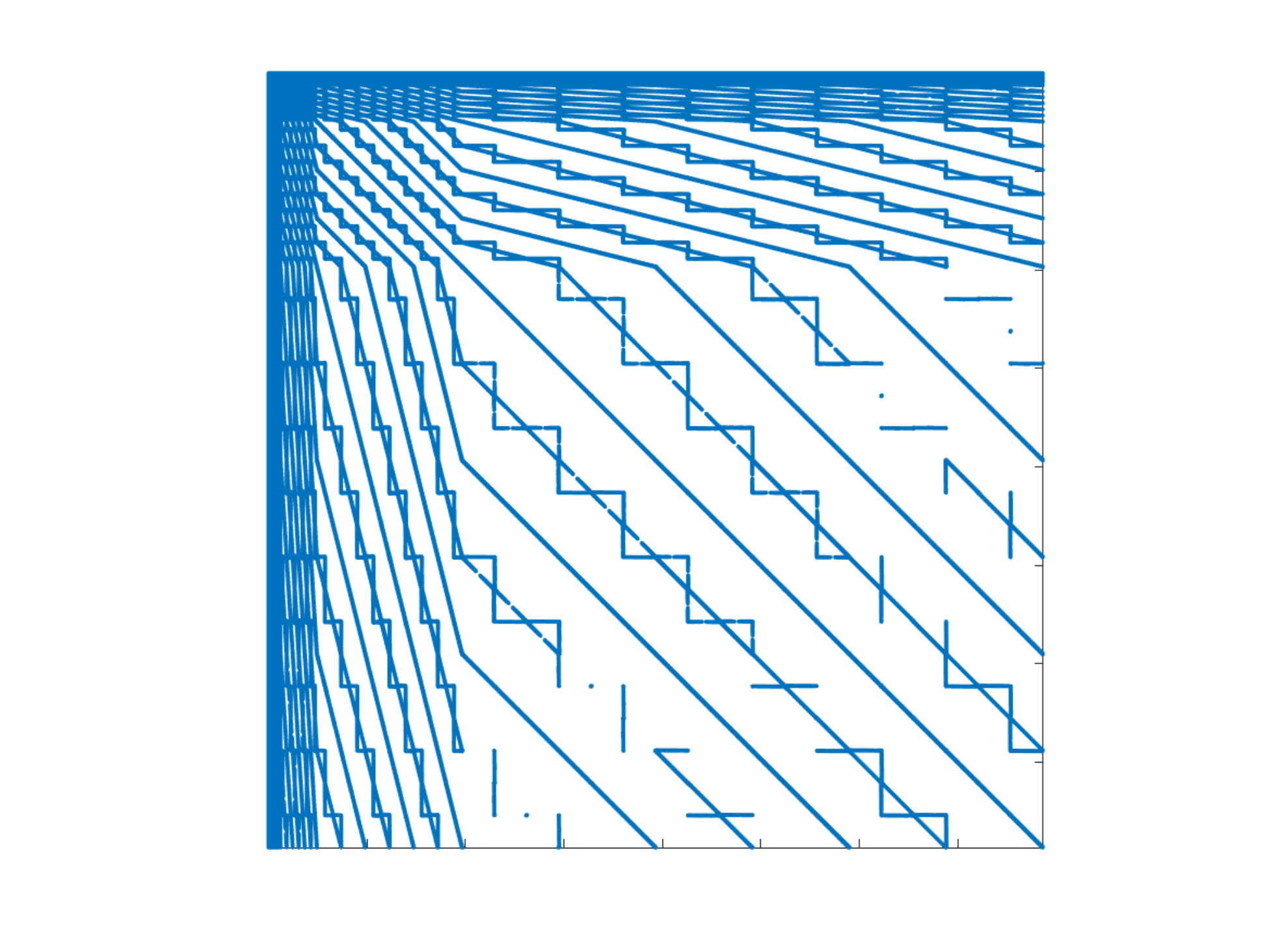}\quad
		\includegraphics[trim={115 45 100 30},clip,width=0.3\textwidth]{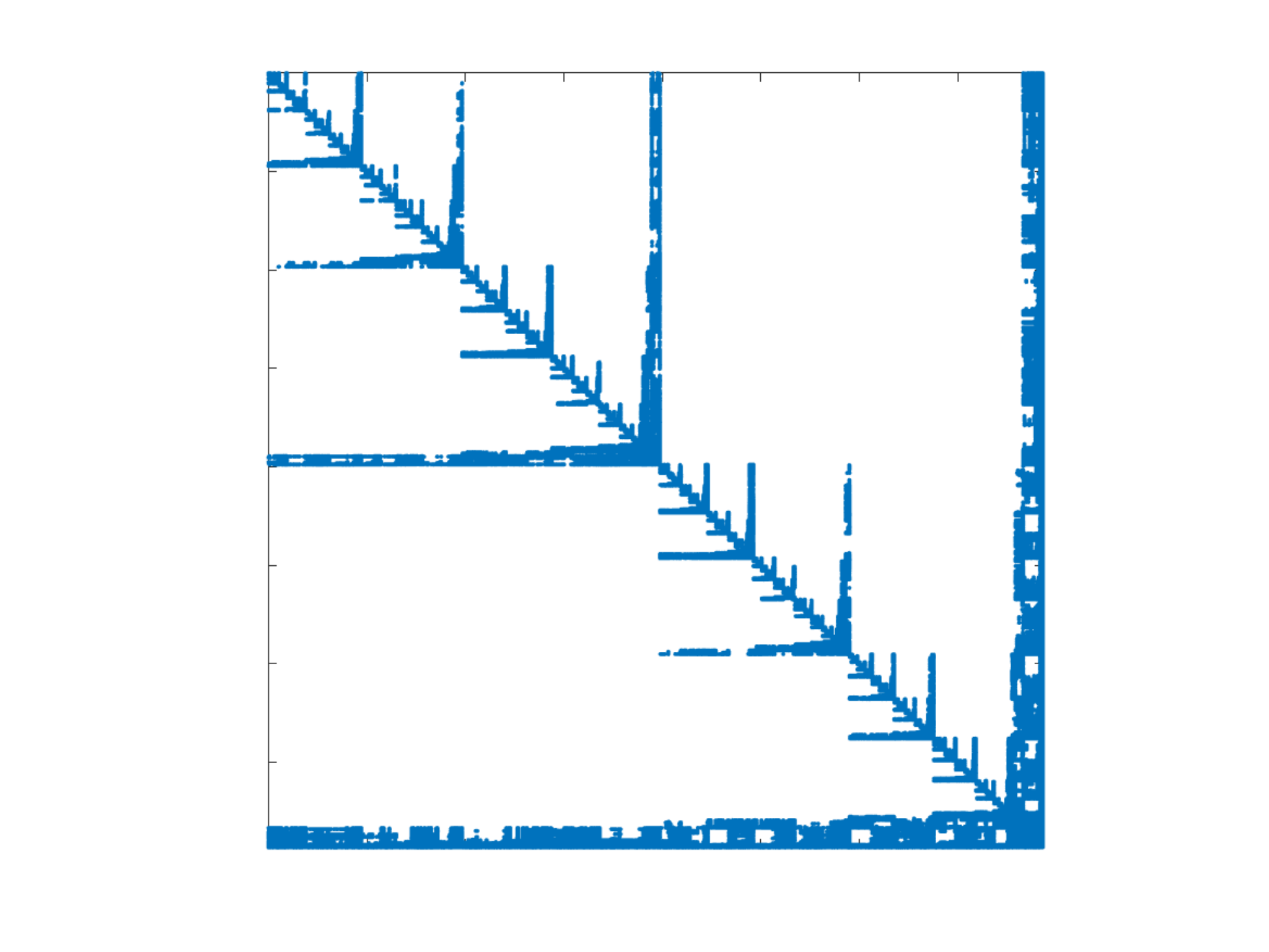}\quad
		\includegraphics[trim={115 45 100 30},clip,width=0.3\textwidth]{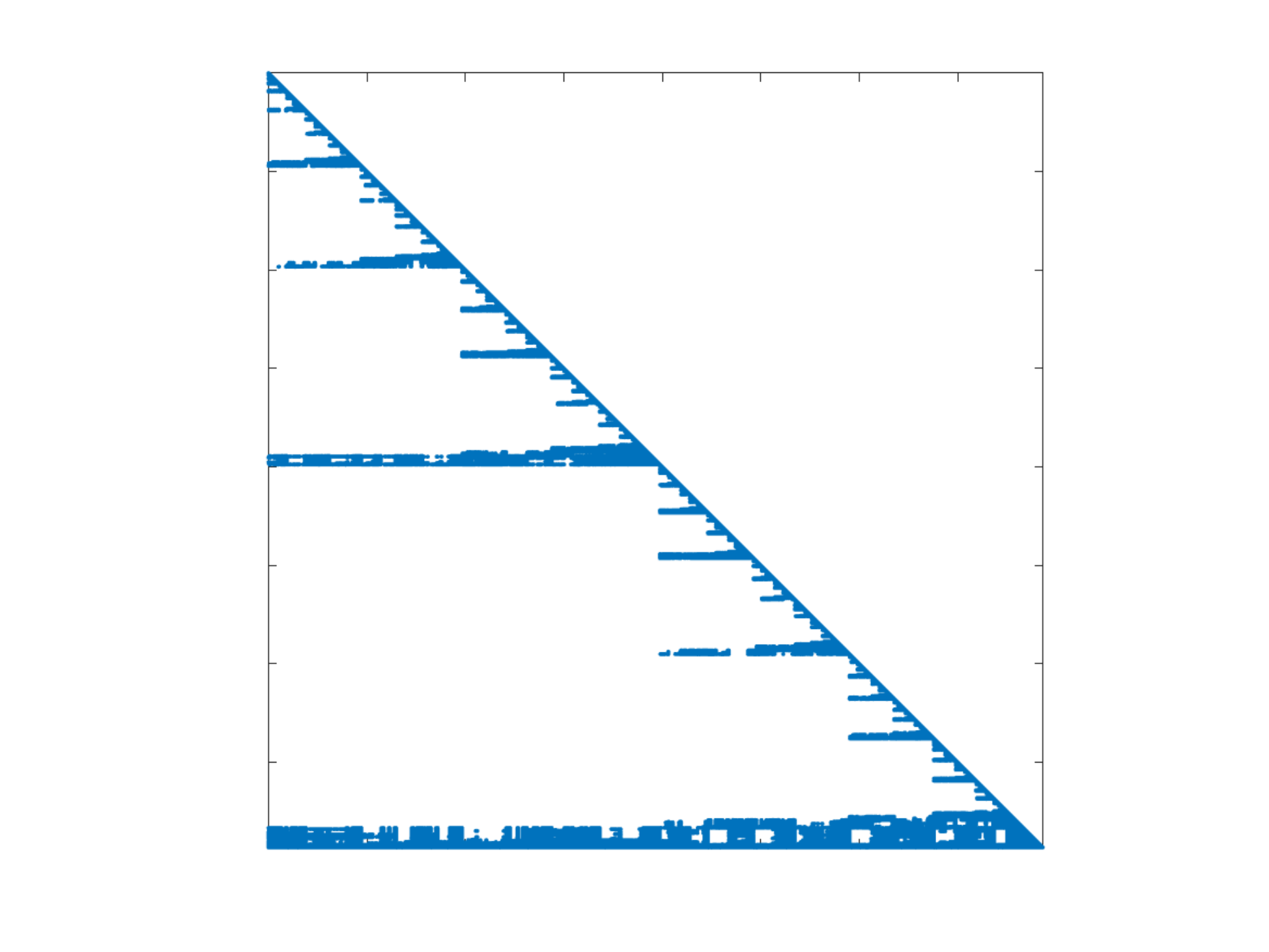}
		\caption{\label{fig:dissection}
			GRF on $\cM = \bbS^2$: sparsity pattern of the compressed 
			covariance operator in wavelet coordinates (\emph{left}),
			its nested dissection, ``skyline'' reordering (\emph{middle}), 
			and sparsity pattern of the exact Cholesky factor (\emph{right}) 
			of the compressed, reordered covariance matrix $\bC_p \in \bbR^{p\times p}$
			for $p=393216$. Consistent with \cite[Proposition 1]{RothmLevina2010},
			see also \cite[Chap.\ 4.2]{GeorgeLiu81}.
		}
	\end{center}
\end{figure}
The efficient drawing of numerically approximated 
random samples proceeds as follows. 
Let ${\bf C}_J^\psi$ denote the compressed covariance operator, 
${\bf C}_J^\psi = {\bf L}_J^\psi\big({\bf L}_J^\psi\big)^\top$ its
Cholesky decomposition, and ${\bf G}_J^\phi$ the mass matrix
with respect to the piecewise constant single-scale basis, which 
is a diagonal matrix. Then, for a uniformly normally distributed 
random vector ${\bf X}(\omega)$, the random vector ${\bf Y}(\omega) 
= \big({\bf G}_J^\phi\big)^{-1}{\bf T}_{\psi\to\phi}{\bf L}_J^\psi{\bf X}(\omega)$ 
represents the sought Gaussian random field on the unit sphere $\bbS^2$, 
expressed 
with respect to the piecewise constant single-scale basis. As can be 
seen in the last column of Table~\ref{tab:sphere}, the computation 
time per sample is very small. Four realizations can be found in 
Figure~\ref{fig:sphere}.

\begin{table}[hbt]
	\begin{center}
		\begin{tabular}{|cc|c|c|c|c|c|}\hline
			\multicolumn{7}{|c|}{Sphere}\\\hline
			$p$ & $J$ & nnz(${\bf C}_J$) & cpu(${\bf C}_J$) & nnz(${\bf L}_J$) & 
			cpu(${\bf L}_J$) & cpu(sample) \\\hline
			6144 & 5 & 4.70 & 18 & 10.3 & 0.65 & 0.0017 \\
			24576 & 6 & 1.22 & 113 & 4.43 & 5.1 & 0.015 \\
			98304 & 7 & 0.43 & 692 & 1.68 & 26 & 0.096 \\
			393216 & 8 & 0.12 & 4108 & 0.59 & 151 & 0.46 \\
			1572864 & 9 & 0.03 & 23374 & 0.20 & 865 & 2.7 \\\hline 
		\end{tabular}
		\caption{\label{tab:sphere}Compression rates and computing
			times in case of the Mat\'ern covariance kernel $k_{1/2}$ on 
			the sphere. 
			Once the Cholesky decomposition has been computed,
			each sample can be computed extremely fast. 
		}
	\end{center}
\end{table}

\begin{figure}[hbt]
	\begin{center}
		\includegraphics[trim={180 105 160 95},clip,width=0.22\textwidth]{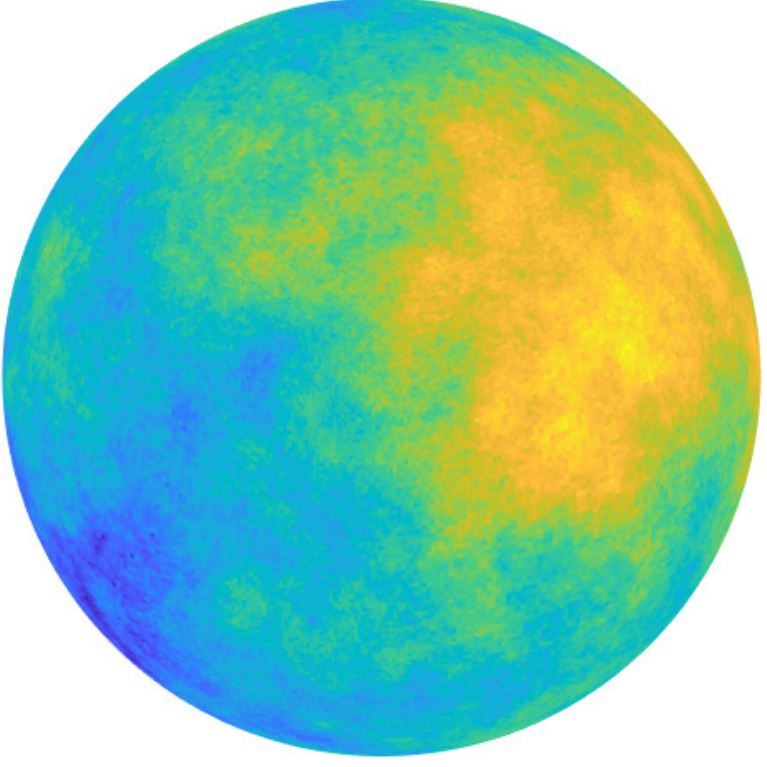}\quad
		\includegraphics[trim={180 105 160 95},clip,width=0.22\textwidth]{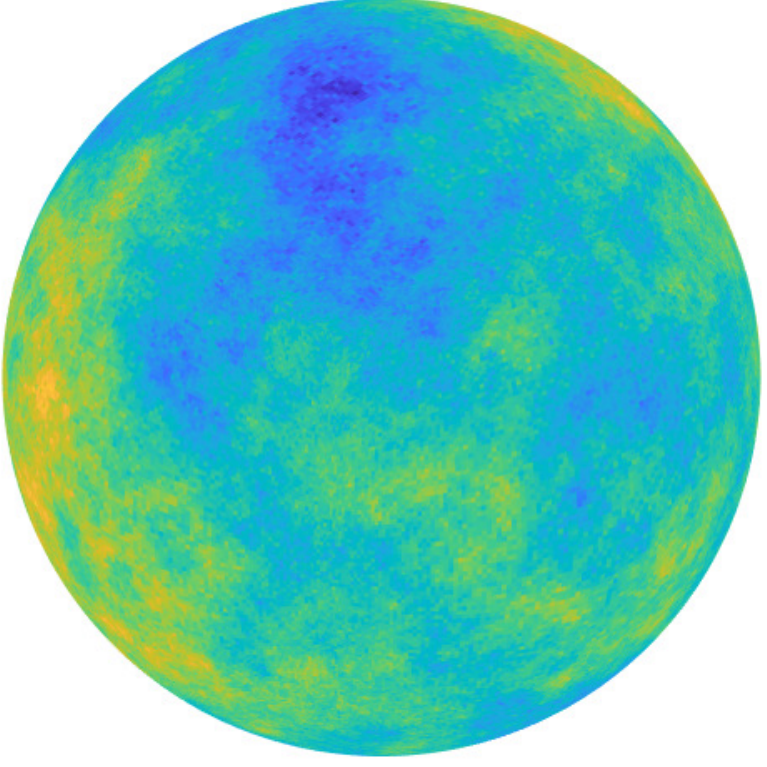}\quad
		\includegraphics[trim={180 105 160 95},clip,width=0.22\textwidth]{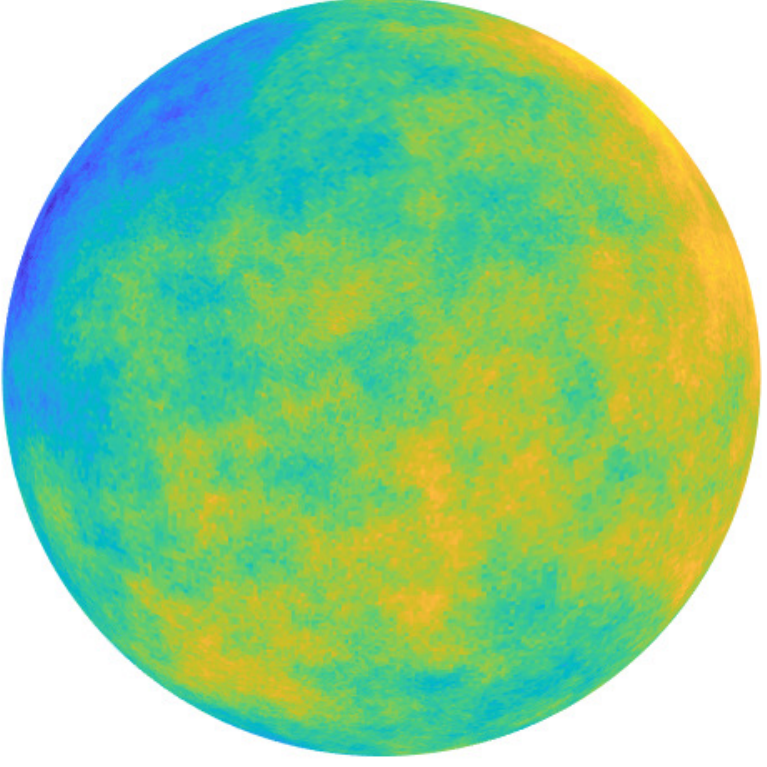}\quad
		\includegraphics[trim={180 105 160 95},clip,width=0.22\textwidth]{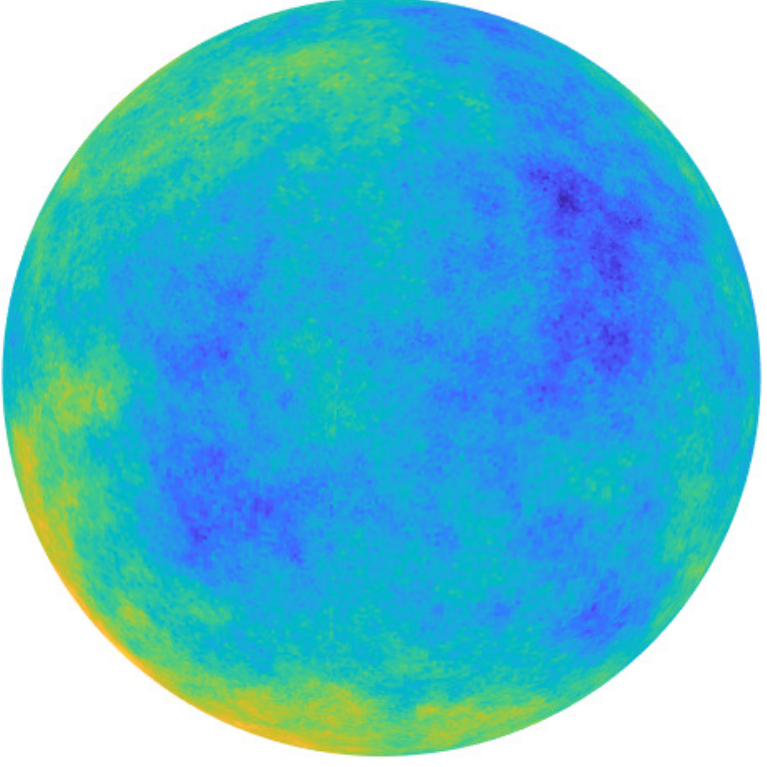}
		\caption{\label{fig:sphere}Four realizations of a Gaussian
			random field on $\bbS^2$ for the Mat\`ern covariance $k_{1/2}$ with
			respect to the geodesic distance.}
	\end{center}
\end{figure}

All the computations have been carried out on a compute server 
with dual 20-core Intel Xeon E5-2698 v4 CPU at 2.2 GHz and 768 
GB RAM. The computation of the compressed covariance operator 
has been done with the help of a C-program on a single core in 
line with \cite{HS2}, while nested dissection and the Cholesky factorization 
have been computed by using {\sc Matlab}\footnote{Release 2018b}.

Let us emphasize that, according to \cite{FemWvltManif_2003,TW}, 
wavelets with the same properties are available also in case 
of unstructured triangulations. 
Furthermore, 
the efficient assembly of the system matrix in 
case of such wavelets has been presented in \cite{AHK}. 
Therefore, 
the algorithm proposed here 
can be expected to perform efficiently also in practical situations, e.g., for high-dimensional
graphical models, 
where the present hypotheses may not hold or may be difficult to verify.

\section{Conclusions}\label{sec:Concl}

For a GRF $\GP$ indexed by  
a smooth manifold  $\cM$ 
which is obtained by ``coloring'' white noise 
$\cW$ with an elliptic, self-adjoint pseudodifferential 
operator $\cA$ as in \eqref{eq:WhNois}, 
we proved that in suitable wavelet coordinates in $L^2(\cM)$
precision and covariance operators $\cP$ and $\cC$ of $\GP$
both admit numerical approximations that are optimally sparse. 
This is to say, 
for any number $p\in \bbN$ of leading wavelet coordinates
of $\GP$, the $p\times p$ sections 
$\bP_p$ and $\bC_p$ of \emph{equivalent, bi-infinite matrix representations}
$\bP, \bC \in \bbR^{\bbN\times\bbN}$ of $\cC$ and $\cP$ 
admit
sparse approximations ${\bP}_p^{\varepsilon}$ and ${\bC}_p^{\varepsilon}$
with $\mathcal{O}(p)$ nonzero entries that are optimally consistent 
with $\cC$ and $\cP$.
The location of the $\mathcal{O}(p)$ ``essential'' entries of 
${\bP}_p^{\varepsilon}$ and ${\bC}_p^{\varepsilon}$ 
is universal (for the pseudodifferential colorings under consideration)
and can either be given a-priori, based on regularity of $\GP$, 
or numerically estimated a-posteriori 
from the decay of the wavelet coefficients, see Figure~\ref{fig:decay}. 
This a-posteriori compression is facilitated by a 
wavelet representation of the GRF $\GP$, since 
sample-wise smoothness of $\GP$ in Sobolev and Besov scales 
on $\cM$ is encoded in the decay (of the components) of 
the random coefficient sequence $\widetilde{\zvec}$ 
corresponding to the MRA. 
We furthermore have proven that diagonal preconditioning renders the 
condition numbers of both, ${\bP}_p^{\varepsilon}$ and ${\bC}_p^{\varepsilon}$,
bounded uniformly with respect to the number of parameters, $p$.

The theory and the algorithms do not rely on stationarity of the GRFs.
Furthermore, 
the assumption of self-adjointness on the coloring operator $\cA$ 
was made only for ease of presentation.
In the general case $\cA^*\neq \cA$ 
the covariance operator $\cC=(\cA^{*} \cA)^{-1}$ 
will still be self-adjoint and 
the GRF $\cZ = \cA^{-1}\cW$ colored by $\cA$ 
has the same distribution 
as a GRF colored by the self-adjoint operator 
$|\cA| = (\cA^* \cA)^{1/2}$. 

The \emph{wavelet-based} numerical covariance compression
and preconditioning
is in line with work to exploit a-priori structural hypotheses
on covariance matrix sparsity for
the efficient numerical approximation of (samples of) 
GRFs and of algorithms to estimate their
covariance operators and functions.
As one possible extension of the present analysis, 
the a-priori known locations of the $\mathcal{O}(p)$ many
nonzero entries of ${\bP}_p^{\varepsilon}$ and ${\bC}_p^{\varepsilon}$ 
may be leveraged in oracle versions of  
covariance estimation methodologies such as (group) LASSO.
These are well established for 
statistical inference in high-dimensional, graphical models
(e.g.\ \cite[Condition~A1]{JJSvdG}), where 
numerical constructions of MRAs have been proposed 
e.g.\ in \cite{CoifEtAlDifWav06}.

The hierarchic nature of wavelet MRAs
naturally facilitates \emph{novel, multilevel 
	versions of established covariance estimation algorithms}
as described in \cite{BickLevina08a,BickLevina08b,RBLZ08,JJSvdG} and the references there.
They amount to \emph{sampling the GRF $\GP$ represented in the MRA with
	a number of samples which depends on the spatial resolution level, with 
	large numbers of low-resolution samples, and 
	only few samples at the highest spatial resolution}.
We presented one such multilevel estimation algorithm and 
proved its asymptotically optimal, linear complexity
for all pseudodifferential colorings under consideration.
We introduced a novel, numerically sparse multilevel algorithm  
for kriging, i.e., for the spatial prediction given data. 
This is only a first application of the present results to methodologies 
in spatial statistics and more are conceivable. 
For instance, the computational benefits of 
wavelet representations may be exploited 
for computationally challenging tasks such as 
statistical inference of parameters. 

\appendix

\section{Pseudodifferential operators on manifolds}
\label{appendix:PDO_review}

We consider orientable manifolds 
satisfying Assumption~\ref{ass:cM-and-cA}\ref{ass:cM-and-cA_I}. 
In the case that $\cM$ is not orientable,
there exists a covering manifold 
$\widetilde{\cM}$ of $\cM$ with two sheets such that
$\widetilde{\cM}$ is orientable \cite[Thm.~I.58]{Aubin1998Riemannian}
and of dimension $n \geq 1$.

\subsection{Surface differential calculus}
\label{sec:DiffGeo}

A \emph{tangent vector} at $x\in \cM$ is a mapping
$X\from f\to X(f)\in \bbR$
which is defined on the set of functions $f$ that are differentiable
in a neighborhood of $x$ which satisfies
\begin{enumerate*}[label=\alph*)]
	\item  for $\lambda, \mu\in \bbR$, $X(\lambda f+ \mu g) = \lambda X(f) + \mu X(g)$,
	\item  $X(f)=0$ if $f$ is flat,
	\item  $X(fg) = f(x)X(g) + g(x)X(f)$.
\end{enumerate*} 
The \emph{tangent space} $T_x(\cM)$ to $\cM$ at $x\in \cM$ 
is the set of tangent vectors at $x$.
In any coordinate system $\{ x^i \}$ in $\cM$ at $x$, 
the vectors $\partial/\partial x^i$ defined by 
$(\partial/\partial x^i)_x(f) 
= 
[\partial(f\circ \varphi^{-1})/\partial x^i]_{\varphi(x)}$ belong to $T_x(\cM)$, 
and form a basis of the \emph{tangent vector space} at $x\in \cM$.
Here, $\varphi$ is any diffeomorphism on a neighborhood of $x$.
Its dual vector space is denoted by $T_x^*(\cM)$.
The \emph{tangent space $T(\cM)$ to $\cM$} is $\bigcup_{x\in \cM} T_x(\cM)$,
the \emph{dual tangent space $T^*(\cM)$} is $\bigcup_{x\in \cM} T^*_x(\cM)$.
The tangent space $T(\cM)$ carries a vector fiber bundle structure.
More generally, for $r,s\in \bbN$, the \emph{fiber bundle $T^r_s(\cM)$ of $(r,s)$ tensors} 
is $\bigcup_{x\in \cM} T_x(\cM)^{\otimes^r} \otimes T_x^*(\cM)^{\otimes^s}$.

The manifold $\cM$  
gives rise to the compact metric space 
$(\cM, \operatorname{dist}(\,\cdot\,, \,\cdot\,) )$, 
where the distance 
$\operatorname{dist}(\,\cdot\,, \,\cdot\,)$ 
can be chosen, for example, 
as the geodesic distance in $\cM$ of two points $x,x'\in \cM$, 
see \cite[Prop.~I.35]{Aubin1998Riemannian}.

\subsubsection{Coordinate charts and triangulations}
\label{sec:manif:Charts}

Provided that the manifold $\cM$ of dimension $n$ 
satisfies Assumption \ref{ass:cM-and-cA}\ref{ass:cM-and-cA_I}, 
it can be locally represented 
as parametric surface consisting of smooth coordinate patches. 
Specifically,
denote by $\Box = [0,1]^n$ the unit cube. 
Then we assume that $\cM$ 
is partitioned into a finite number $M$ of closed patches $\cM_i$ 
such that
\begin{equation*} 
\cM = \bigcup_{i=1}^M \cM_i,
\quad 
\cM_i = \gamma_i(\Box), \;\; i=1,\ldots,M \;.
\end{equation*}
Here, each $\gamma_i \from\Box \to \cM_i$ 
is assumed to be a smooth diffeomorphism.
We also assume that there exist smooth extensions 
$\widetilde{\cM}_i\supset \cM_i$
and $\widetilde{\gamma}_i \from \widetilde{\Box} \to \widetilde{\cM}_i$
such that $\widetilde{\gamma}_i\vert_\Box = \gamma_i$, 
where $\widetilde{\Box} = (-1,2)^n$.
Note that, 
in the notation of Section~\ref{section:GRFs-manifolds}, 
$G = \widetilde{\Box}$. 

The intersections $\cM_i \cap \cM_j$ for $i\ne j$ are either 
assumed to be empty or to be diffeomorphic to $[0,1]^k$ for some
$0\leq k <n$. 
We assume the charts $\gamma_i$ to be \emph{$C^0$-compatible} 
in the sense that 
for every $\hat{x} \in \cM_i \cap \cM_{i'}$ exists a bijective mapping
$\Theta\from\Box\to \Box$ such that 
$(\gamma_{i'}\circ \Theta)(x) = \hat{x}$
for $x=(x_1,\ldots,x_n)\in \Box$ 
with $\gamma_i(x) = \hat{x}$.
Note that $C^0$-compatibility admits 
$\cM = \partial G$ for certain polytopal 
domains $G$. 
In the case that $\cM = \partial G$ is smooth, 
we shall assume that the extensions 
satisfy $\widetilde{\cM}_i \subset \cM$ 
and that the charts $\gamma_i$ are smoothly compatible.

In the construction of MRAs on $\cM$, we shall require triangulations of $\cM$.
We shall introduce these in the Euclidean parameter domain $\Box$ and lift them to 
the coordinate patches $\cM_i$ on $\cM$ via the charts $\gamma_i$.

A mesh of refinement level $j$ on $\cM$ is obtained by dyadic subdivisions of 
depth~$j$ of $\Box$ into $2^{jn}$ subcubes $C_{j,k} \subseteq \Box$, where 
the multi-index $k = (k_1,\ldots,k_n) \in \bbN_0^n$ tags the location of $C_{j,k}$
with $0\leq k_m < 2^j$. 
With this construction in each co-ordinate patch, and taking into account
the inter-patch compatibility of the charts $\gamma_i$,
this results in a regular quadrilateral triangulation of $\cM$ consisting of
$2^{jn}M$ cells $\Gamma_{i,j,k} := \gamma_i(C_{j,k}) \subset \cM_i \subset \cM$.

\subsubsection{Sobolev spaces}
\label{sec:manifolds:SobSpc}

Let $\cM$ denote a compact manifold as in 
Section~\ref{sec:manif:Charts}.
Sobolev spaces on $\cM$ are invariantly defined 
in the usual fashion, i.e.,
in local coordinates of a smooth atlas 
$\{\widetilde{\gamma}_i\}_{i=1}^{M}$
of coordinate charts on $\cM$.

As in Assumption~\ref{ass:cM-and-cA}\ref{ass:cM-and-cA_I},
we assume that 
$\cM$ has dimension $n\in \bbN$, $\partial \cM = \emptyset$,
and is equipped with a (surface) measure $\mu$. 
It is given in terms of the first fundamental form on $\cM$ which, on $\cM_i$, 
is given by 
\begin{equation}\label{eq:Ki}
K_i(x) = 
\bigl( \partial_j \widetilde\gamma_i (x) \cdot \partial_{j'} \widetilde\gamma_i(x) \bigr) _{j,j'=1}^n , 
\quad x\in \widetilde\gamma_i^{-1}(\cM_i).
\end{equation}
The matrix $K_i$ in \eqref{eq:Ki} is symmetric and 
positive definite uniformly in $x\in \widetilde\cM_i$.
The $L^2(\cM)$ inner product on $\cM$ can then be expressed in the local chart
coordinates via
\begin{equation*} 
\begin{aligned}
( v , w )_{L^2(\cM)}
&:= \int_{\cM} v(x)w(x) \, \rd\mu(x) 
\\
&=  
\sum_{i=1}^M 
\int_{\widetilde\Box} 
((\chi_i v) \circ\widetilde\gamma_i)(x) ((\chi_iw) \circ \widetilde\gamma_i)(x) 
\sqrt{\det(K_i(x))} 
\, \rd x,
\end{aligned}
\end{equation*}
where $\{\chi_i\}_{i=1}^M$ denotes a smooth partition of unity 
which is subordinate to the atlas $\{\widetilde\gamma_i\}_{i=1}^M$.
For $1\leq p \leq \infty$, $L^p(\cM)$ shall denote the usual
space of real-valued, strongly measurable maps 
$v\from \cM\to \bbR$ which
are $p$-integrable with respect to $\mu$. 

Sobolev spaces on $\cM$ are invariantly 
defined by lifting their Euclidean versions 
on $\widetilde\Box$ to $\widetilde\cM_i$ via $\widetilde\gamma_i$.  
For $s\geq0$, the respective norm on $H^s(\cM)$ may be defined by 
\begin{equation*}
\| v\|_{H^s(\cM)}
:=  
\sum_{i=1}^M
\|
(\chi_i v)\circ \widetilde{\gamma}_i
\|_{H^s(\widetilde\Box)}. 
\end{equation*}
This definition is equivalent to the definition 
of $H^s(\cM)$ and $\norm{\,\cdot\,}{H^s(\cM)}$ 
in~\eqref{eq:def:Sobolev_sp}. 
For further details the reader is referred 
to~\cite[pp.~30--31 of Appendix~B]{HLS_2018} 
and the references therein. 
(Note that the proof of this equivalence 
on the sphere $\cM=\bbS^2$ as elaborated in \cite{HLS_2018} 
exploits only compactness and smoothness of $\bbS^2$. 
Thus, it can be generalized to any manifold as considered in this work.)

For $s<0$, the spaces $H^s(\cM)$ are defined by duality, 
here and throughout identifying $L^2(\cM)$ with its dual space. 

\subsection{(Pseudo)differential operators}
\label{sec:PDO}

We review basic definitions and notation from the 
H\"ormander--Kohn--Nirenberg 
calculus of pseudodifferential operators,  
to the extent that they are needed 
in our analysis of covariance kernels and operators. 

\subsubsection{Basic definitions}
\label{sec:PDOBasic}

Let $G$ be an open, bounded subset of $\bbR^n$, 
$r$, $\rho$, $\delta \in \bbR$ 
with $0\leq \rho\leq\delta\leq1$. 
The H\"ormander symbol class $S^r_{\rho,\delta}(G)$
consists of all $b\in C^\infty(G\times \bbR^n)$ such that, 
for all $K\subset\subset G$
and for any $\alpha,\beta \in \bbN_0^n$, 
there is a constant $C_{K,\alpha,\beta}>0$ with
\begin{equation}\label{eq:PDOest}
\forall x\in K, 
\;\;
\forall \xi\in \bbR^n :
\quad
\bigl| \partial^\beta_x \partial^\alpha_\xi b(x,\xi) \bigr| 
\leq 
C_{K,\alpha,\beta} (1+|\xi|)^{r-\rho|\alpha|+\delta|\beta|}.
\end{equation}
When the set $G$ is clear from the context, we write $S^r_{\rho,\delta}$.
In what follows, we shall restrict ourselves to the particular case
$\rho = 1$, $\delta = 0$, 
and consider $S^r_{1,0}$.
In addition, we write $S^{-\infty}_{1,0} := \bigcap_{r\in\bbR} S^r_{1,0}$.
A symbol $b\in S^r_{1,0}$ gives rise to a 
\emph{pseudodifferential operator}~$B$ via the relation  \eqref{eq:OPS-symbol}.

When $b\in S^r_{1,0}$, the operator $B$ is said to belong to 
$OPS^r_{1,0}(G)$ and it is (in a suitable topology) 
a continuous operator 
$B\from C_0^\infty(G) \to C^\infty(G)$, 
cf.~\cite[Thm.~II.1.5]{METaylorPDOs1981}.
We write $OPS^{-\infty}(G) = \bigcap_{r\in\bbR} OPS^{r}_{1,0}(G)$. 
We say that the operator $B \in OPS^r_{1,0}(G)$ 
is \emph{elliptic of order $r\in\bbR$} 
if, for each compact $K\subset\subset G$,  
there exist constants $C_K>0$ and $R>0$
such that 
\[
\forall x\in K, 
\;\; 
\forall |\xi|\geq R:
\quad 
|b(x,\xi)| \geq C_K \bigl( 1+|\xi|^2 \bigr)^r . 
\]
%

\subsubsection{(Pseudo)differential operators on manifolds}
\label{sec:PDOM}

We suppose Assumption~\ref{ass:cM-and-cA}.
Having introduced the 
class $OPS^r_{1,0}(G)$ for an Euclidean domain $G$, 
the operator class $OPS^r_{1,0}(\cM)$ is defined by 
the usual
``lifting to $\cM$ in local coordinates'' as described, 
e.g., in \cite[Sec.~II.5]{METaylorPDOs1981}. 
The definition is based on the behavior of $OPS^r_{1,0}(G)$
under smooth diffeomorphic changes of coordinates which we consider first. 

Let $G, \cO \subset \bbR^n$ be open and let  
$\gamma \from G \to \cO$ be a diffeomorphism. 
Consider $B \in OPS^r_{1,0}(G)$, so that 
$B\from C_0^\infty(G) \to C^\infty(G)$. 
We define the transported operator 
$\widetilde{B}$ by
\[
\widetilde{B}\from C^\infty_0(\cO) \to C^\infty(\cO) , 
\quad 
u \mapsto  B(u\circ\gamma)\circ \gamma^{-1} .
\]
For $r\in \bbR$, we then consider $OPS^r_{1,0}(\cM)$, 
the H\"ormander class of pseudodifferential operators on $\cM$
(investigated earlier by Kohn and Nirenberg \cite{KohnNirenbg65}).
We alert the reader to the use of the notation $OPS^r(\cM)$ for the 
so-called \emph{classical pseudo- differential operators} 
which afford (pseudohomogeneous) symbol expansions and
comprise a strict subset of $OPS^r_{1,0}(\cM)$, 
see, e.g., \cite{Seeley67}. 

Pseudodifferential operators in $OPS^{r}_{1,0}(\cM)$ 
on manifolds $\cM$ are defined in local coordinates.
A linear operator $\cB \from C^\infty(\cM)\to C^\infty(\cM)$ is a
pseudodifferential operator of order $r\in \bbR$ on $\cM$, 
$\cB \in OPS^r_{1,0}(\cM)$, 
if for any finite, smooth partition of unity 
$\bigl\{ \chi_i \in C_0^\infty(\widetilde{\cM}_i): i=1,\ldots,\bar{m} \bigr\}$
with respect to any atlas 
$\bigl\{ \bigl(\widetilde{\cM}_i, \widetilde\gamma_i \bigr) \bigr\}_{i=1}^{\bar{m}}$
of $\cM$ all
\emph{transported operators} satisfy
\begin{equation}\label{eq:appendix_OPS_transported}
f\mapsto B_{i,i'}f 
=
f \mapsto 
\bigl[ \bigl( 
\cB[ \chi_i ( f\circ\widetilde{\gamma}_i^{-1} )] 
\bigr) 
\chi_{i'} \bigr]\circ\widetilde{\gamma}_{i'}
\in 
OPS^{r}_{1,0}(\widetilde\gamma_i^{-1}(\widetilde{\cM}_i))
\quad 
\forall i,i'=1,\ldots,\bar{m}
.
\end{equation}

The class of all such operators is denoted $OPS^r_{1,0}(\cM)$. 
Importantly, 
$OPS^r_{1,0}(\cM)$ defined in this way
does not depend on the choice of 
the atlas of $\cM$ and is invariantly defined 
\cite[Sec.~II.5]{METaylorPDOs1981}, \cite[Def.~18.1.20]{HorIII}. 

\subsubsection{Principal symbols}
\label{sec:PSym}

For a bounded, open set $\cO\subset \bbR^n$,
the principal symbol $b_0(x,\xi)$ of $B\in OPS^r_{1,0}(\cO)$ is 
the equivalence class in $S^r_{1,0}(\cO)/S^{r-1}_{1,0}(\cO)$ 
(see, e.g., \cite[p.\ 49]{METaylorPDOs1981}).
Any member of the equivalence class will be called a principal symbol 
of $B$. 
For $\cB\in OPS^r_{1,0}(\cM)$, its principal symbol $b_0(x,\xi)$ is
invariantly (with respect to the choice of atlas on $\cM$) defined on $T^*(\cM)$
(see \cite[Eq.~(5.6)]{METaylorPDOs1981}).

\subsubsection{Pseudodifferential calculus}
\label{sec:PDOCalc}

The symbol class $S^r_{1,0}$ admits a 
\emph{symbolic calculus} (e.g., \cite[Prop.~II.1.3]{METaylorPDOs1981}).
In the sequel, we assume all pseudodifferential operators
to be \emph{properly supported}, 
see~\cite[Def.~II.3.6]{METaylorPDOs1981}. 
This is not restrictive, as every 
$\cB\in OPS^r_{1,0}(\cM)$ can be written as $\cB = \cB_1 + \cR$, 
where $\cB_1\in OPS^r_{1,0}(\cM)$ is properly supported
and where $\cR\in OPS^{-\infty}(\cM)$ \cite[Prop.~18.1.22]{HorIII}. 
\begin{proposition}
	\label{prop:OPS_algebra_prop}
	Let $r,t\in \bbR$ and 
	$\cA\in OPS^r_{1,0}(\cM)$, $\cB\in OPS^t_{1,0}(\cM)$
	be properly supported. Then, it holds
	\begin{enumerate}
		\item\label{item:sum} $\cA + \cB \in OPS^{\max\{ r , t \}}_{1,0}(\cM)$,
		\item\label{item:composition} $\cA\cB \in OPS^{ r+t }_{1,0}(\cM)$,
		\item\label{item:range} $\forall s\in \bbR: \; \cA\from H^{s}(\cM) \to H^{s-r}(\cM)$ is continuous.
	\end{enumerate}
\end{proposition}
\begin{proof} 
	The symbol class $S^{r}_{1,0}$ is constructed
	such that by~\eqref{eq:PDOest}, 
	$S^{r}_{1,0} \subseteq S^{\max\{r, t\}}_{1,0}$
	and also $S^{t}_{1,0} \subseteq S^{\max\{r,t\}}_{1,0}$. 
	Assertion~\ref{item:sum} follows from the construction 
	of the class $OPS^{r}_{1,0}(\cM)$ via an atlas, 
	see Section~\ref{sec:PDOM}.
	
	Recall the atlas $\{ \widetilde{\gamma}_i \}_{i=1}^{M}$ 
	of $\cM$ with subordinate smooth partition of unity $\{ \chi_i \}_{i=1}^{M}$.
	The transported operators $A_{i,i'}$ and $B_{i,i'}$ 
	defined according to~\eqref{eq:appendix_OPS_transported}
	belong to $OPS^r_{1,0}(\widetilde\Box)$ and to $OPS^t_{1,0}(\widetilde\Box)$.
	By~\cite[Thm.~II.4.4]{METaylorPDOs1981}, 
	$A_{j,j'} B_{j,j'} \in OPS^{r +t}_{1,0}(\widetilde\Box)$.
	Thus, claim~\ref{item:composition} 
	holds by the construction 
	of the class $OPS^{r+t}_{1,0}(\cM)$
	in Subsection~\ref{sec:PDOM}.
	
	Finally, the third assertion~\ref{item:range} 
	follows from~\cite[Thm.~II.6.5]{METaylorPDOs1981}, 
	which is elucidated on~\cite[p.~53 of Sec.~II.7]{METaylorPDOs1981}.
\end{proof}

\begin{proposition}
	\label{prop:OPS_real_powers}
	Let $r\in \bbR$ and $\cA\in OPS^r_{1,0}(\cM)$ 
	be self-adjoint, positive definite, and elliptic, 
	i.e., 
	there exists a constant $a_- > 0$ such that
	\[ 
	\forall w \in H^{r/2}(\cM): \quad 
	\langle \cA w, w\rangle \geq a_- \|w\|_{H^{r/2}(\cM)}^2 .
	\] 
	Then, for every $\beta\in \bbR$, $\cA^\beta \in OPS^{\beta r}_{1,0}(\cM)$. 
\end{proposition}
\begin{proof}
	Since $a_- >0$, $\cA$ is invertible. The assertion follows from~\cite[Thm.~3]{Seeley67}.
\end{proof}

\section{Multiresolution bases on manifolds}
\label{appendix:wavelets} 

In this section we briefly explain how 
the single-scale 
basis $\bPhi_j$, the dual single-scale basis $\widetilde{\bPhi}_j$ 
as well as the biorthogonal complement bases 
$\bPsi_j$ and $\widetilde{\bPsi}_j$ in 
\eqref{eq:bPhij}, \eqref{eq:bPhij-tilde} and \eqref{eq:bPsij}
can be constructed on a manifold 
$\cM$ which satisfies Assumption~\ref{ass:cM-and-cA}\ref{ass:cM-and-cA_I}. 
Furthermore, we collect some of their basic properties. 

We recall from \eqref{eq:bPhij}
that, for $j>j_0$, 
the subspaces $V_j \subset V_{j+1} \subset \ldots \subset L^2(\cM)$ 
are spanned by
\emph{single-scale} bases $\bPhi_j := \{ \phi_{j,k}: k\in \Delta_j \}$,
where $\Delta_j$ denote suitable index sets describing spatial localization
of the $\phi_{j,k}$. 
Furthermore, the subspaces are of cardinality 
$\dim(V_j) = \cO( 2^{nj} )$.
We assume elements 
$\phi_{j,k}\in V_j$ to be normalized in $L^2(\cM)$,
and their supports to scale according 
to $\operatorname{diam} (\operatorname{supp} \phi_{j,k}) \simeq 2^{-j}$.
We associate with these bases so-called \emph{dual single-scale bases} 
$\widetilde{\bPhi}_j := \{ \widetilde{\phi}_{j,k}: k\in \Delta_j \}$, 
for which one has 
$\langle \phi_{j,k}, \widetilde{\phi}_{j,k'}\rangle = \delta_{k,k'}$
for $k,k'\in \Delta_j$.
Such dual systems of one-scale bases on $\cM$ can be lifted 
in charts $\cM_j$ via parametrizations $\gamma_j$ 
from tensor products of univariate systems in 
the parameter domains $\Box\subset \bbR^n$.
For example, for primal bases $\bPhi_j$ obtained from tensorized, 
univariate B-splines of order $d$ in $\Box$ with 
dual bases of order $\widetilde{d}$ such that 
$d + \widetilde{d}$ is even,
the $\bPhi_j$ and $\widetilde{\bPhi}_j$ have approximation orders
$d$ and $\widetilde{d}$, respectively, see \eqref{eq:VjOrdReg}.  
The respective regularity indices $\gamma$ and $\widetilde{\gamma}$, 
see \eqref{eq:VjOrdReg}, satisfy 
$\gamma = d - 1/2$, whereas 
$\widetilde{\gamma} \sim \widetilde{d}$.
We refer to \cite{FemWvltManif_2003,DSWavManif99,FemWvltQuant_2009} 
for detailed constructions.

The biorthogonality of the systems 
$\bPhi_j$, $\widetilde{\bPhi}_j$ allows 
to introduce canonical projectors 
$Q_j$ and $Q^*_j$ for $j\in \bbN$ 
with $j > j_0$: 
\begin{equation*} 
Q_j v 
:= 
\sum_{k\in \Delta_j} \langle v,\widetilde{\phi}_{j,k}\rangle {\phi}_{j,k},
\qquad 
Q^*_j v 
:=  \sum_{k\in \Delta_j} \langle v, \phi_{j,k}\rangle \widetilde{\phi}_{j,k},
\end{equation*}
associated with corresponding \emph{multiresolution sequences} 
$\{ V_j \}_{j>j_0}$ and $\{ \widetilde{V}_j \}_{j>j_0}$.

The $L^2(\cM)$-boundedness of $Q_j$ implies the 
\emph{Jackson} and \emph{Bernstein} inequalities,
\begin{equation*} 
\| v - Q_j v \|_{H^s(\cM)} \lesssim 2^{-j(t-s)} \| v \|_{H^t(\cM)} 
\quad 
\forall v\in H^t(\cM),
\end{equation*}
for all $-\widetilde{d} \leq s\leq t \leq d$, 
$s<\gamma$, $-\widetilde{\gamma} < t$, and
\begin{equation*} 
\| Q_j v \|_{H^s(\cM)} \lesssim 2^{j(s-t)} \| Q_j v \|_{H^t(\cM)} 
\quad 
\forall v\in H^t(\cM),  
\end{equation*}
for all $t\leq s \leq \gamma$, 
with constants implied in $\lesssim$ 
which are uniform with respect to $j$.

To define MRAs, we start by introducing 
index sets $\nabla_j := \Delta_{j+1}\backslash \Delta_j$, $j>j_0$.
Given single-scale bases $\bPhi_j$, $\widetilde{\bPhi}_j$, 
the \emph{biorthogonal complement bases} 
$\bPsi_j$ and 
$\widetilde{\bPsi}_j$ 
in \eqref{eq:bPsij}
satisfying the \emph{biorthogonality relation}
\eqref{eq:Biorth}
can be constructed such that 
\eqref{eq:SupPsi} holds. 
We refer to 
\cite{FemWvltManif_2003,FemWvltQuant_2009,RekStevP2WavD2018} 
for particular constructions.

With the convention 
$Q_{j_0} = Q^*_{j_0} = 0$, one has for $v_J \in V_J$ 
and for $\widetilde{v}_J \in \widetilde{V}_J$ that 
\begin{align*} 
v_J 
&\textcolor{white}{:}= 
\sum_{j=j_0}^{J-1} (Q_{j+1}-Q_j)v_J, 
& 
\widetilde{v}_J 
&\textcolor{white}{:}= 
\sum_{j=j_0}^{J-1} (Q^*_{j+1} - Q^*_{j})\widetilde{v}_J,
\\
(Q_{j+1}-Q_j)v 
&:= 
\sum_{k\in \nabla_j} 
\langle v, \widetilde{\psi}_{j,k} \rangle \psi_{j,k},
& 
(Q^*_{j+1} - Q^*_{j}) v
&:= 
\sum_{k\in \nabla_j} 
\langle v, \psi_{j,k} \rangle \widetilde{\psi}_{j,k} .
\end{align*}
From this observation, a 
second wavelet basis $\widetilde{\bPsi}$ such that
$\bPsi$ and $\widetilde{\bPsi}$ are mutually biorthogonal in
$L^2(\cM)$ is now
obtained from the union of the coarse 
single-scale basis and complement bases, i.e., 
\[
\bPsi = \bigcup_{j\geq j_0} \bPsi_{j} 
\quad 
\text{and} 
\quad 
\widetilde{\bPsi} = \bigcup_{j\geq j_0} \widetilde{\bPsi}_j
\]
where we use the convention 
$\bPsi_{j_0} := \bPhi_{j_0+1}$, 
$\widetilde{\bPsi}_{j_0} := \widetilde{\bPhi}_{j_0+1}$ 
and assume that all basis functions 
are normalized in $L^2(\cM)$. 
The bases $\bPsi$ and $\widetilde{\bPsi}$ 
are called the primal and dual MRAs, respectively. 

The key to the preconditioning 
results 
for the covariance and precision matrices 
in Subsection~\ref{subsec:results:precon}
is the effect of 
diagonal preconditioning 
for pseudodifferential operators in MRAs. 
To address this, we let 
$\cB \in OPS^r_{1,0}(\cM)$ 
be a pseudodifferential operator which 
satisfies Assumption~\ref{ass:cM-and-cA}\ref{ass:cM-and-cA_II}, 
so that 
$\cB \from H^{r/2}(\cM) \to H^{-r/2}(\cM)$
is an isomorphism.
Assume that $\gamma > 0$. 
By \eqref{eq:Riesz},
$\bPsi$ is a Riesz basis for $L^2(\cM)$, so that 
the corresponding finite section matrices
\[
\bB_{J} 
= 
\bigl( \langle \cB \psi_{j',k'} , \psi_{j,k}\rangle 
\bigr)_{j_0\leq j,j'\leq J, \, k\in \nabla_j, \, k'\in \nabla_{j'}}
\]
are ill-conditioned,  
$\operatorname{cond}_{2}(\bB_J) \simeq 2^{|r|J}$.
Stability of the Galerkin projection in $H^{r/2}(\cM)$ 
\emph{and} the Riesz-basis property \eqref{eq:Riesz} 
in $H^{r/2}(\cM)$
imply the following result on diagonal preconditioning 
of $\bB_J$.

\begin{proposition}\label{prop:DiagPC}
	For $r\in\bbR$, 
	define the diagonal matrix 
	$\bD^r_J \in \bbR^{p(J) \times p(J)}$ 
	by 
	\[
	\bD^r_J = 
	\operatorname{diag} 
	\bigl( 
	2^{r|\lambda|} : 
	\lambda \in \Lambda_J \bigr) , 
	\]
	where $|\lambda|=j$ for $\lambda=(j,k)$ and, 
	as in \eqref{eq:p=p(J)}, 
	\[
	\Lambda_J 
	:= 
	\{ (j,k) : 
	j_0\leq j,j' \leq J, 
	k \in \nabla_j\}, 
	\qquad 
	p(J) := \#(\Lambda_J). 
	\]
	Suppose that the manifold and the operator 
	$\cB \in OPS^r_{1,0}(\cM)$ satisfy 	
	Assumptions~\ref{ass:cM-and-cA}\ref{ass:cM-and-cA_I} 
	and~\ref{ass:cM-and-cA_II}, respectively.  
	Furthermore, assume that~\eqref{eq:Riesz} holds  
	with 
	\begin{equation}\label{eq:rgamtgam}
	r/2 \in  
	(-\widetilde{\gamma}, \gamma)
	,
	\end{equation}
	Then, 
	for every $J\in \bbN$, the diagonal matrices $\bD^r_J$ define 
	uniformly spectrally equivalent preconditioners for $\bB_J$,
	i.e., 
	\begin{equation}\label{eq:DiagPC}
	\operatorname{cond}_{2} \bigl( \bD_J^{-r/2} \bB_J \bD_J^{-r/2} \bigr) 
	\simeq 1,
	\end{equation}
	with constants implied in $\simeq$ independent of $J$.
\end{proposition}
\begin{proof}
	Under
	Assumptions~\ref{ass:cM-and-cA}\ref{ass:cM-and-cA_I}--\ref{ass:cM-and-cA_II}
	the operator 
	$\cB \in OPS^{r}_{1,0}(\cM)$ 
	defines an isomorphism  
	between $H^{r/2}(\cM)$ and $H^{-r/2}(\cM))$, 
	see Proposition~\ref{prop:OPS_algebra_prop}\ref{item:range} 
	and Proposition~\ref{prop:OPS_real_powers}, 
	and the norm equivalence
	\[	
	\| v \|^2_{H^{r/2}(\cM)} 
	\simeq 
	\langle \cB v, v \rangle \quad
	\forall v\in H^{r/2}(\cM) 
	\]
	holds. 
	Here, $\langle\,\cdot\,, \,\cdot\,\rangle$ 
	denotes the $(H^{-r/2}(\cM),H^{r/2}(\cM))$
	duality pairing. 
	The assertion then follows from the 
	Riesz basis property \eqref{eq:Riesz}.
\end{proof}

\section{Coloring of Whittle--Mat\'{e}rn type}
\label{appendix:ColShft}

Three essential characteristics of the 
covariance structure of a random field  
are given by its smoothness, 
the correlation length,  
and the marginal variance. 
A convenient approach to 
define models, 
for which   
these important properties 
can be parametrized, 
i.e., controlled in terms of certain numerical parameters, 
is to generalize 
the Mat\'{e}rn covariance family. 
Such a parametrization in turn   
facilitates for instance likelihood-based 
inference in spatial statistics. 

Specifically, let us consider    
the white noise equation 
\eqref{eq:WhNois} for an  
elliptic, self-adjoint coloring pseudodifferential operator $\cA$ 
which is  
a fractional power $\beta>0$
of an elliptic ``base (pseudo)differential coloring operator'' 
$\cL \in OPS^{\bar{r}}_{1,0}(\cM)$ of order $\bar{r}>0$, 
shifted by  
the multiplication operator with respect to a nonnegative 
length-scale function $\kappa\from \cM \to \bbR$,
i.e.,  
\begin{equation}\label{eq:WhitMat}
\cA = \bigl( \cL+\kappa^2 \bigr)^\beta 
\quad \text{for some}\; \beta > 0.
\end{equation}
Here, $\beta>0$ and $\cL$ are such that 
the resulting coloring operator $\cA$ fulfills 
Assumption~\ref{ass:cM-and-cA}\ref{ass:cM-and-cA_II}.
In particular, $\kappa\in C^\infty(\cM)$. 

For a linear, second-order (so that $\bar{r}=2$) elliptic
(surface) differential operator 
$\cL$ on $\cM$ 
in divergence form, 
models of this type 
have been developed, e.g., 
in \cite{Bolin2011,Lindgren2011}. 
Moreover, computationally efficient methods 
to sample from such random fields or 
to employ the models in statistical applications, 
involving for instance inference or spatial predictions, have been 
discussed recently, e.g.,  
in \cite{Bolin2020,bkk-strong,CK2020,HKS1}. 
The following proposition extends and unifies these approaches,
admitting rather general operators $\cL$ 
(which, in the classic Mat\'{e}rn case, 
see~\cite{matern1960,Whittle1963},  
is the  
Laplace--Beltrami operator
$\cL = -\Delta_\cM \in OPS^{2}_{1,0}(\cM)$, 
with $\bar{r}=2$ and 
constant correlation length parameter $\kappa > 0$).

\begin{proposition}\label{prop:PrecCovOp-kappa-shift}
	Suppose that the manifold 
	$\cM$ satisfies Assumption~\ref{ass:cM-and-cA}\ref{ass:cM-and-cA_I} 
	and that  
	$\cL\in OPS^{\bar{r}}_{1,0}(\cM)$ for some $\bar{r}>0$ 
	is self-adjoint and positive.  
	Let $\beta > 0$ be such that $\bar{r}\beta > n/2$ 
	and let $\GP_\beta$ denote the 
	GRF solving the white noise equation \eqref{eq:WhNois} with 
	coloring operator $\cA = (\cL+\kappa^2 )^\beta$ on $\cM$, 
	where $\kappa\from\cM\to\bbR$ is smooth. 
	Then, 
	the covariance operator $\cC_\beta$ of the GRF $\GP_\beta$ 
	is a self-adjoint operator,
	(strictly) positive definite, compact operator on $L^2(\cM)$, 
	with finite trace. 
	Furthermore, 
	the covariance operator of $\GP_\beta$ is given by 
	\[
	\cC_\beta = \bigl(\cL+\kappa^2 \bigr)^{-2\beta} \in OPS^{-2\bar{r}\beta}_{1,0}(\cM) .
	\] 
	It defines an isomorphism 
	between $H^{s}(\cM)$ and $H^{s+2\bar{r}\beta}(\cM)$ 
	for all $s\in\bbR$. 
	
	The associated precision operator $\cP_\beta$ satisfies, 
	for all $s\in\bbR$, 
	\[ 
	\cP_\beta = \bigl( \cL+\kappa^2 \bigr)^{2\beta} 
	\in OPS^{2\bar{r}\beta}_{1,0}(\cM)  
	\] 
	and, for any $s\in\bbR$, 
	it defines an isomorphism 
	between $H^s(\cM)$ and $H^{s-2\bar{r}\beta}(\cM)$. 
	
	A GRF $\GP_\beta$ defined as in \eqref{eq:WhNois}
	with  coloring operator 
	$\cA = (\cL+\kappa^2 )^{\beta}$
	admits the regularity
	\[
	\GP_{\beta} 
	\in H^{s}(\cM), 
	\quad 
	\text{$\bbP$-a.s.},
	\quad\text{for}\quad s < \bar{r}\beta - n/2 .
	\]
\end{proposition}

\begin{proof}
	We first note that 
	by Proposition~\ref{prop:OPS_algebra_prop}\ref{item:sum} 
	$\cL+\kappa^2 \in OPS^{\bar{r}}_{1,0}(\cM)$, 
	since the multiplication operator 
	with the function $\kappa^2\in C^\infty(\cM)$ 
	is an element of $OPS^0_{1,0}(\cM)$ 
	By Proposition~\ref{prop:OPS_real_powers}
	$\cA = (\cL+\kappa^2)^\beta \in OPS^{\bar{r}\beta}_{1,0}(\cM)$. 
	Therefore all results 
	follow by the same arguments 
	as used in the proof 
	of Proposition~\ref{prop:C-and-P-ops}.
\end{proof}

Proposition~\ref{prop:PrecCovOp-kappa-shift} 
shows that the covariance and precision 
operator of the GRF $\GP_\beta$ 
defined by the white noise 
driven SPDE \eqref{eq:WhNois} 
with coloring operator of Whittle--Mat\'ern type, 
$\cA=(\cL+\kappa^2)^\beta$, 
satisfy 
$\cC_\beta\in OPS^{-2\bar{r}\beta}_{1,0}(\cM)$ 
and 
$\cP_\beta\in OPS^{2\bar{r}\beta}_{1,0}(\cM)$. 
For this reason, 
all results of Subsections~\ref{subsec:results:precon} 
and~\ref{subsec:results:sparse}
on optimal preconditioning 
and matrix compression 
are applicable  
for covariance operators  
$\cB=\cC_\beta$ 
and precision operators 
$\cB = \cP_\beta$ 
of Whittle--Mat\'ern type, 
where the order $r\in\bbR$ 
is given by $-2\bar{r}\beta$ 
and $2\bar{r}\beta$, respectively. 

\begin{remark}\label{rem:kappa-plateau} 
	The coefficient $\beta>0$ in the Whittle--Mat\'{e}rn like coloring 
	operator $\cA = (\cL + \kappa^2 )^\beta$ 
	and the order $\bar{r}>0$ 
	of the base operator 
	$\cL \in OPS^{\bar{r}}_{1,0}(\cM)$ 
	govern the spatial regularity of 
	the GRF $\GP_\beta$ 
	(in $L^p(\Omega)$-sense 
	and $\bbP$-a.s.). 
	The shift $\kappa^2$ does not influence the 
	the smoothness, but controls the spatial correlation 
	length of $\GP_\beta$. 
	Allowing for a function-valued shift $\kappa^2\in C^\infty(\cM)$ 
	thus corresponds to models 
	with a spatially varying correlation length 
	which form an important extension of 
	the classical Mat\'ern model. 
	
	As noted in Proposition~\ref{prop:PrecCovOp-kappa-shift}  above, 
	the corresponding Whittle--Mat\'{e}rn like covariance operator $\cC_\beta$ 
	is a self-adjoint, positive definite, compact operator on the 
	Hilbert space $L^2(\cM)$. 
	By the spectral theorem and by the (assumed) nondegeneracy of $\cC$, 
	there exists a countable system 
	$\{ e_j \}_{j\in\bbN}$ of eigenvectors for $\cC_\beta$  
	which forms an orthonormal basis for $L^2(\cM)$. 
	The corresponding positive eigenvalues 
	$\{\lambda_j(\cC_\beta) \}_{j\in\bbN}$
	accumulate only at zero and we may 
	assume that they are in non-increasing order. 
	This gives rise to a \emph{Karhunen--Lo\`{e}ve expansion} of the
	centered GRF $\GP_\beta$, 
	\begin{equation}\label{eq:KLGP}
	\GP_\beta (x,\omega) 
	= 
	\sum_{j \in \bbN} 
	\sqrt{\lambda_j(\cC_\beta)} \, e_j(x) \xi_j(\omega) ,
	\end{equation}
	with equality in $L^2(\Omega; L^2(\cM))$. 
	Here, 
	$\{ \xi_j \}_{j\in\bbN}$ are i.i.d.\ $\normal(0,1)$-distributed random variables. 

	Partial sums of the Karhunen--Lo\`eve 
	expansion \eqref{eq:KLGP} are of great importance for
	deterministic numerical approximations of PDE models in UQ which take 
	$\GP_\beta$ as a model for a distributed uncertain input data, 
	see, e.g., \cite{CST13,GKNSSS15,HKS1} and the references there. 
	The error in a $J$-term truncation of the expansion \eqref{eq:KLGP} is governed
	by the eigenvalue decay $\lambda_j(\cC_\beta)\to 0$ as $j\to \infty$.
	Assuming that \emph{$\kappa >0$ is constant on $\cM$}, 
	we find by using the spectral asymptotics 
	$ \lambda_j(\cL) = c' j^{\bar{r}/n} + o( j^{\bar{r}/n} )$
	for $\cL$  
	\cite[Thm.\ XII.2.1]{METaylorPDOs1981} 
	as well as  
	the spectral mapping theorem that 
	\[ 
	\forall j\in \bbN: 
	\quad 
	\lambda_j(\cC_\beta) 
	= 
	\bigl( \kappa^2 +\lambda_j(\cL) \bigr)^{-2\beta}
	=
	\kappa^{-4\beta}
	\bigl( 1+\kappa^{-2} c'j^{\bar{r}/n} + o\bigl( j^{\bar{r}/n} \bigr) \bigr)^{-2\beta} .
	\] 
	This shows that the asymptotic behavior 
	$\lambda_j(\cC_\beta) \simeq j^{-2\beta\bar{r}/n}$,  
	which is expected
	from \cite[Thm. XII.2.1]{METaylorPDOs1981}  
	applied for the operator 
	$\cC_\beta\in OPS^{-2\beta\bar{r}}_{1,0}(\cM)$, 
	is only visible for 
	$j > J^* = J^*(\kappa,\cL) = \cO\bigl( \kappa^{2n/\bar{r}} \bigr)$, 
	where the constant implied in $\cO( \,\cdot\, )$ is independent of
	the value of $\beta>0$. 
	For $1\leq j \leq J^*$, one expects 
	an eigenvalue ``plateau'' 
	\begin{equation}\label{eq:WM-EVplateau} 
	\lambda_j(\cC_\beta) \simeq \kappa^{-4\beta}, 
	\quad 
	1\leq j \leq J^* =  \cO( \kappa^{2n/\bar{r}} ). 
	\end{equation}
	Since in models of Whittle--Mat\'{e}rn type 
	with $\cL \in OPS^{\bar{r}}_{1,0}(\cM)$ 
	the (nondimensional) spatial correlation 
	length $\bar{\lambda}$ is $\kappa^{-2/\bar{r}}$,
	\eqref{eq:WM-EVplateau}
	indicates that for small values of $\bar{\lambda}$,
	the plateau in the spectrum of $\cC_\beta$ 
	scales as 
	$J^* = \cO( \kappa^{2n/\bar{r}} ) = \bar{\lambda}^{-n}$. 
	Due to 
	$\lambda_j(\cP_\beta) 
	= \lambda_j(\cC_\beta^{-1}) 
	= 1/\lambda_j(\cC_\beta)$,
	analogous statements hold for the precision operator 
	$\cP_\beta$.
\end{remark}

\section{Proof of 
	\texorpdfstring{Theorem~\ref{thm:mlmc_cov_estimation}}{Theorem~4.3}}
\label{appendix:sample_numbers}

The idea of this proof is similar to techniques 
in~\cite{HS_2019} and the references therein.

\begin{proof}[Proof of Theorem~\ref{thm:mlmc_cov_estimation}]
	We recall the asymptotic estimates of the computational work and error
	of the MLMC covariance estimation from~\eqref{eq:comp_cost_mlmc}
	and~\eqref{eq:error_mlmc}
	\begin{equation}\tag{\ref{eq:comp_cost_mlmc}}
	{\rm work} 
	=
	\cO 
	\left(\sum_{j=j_0}^J \widetilde{M}_{j} 2^{jn}\right)
	\end{equation}
	and
	\begin{equation}\tag{\ref{eq:error_mlmc}}
	{\rm error} 
	= 
	\cO
	\left( 
	2^{-J\alpha_0}
	+
	\sum_{j=j_0}^J
	\widetilde{M}_j^{-1/2}
	2^{-j \alpha}
	\right).
	\end{equation}
	We seek to find sample numbers $\widetilde{M}_{j}$, $j=0,\ldots,J$
	that optimize the computational work to achieve a certain accuracy.
	We consider $\widetilde{M}_j$ as a continuous variable and 
	seek to find stationary points of the 
	Lagrange multiplier function
	\begin{equation*}
	\xi \mapsto 
	g(\xi):=
	2^{-J\alpha_0}
	+
	\sum_{j=j_0}^J
	\widetilde{M}_j^{-1/2}
	2^{-j \alpha}
	+ 
	\xi 
	\sum_{j=j_0}^J \widetilde{M}_{j} 2^{jn}
	.
	\end{equation*}
	Hence, we seek $\widetilde{M}_j$, $j=j_0,\ldots,J$
	such that $\partial g(\xi)/\partial \widetilde{M}_j =0$, 
	$j=j_0,\ldots,J$.
	This results in the conditions
	$\widetilde{M}_j = 2^{-j(n+\alpha)2/3}$, $j=j_0 + 1,\ldots,J$, 
	and we thus choose 
	\begin{equation*}
	\widetilde{M}_j
	=
	\lceil 
	\widetilde{M}_{j_0}
	2^{-j(n+\alpha)2/3}
	\rceil,
	\quad 
	j=j_0 + 1,\ldots,J
	,
	\end{equation*}
	where $\widetilde{M}_{j_0}$ 
	is still to be determined.
	This yields
	\begin{equation}\label{eq:app:work_est1}
	{\rm work} 
	=
	\cO 
	\left(
	\widetilde{M}_{j_0}
	\sum_{j=j_0}^J 
	E_j
	\right)
	\end{equation}
	and
	\begin{equation*}
	{\rm error} 
	= 
	\cO
	\left(2^{-J\alpha_0}
	+
	\widetilde{M}_{j_0}^{-1/2}
	\sum_{j=j_0}^J 
	E_j
	\right)
	,
	\end{equation*}
	where $E_j = 2^{-j\alpha 2/3 + jn/3}$, $j=j_0,\ldots,J$.
	It holds that 
	\begin{equation*}
	\sum_{j=j_0}^J 
	E_j
	=
	\begin{cases}
	\cO(1) & \text{if } 2\alpha > n, \\
	\cO(J) & \text{if } 2\alpha = n, \\
	\cO(2^{J(n/3 - \alpha2/3)}) &\text{if } 2\alpha <n.
	\end{cases}
	\end{equation*}
	We choose $\widetilde{M}_{j_0}$ to equilibrate the error contributions 
	in 
	$2^{-J\alpha_0}
	+
	\widetilde{M}_{j_0}^{-1/2}
	\sum_{j=0}^J 
	E_j$,
	which leads us to 
	\begin{equation*}
	\widetilde{M}_{j_0}
	=
	\begin{cases}
	2^{2J\alpha_0} & \text{if } 2\alpha >n ,\\
	2^{2J\alpha_0}J^2 & \text{if } 2\alpha = n ,\\
	2^{J(2\alpha_0 +2n/3 - 4\alpha/3)} & \text{if } 2\alpha <n .
	\end{cases}
	\end{equation*}
	By inserting the corresponding value of $\widetilde{M}_{j_0}$
	and of $\sum_{j=j_0}^J E_j$ 
	into \eqref{eq:app:work_est1},
	we obtain that
	\begin{equation*}
	{\rm work} =
	\begin{cases}
	\cO(2^{J2\alpha_0}) & \text{if } 2\alpha>n , \\
	\cO(2^{J2\alpha_0} J^3 ) & \text{if } 2\alpha = n , \\
	\cO(2^{J(n-2(\alpha_0 - \alpha))}) & \text{if } 2\alpha < n \;. 
	\end{cases}
	\end{equation*}
	The assertion now follows by expressing the computational work
	as a function of $\varepsilon$ with the choice $\varepsilon = 2^{2J\alpha_0}$.
\end{proof}

\bibliographystyle{abbrv}
\bibliography{GRF-Cov-bib}

\end{document}